\documentclass[11pt,letterpaper,reqno]{amsart}

\numberwithin{equation}{section}

\usepackage{floatrow}

\usepackage{amssymb}
\usepackage{amsmath}
\usepackage{mathrsfs}
\usepackage{dsfont}
\usepackage{verbatim}
\usepackage{stmaryrd, MnSymbol}
\usepackage{enumerate, xspace}
\usepackage{changebar}
\usepackage{hyperref}
\usepackage{comment}
\usepackage[tmargin=3cm,lmargin=3cm,rmargin=3cm,bmargin=3cm,headheight=1cm,footskip=1.5cm]{geometry}
\usepackage{mathtools}
\mathtoolsset{showonlyrefs,showmanualtags}
\usepackage{tikz}
\usepackage{tikz-cd}
\usetikzlibrary{arrows.meta,decorations.pathreplacing}
\usepackage{youngtab}
\usetikzlibrary{graphs,shapes}
\usepackage[edges]{forest}
\usepackage{algorithm2e}
\usepackage{cancel}
\usepackage{here}
\usepackage{amsthm}
\usepackage{caption}
\usepackage{ulem}
\usepackage[foot]{amsaddr}

\newtheorem{theorem}{Theorem}[section]
\newtheorem*{theorem*}{Theorem}
\newtheorem{lemma}[theorem]{Lemma}
\newtheorem{corollary}[theorem]{Corollary}
\newtheorem{proposition}[theorem]{Proposition}
\newtheorem*{proposition*}{Proposition}

\newtheorem{remark}[theorem]{Remark}

\numberwithin{equation}{section}

\renewcommand{\a}{\alpha}
\renewcommand{\b}{\beta}
\newcommand{\e}{\varepsilon}

\newcommand{\fa}{\varphi}

\newcommand{\p}{\mathbf{p}}
\newcommand{\q}{\mathbf{q}}

\def\i{{\mathrm{i}}}
\def\R{{\mathbb{R}}}
\def\N{{\mathbb{N}}}
\def\Z{{\mathbb{Z}}}

\def\E{{\mathbb{E}}}

\def\E{{\mathbb{E}}}

\title[Seat number configuration of the box-ball system]{Seat number configuration of the box-ball system, and its relation to the 10-elimination and invariant measures}

\author{Hayate Suda}

\begin{document}

\begin{abstract}
    The box-ball system (BBS) is a soliton cellular automaton introduced in \cite{TS}, and it is known that the dynamics of the BBS can be linearized by several methods. Recently, a new linearization method, called the seat number configuration, was introduced in \cite{MSSS}. 
{In this paper, we develop this method further by introducing
the $k$-skip map, which is a natural operation on the seat number
configuration. From the soliton point of view, this map lowers the height
of each soliton by $k$. We first show that the $k$-skip map shifts the seat
number configuration and that, for finite ball configurations on the
half-line, the 1-skip map coincides with the 10-elimination introduced in \cite{MIT}. We
then extend the seat number configuration and the $k$-skip map to the BBS
on the whole-line. Finally, we study the distribution of the $k$-skipped configuration under the 
invariant measures introduced in \cite{FG}. As an application, we compute expectations of
the carriers with seat numbers, which are related to the stationary
current and the effective velocity of solitons. }
\end{abstract}

\maketitle

        \section{Introduction}
            
            The box-ball system (BBS) is a soliton cellular automaton introduced in \cite{TS}. In this paper, we consider the BBS($\ell$), $\ell \in \N\cup\{\infty\}$, which is a class of generalizations of the BBS introduced in \cite{TM}. 
            The configuration space $\Omega$ is either $\left\{0,1 \right\}^{\N}$ or a certain subset of $\left\{0,1 \right\}^{\Z}$, and depending on the configuration space, we refer to the BBS as the BBS on the half-line or the BBS on the whole-line, respectively.  
            Here, for $\eta \in \Omega$, $\eta(x) = 0$ (resp. $\eta(x) = 1$) means that the site $x$ is vacant (resp. occupied). 
            The dynamics of the BBS($\ell$) on the half-line is given via the {\it carrier with capacity $\ell$}, $W_{\ell}\left(\eta, \cdot\right) \ : \Z_{\ge 0} \to \Z_{\ge 0}$, which is recursively constructed as follows : 
                \begin{itemize}

                    \item An empty carrier starts from $x = 0$, i.e., $W_{\ell}\left(\eta, 0\right) = 0$. 
                
                    \item If there is a ball at $x$ and the carrier is not full, then the carrier picks it up, i.e., if $\eta(x) = 1$ and $W_{\ell}\left(\eta, x - 1\right) < \ell$, then $W_{\ell}\left(\eta, x\right) = W_{\ell}\left(\eta, x - 1\right) + 1$.

                    \item If $x$ is vacant and the carrier is not empty, then the carrier puts down a ball, i.e., if $\eta(x) = 0$ and $W_{\ell}\left(\eta, x - 1\right) > 0$ then $W_{\ell}\left(\eta, x\right) = W_{\ell}\left(\eta, x - 1\right) - 1$. 

                    \item Otherwise, the carrier just goes through. 
                \end{itemize}
            Then, by using $W_{\ell}\left(\eta, \cdot\right)$, the one-step time evolution of the BBS is described by the operator $T_{\ell} : \Omega \to \Omega$ defined as
                \begin{align}\label{eq:dynamics}
                    T_{\ell}\eta(x) := \eta(x) + W_{\ell}\left(\eta, x - 1\right) - W_{\ell}\left(\eta, x\right).
                \end{align}
            Note that for the case $\Omega \subset \left\{0,1\right\}^{\Z}$, the domain of $W_{\ell}(\eta, \cdot)$ can be extended to $\Z$, and the one-step time evolution of the BBS($\ell$) on the whole-line is also described by $T_{\ell} : \Omega \to \Omega$, which is defined via the same equation \eqref{eq:dynamics}, see Section 4 for details. 
            
            Despite the simple description of the dynamics, it is known that the BBS exhibits solitonic behavior, i.e., the BBS is a classical integrable system.
            There are {infinitely many} types of solitons in the BBS, and each type is labeled by $k \in \N$. A $k$-soliton is composed of $k$ $1$s and $k$ $0$s, and solitons in the BBS are identified by the {\it Takahashi--Satsuma algorithm}, see Appendix \ref{app:TS} for details. { Figure \ref{fig:intro_soliton} gives a simple illustration of how solitons can be visualized
through the graph of the carrier $W_{\infty}$. The colored solitons in the figure are those
identified by the Takahashi--Satsuma algorithm.  In the carrier graph
representation, a $k$-soliton appears as a peak of height $k$.  Thus, in
Figure~\ref{fig:intro_soliton}, the blue curve represents a $3$-soliton, while the orange curve
represents a $1$-soliton.  The figure also shows how solitons of different sizes interact.  After the interaction, both solitons recover their original shapes, but their positions are shifted compared
with the positions expected from free propagation. This shift is called
the phase shift of the soliton interaction.}  In addition {to the solitonic behavior of the BBS}, it is known that the BBS is also integrable in the quantum sense, and the origin of the integrability of the BBS has been well-studied, see \cite{IKT} and references therein for details. We note that the concept of ``soliton" is not necessary to state the results of this paper, but if something can be interpreted intuitively in terms of solitons, it will be mentioned as remarks (e.g., Remarks \ref{rem:seat_sol}, \ref{rem:zeta_sol}, \ref{rem:skip_sol}.). 

\begin{figure}
    \centering
    \begin{tikzpicture}[scale=0.6, every node/.style={transform shape}]
\colorlet{BigSol}{blue!75!black}
\colorlet{SmallSol}{orange!85!black}
\colorlet{GridCol}{gray!35}


\newcommand{\balls}[1]{\foreach \i in {#1}{\fill ({\i+0.5},0.40) circle[radius=0.18];}}

\newcommand{\bbsrow}[8]{%
  \begin{scope}[shift={(0,#1)}]
    \draw[black,thick] (1,0) -- (22.0,0);
    \foreach \x in {1,...,22}{\draw[GridCol] (\x,-0.15) -- (\x,3.55);}    
    \foreach \y in {1,2,3}{\draw[gray!18] (1,\y) -- (22.0,\y);}    
    \foreach \y/\lab in {1/1,2/2,3/3}{\node[left,gray!80] at (0.65,\y) {\small $\lab$};}
    \node[left] at (-0.2,1.7) {#2};
    \balls{#3}
    \draw[black!25,line width=0.8pt] #4;
    \draw[BigSol,line width=1.25pt] #5;
    \draw[SmallSol,line width=1.25pt] #6;
    #7
    #8
  \end{scope}%
}


\bbsrow{16}{$\eta$}
 {1,2,3,10}
 {(1,0)--(2,1)--(3,2)--(4,3)--(5,2)--(6,1)--(7,0)--(8,0)--(9,0)--(10,0)--(11,1)--(12,0)--(22,0)}
 {(1,0)--(2,1)--(3,2)--(4,3)--(5,2)--(6,1)--(7,0)}
 {(10,0)--(11,1)--(12,0)}
 {\node[BigSol] at (2.5,2.6) {$3$-soliton};}
 {\node[SmallSol] at (11.0,1.35) {$1$-soliton};}

\bbsrow{12}{$T_{\infty}\eta$}
 {4,5,6,11}
 {(1,0)--(2,0)--(3,0)--(4,0)--(5,1)--(6,2)--(7,3)--(8,2)--(9,1)--(10,0)--(11,0)--(12,1)--(13,0)--(22,0)}
 {(4,0)--(5,1)--(6,2)--(7,3)--(8,2)--(9,1)--(10,0)}
 {(11,0)--(12,1)--(13,0)}

\bbsrow{8}{$T^{2}_{\infty}\eta$}
 {7,8,9,12}
 {(1,0)--(2,0)--(3,0)--(4,0)--(5,0)--(6,0)--(7,0)--(8,1)--(9,2)--(10,3)--(11,2)--(12,1)--(13,2)--(14,1)--(15,0)--(22,0)}
 {(7,0)--(8,1)--(9,2)--(10,3)--(11,2)--(12,1) (14,1)--(15,0)}
 {(12,1)--(13,2)--(14,1)}

\bbsrow{4}{$T^{3}_{\infty}\eta$}
 {10,11,13,14}
 {(1,0)--(2,0)--(3,0)--(4,0)--(5,0)--(6,0)--(7,0)--(8,0)--(9,0)--(10,0)--(11,1)--(12,2)--(13,1)--(14,2)--(15,3)--(16,2)--(17,1)--(18,0)--(22,0)}
 {(10,0)--(11,1)--(12,2) (14,2)--(15,3)--(16,2)--(17,1)--(18,0)}
 {(12,2)--(13,1)--(14,2)}

\bbsrow{0}{$T^{4}_{\infty}\eta$}
 {12,15,16,17}
 {(1,0)--(2,0)--(3,0)--(4,0)--(5,0)--(6,0)--(7,0)--(8,0)--(9,0)--(10,0)--(11,0)--(12,0)--(13,1)--(14,0)--(15,0)--(16,1)--(17,2)--(18,3)--(19,2)--(20,1)--(21,0)--(22,0)}
 {(15,0)--(16,1)--(17,2)--(18,3)--(19,2)--(20,1)--(21,0)}
 {(12,0)--(13,1)--(14,0)}

\foreach \x in {0,...,21}{\node[below,gray!75] at (\x+1,-0.72) {\scriptsize \x};}

\end{tikzpicture}

    \caption{Graphs of the carrier $W_{\infty}\left(\cdot, \cdot\right)$ for $\eta, T_{\infty}\eta, T_{\infty}^2\eta, T_{\infty}^3\eta$ and $T_{\infty}^4\eta$ where $\eta = 111000000100\dots \in \{0,1\}^{\N}$. The balls are drawn at box centers, while the carrier graph is drawn on box
boundaries; hence the two grids are shifted by $1/2$. The solitons shown in the figure are identified by the Takahashi--Satsuma
algorithm. }
    \label{fig:intro_soliton}
\end{figure}

            As in the cases of classical integrable systems, such as the Korteweg-de Vries equation, the initial value problem of the BBS can be solved via the explicit linearization methods \cite{T,KOSTY,MIT,FNRW}. 
            Recently, a new linearization of the BBS($\ell$) on the half-line, called the seat number configuration, was introduced in \cite{MSSS} and used to study the explicit relationships between the rigged configuration and the slot decomposition. 

            {The aim of this paper is to develop the seat number configuration further
and to apply it to the study of invariant measures of the BBS.  The main
new object introduced in this paper is the $k$-skip map $\Psi_k$.  This
map is a natural operation from the viewpoint of the seat number
configuration: it removes the sites corresponding to the first $k$ seats
of the carrier with seat numbers and reads the remaining sites from left
to right.  In terms of solitons, this operation has a simple
interpretation.  It removes the first $k$ heads and tails from each
soliton, and hence a soliton of size $k+\ell$ in the original
configuration is regarded as a soliton of size $\ell$ after applying
$\Psi_k$, while solitons of size at most $k$ disappear.

The rest of the paper is organized as follows.  First, we study the
basic properties of the $k$-skip map on the half-line and show that it
acts as a shift of the seat number configuration.  We then relate this map
to the classical 10-elimination, which was introduced in \cite{MIT} to
solve the initial value problem of the periodic BBS.  Next, in order to
apply the method to the randomized BBS on $\Z$, we extend the seat number
configuration and the $k$-skip map to the whole-line.  Finally, we apply
the $k$-skip map to the invariant measures introduced in \cite{FG}.  This
allows us to describe the distribution of the configuration obtained by
lowering the height of each soliton by $k$.  As an application, we compute
expectations of the carriers with seat numbers under these invariant
measures.  These quantities are related to the stationary current of
solitons and the {\it effective velocity} of solitons. Here the effective velocity $v^{\mathrm{eff}}_{k}$ of $k$-solitons is the average speed of $k$-solitons in the bi-infinite configuration. Since there are infinitely many solitons coming from $-\infty$ for the whole-line case, $v^{\mathrm{eff}}_{k}$ is not equal to $k$ in general, see \cite{FNRW,OSS} for the effective velocities under invariant measures. }

\subsection{Main results and organization}

{We now summarize the main results of this paper and explain how they are
related to each other.  The first group of results concerns the construction of the
$k$-skip map and its basic properties.  In Section \ref{sec:seat}, we
define $\Psi_k$ on the half-line.  In Proposition
\ref{prop:seat_semig}, we show that $\Psi_k$ shifts the seat
number configuration: the $(k+\ell,\sigma)$-seats of the original
configuration become the $(\ell,\sigma)$-seats of the $k$-skipped
configuration.  As consequences, we obtain the semigroup property of
$(\Psi_k)_{k \in \Z_{\geq 0}}$ in Proposition \ref{prop:semig} and the
corresponding shift relation for the counting function of solitons in
Proposition \ref{prop:shift}.  From the soliton point of view, these
results mean that $\Psi_k$ lowers the height of each soliton by $k$.

The second group of results concerns the relation with the 10-elimination.  In
Section \ref{sec:10}, we recall the half-line version of the
10-elimination and its rigging.  In Theorem \ref{thm:10=seat}, we show
that, for finite ball configurations on the half-line, the 1-skip map
coincides with the 10-elimination.  Moreover, the soliton-counting
functions defined from the seat number configuration coincide with the
corresponding 10-riggings.  

The third group of results concerns the extension of the seat number configuration
and the $k$-skip map to the whole-line.  This extension is needed in order
to apply the seat number configuration to the study of the BBS under stationary distributions on the whole-line. In Section
\ref{sec:line}, we define the carrier with seat numbers, the seat number
configuration and the $k$-skip map on the whole-line.  We also prove the
corresponding linearization formula in Theorem \ref{thm:linear_whole}. 
Compared with the half-line case, an offset term appears in the
whole-line setting, and this term describes the effect coming from the left of the origin. In Remarks \ref{rem:offset_ex}, \ref{rem:offset_vel} and Appendix \ref{app:eff_velo}, we explain how the expectation of the offset relates to the effective velocity of solitons. 

The fourth group of results concerns invariant measures.  In Section
\ref{sec:dis}, we consider the invariant measures introduced in
\cite{FG}. Theorem \ref{thm:skip_stat} describes the distribution of the $k$-skipped configurations under these measures.  In soliton language, this theorem
gives the distribution of the configuration obtained by lowering the
height of each soliton by $k$. The distribution of the $k$-skipped configurations again belongs to
the same class of invariant measures under a certain conditioning. We then consider two-sided Markov distributions, which are important examples of invariant measures. Theorem \ref{thm:skip_markov} shows that under such a two-sided Markov distribution conditioned on a certain event, the distribution of the $k$-skipped configuration is again a two-sided Markov distribution with a  different transition matrix. 

Finally, we use the above results to compute expectations of
the carriers with seat numbers. The key point is that the $k$-skip map naturally gives recursive
relations for the expectations of certain quantities under invariant measures. By
analyzing this recursion, we obtain explicit formulae for the carrier
expectations and related quantities, such as the stationary current of solitons and
the effective velocity of solitons, see Section \ref{subsec:expectation}. In particular, in a two-sided Markov distribution case with transition matrix $P = \left(p(i,j)\right)_{i,j = 0,1}$ satisfying $0 < p(0,1) < p(1,0)$, we give explicit formulae for the above quantities in terms of parameters $c = \frac{p(1,1)}{p(0,0)}, d = \frac{p(0,1)}{p(1,0)}$, see Propositions \ref{prop:carrier_ex} and \ref{prop:eff_velo}. These formulae are consistent with those previously derived in the physics literature \cite{KMP}, see  Remark \ref{rem:phys}. To the best of our knowledge,
the present work gives the first rigorous derivation of these formulae from two-sided Markov distributions.}
            
            The contents of Sections  \ref{sec:seat}, \ref{sec:10} and \ref{sec:line} are closely related to \cite{MSSS} and can be considered as a continuation of it.  On the other hand, the contents of Section \ref{sec:dis} are related to the field of the randomized BBS, which has been actively studied in recent years \cite{CKST, CS, CS2, FG, KL, KLO, KMP, KMP2, KMP3, LLP, LLPS}. 
            We expect that the statistical properties of the seat number configuration and the $k$-skip map introduced in this paper will be useful in the analysis of the randomized BBS. 
            In fact, in the forthcoming paper \cite{OSS}, the space-time scaling limit of a tagged soliton is considered, and the $k$-skip map plays an essential role in the proof. We look forward to seeing more studies utilizing the seat number configuration and the $k$-skip map in the future. 

        \section{Seat number configuration}\label{sec:seat}
        
            In this section, first we briefly recall the definition of the carrier with seat numbers and the corresponding seat number configuration introduced in \cite{MSSS}. Then, in the subsequent subsection, we introduce the notion of the $k$-skip map, and show that the $k$-skip map induces a shift operator of the seat number configuration. Throughout this section, the configuration space of the BBS is 
            $\Omega = \left\{0,1 \right\}^{\N}$. 

        {The seat number configuration has a natural interpretation in terms of
solitons.  It records, through the numbered seats of the carrier, how the
balls belonging to solitons are transported.  This representation is
particularly useful because it separates solitons according to their
sizes.  More precisely, the definitions below allow us to count solitons
of each size and to describe how this soliton content changes under the
maps considered in this paper.  This solitonic interpretation will be
made more explicit in Remarks~2.2, 2.4 and 2.11.}
            

            \subsection{Carrier with seat numbers}

                We consider a situation where the seats of the carrier are numbered by $\N$, and introduce functions $\mathcal{W}_{k}(\eta, \cdot) : \Z_{\ge 0} \to \{0,1\}$ {which indicate whether the No.~$k$ seat is occupied},  i.e., $\mathcal{W}_{k}(\eta, x) = 0$ (resp. $\mathcal{W}_{k}(\eta, x) = 1$) means that the No. $k$ seat is vacant (resp. occupied) at site $x$.
                Then, the refined construction of the carrier with seat numbers is given as follows : 
                    \begin{itemize}
                        \item An empty carrier {with infinite capacity} starts from $x = 0$, i.e., $\mathcal{W}_{k}(\eta, 0) = 0$ for any $k \in \N$.
                        
                        \item If there is a ball at site $x$, then the carrier picks the ball and puts it at the empty seat with the smallest seat number, i.e., if $\eta(x) = 1$ and $\min\left\{ k \in \N  \ ; \ \mathcal{W}_{k}(\eta, x - 1) = 0 \right\} = \ell$, then 
                            \begin{align}
                                \mathcal{W}_{k}(\eta, x) = \begin{dcases}
                                    1 \ & \  k = \ell, \\
                                    \mathcal{W}_{k}(\eta, x - 1) \ & \ k \neq \ell.
                                \end{dcases}
                            \end{align}
                        
                        \item If the site $x$ is empty, and if there is at least one occupied seat, then the carrier puts down the ball at the occupied seat with the smallest seat number, i.e., if $\eta(x) = 0$ and $\min\left\{ k \in \N  \ ; \ \mathcal{W}_{k}(\eta, x - 1) = 1 \right\} = \ell < \infty$, then 
                            \begin{align}
                                \mathcal{W}_{k}(\eta, x) = \begin{dcases}
                                    0 \ & \  k = \ell, \\
                                    \mathcal{W}_{k}(\eta, x - 1) \ & \ k \neq \ell.
                                \end{dcases}
                            \end{align}

                        \item Otherwise, the carrier just goes through, i.e., if $\eta(x) = 0$ and $ \mathcal{W}_{k}(\eta, x - 1) = 0$ for any $k \in \N$, then 
                            \begin{align}
                                \mathcal{W}_{k}(\eta, x) = \mathcal{W}_{k}(\eta, x - 1) = 0
                            \end{align}
                        for any $k \in \N$.
                    \end{itemize}
        {Figure \ref{fig:carrier-graph-with-seat-numbers} illustrates this construction
for the configuration $\eta=11011000\cdots$, showing both the carrier graph
$W_\infty(\eta,\cdot)$ and the corresponding state of the numbered seats.}
    \begin{figure}[t]
\centering
\begin{tikzpicture}[
    x=0.85cm,y=0.85cm,
    >=Stealth,
    font=\small
]

\colorlet{seatone}{red!70!black}
\colorlet{seattwo}{green!50!black}
\colorlet{seatthree}{blue!70!black}

\tikzset{
  eventlabel/.style={font=\scriptsize, inner sep=1pt, fill=white},
  carrierbox/.style={draw=gray!75, line width=0.55pt},
  smallball/.style={circle, fill=black, inner sep=1.7pt}
}

\foreach \x in {0,...,9}{
  \draw[gray!18] (\x+0.5,-0.85)--(\x+0.5,3.35);
}
\foreach \y in {0,...,3}{
  \draw[gray!18] (0.35,\y)--(9.65,\y);
}

\draw[->,gray!70] (0.35,0)--(9.9,0) node[right] {$x$};
\draw[->,gray!70] (0.5,-0.05)--(0.5,3.45)
  node[above] {$W_\infty(\eta,\cdot)$};

\node[left,gray!80] at (0.35,0) {$0$};
\node[left,gray!80] at (0.35,1) {$1$};
\node[left,gray!80] at (0.35,2) {$2$};
\node[left,gray!80] at (0.35,3) {$3$};

\draw[line width=1.1pt,black]
  (0.5,0)--(1.5,1)--(2.5,2)--(3.5,1)
  --(4.5,2)--(5.5,3)--(6.5,2)--(7.5,1)--(8.5,0);
\draw[line width=1.1pt,gray!60] (8.5,0)--(9.5,0);

\foreach \p in {(0.5,0),(1.5,1),(2.5,2),(3.5,1),(4.5,2),(5.5,3),(6.5,2),(7.5,1),(8.5,0)}{
  \fill \p circle (2.2pt);
}

\node[left] at (0.35,-0.45) {$\eta$};

\foreach \x in {1,2,4,5}{
  \fill (\x,-0.45) circle (2.4pt);
}
\foreach \x in {3,6,7,8,9}{
  \draw[gray!80,fill=white] (\x,-0.45) circle (2.3pt);
}

\foreach \x in {1,...,9}{
  \node[below=13pt,gray!80] at (\x,-0.45) {\scriptsize $\x$};
}

\node[eventlabel,text=seatone] at (1.0,0.5) {$(1,\uparrow)$};
\node[eventlabel,text=seattwo] at (2.0,1.5) {$(2,\uparrow)$};
\node[eventlabel,text=seatone] at (3.0,1.5) {$(1,\downarrow)$};
\node[eventlabel,text=seatone] at (4.0,1.5) {$(1,\uparrow)$};
\node[eventlabel,text=seatthree] at (5.0,2.5) {$(3,\uparrow)$};
\node[eventlabel,text=seatone] at (6.0,2.5) {$(1,\downarrow)$};
\node[eventlabel,text=seattwo] at (7.0,1.5) {$(2,\downarrow)$};
\node[eventlabel,text=seatthree] at (8.0,0.5) {$(3,\downarrow)$};



\node[left,text=seatthree] at (0.25,-1.50) {\scriptsize No.\,3};
\node[left,text=seattwo]   at (0.25,-1.90) {\scriptsize No.\,2};
\node[left,text=seatone]   at (0.25,-2.30) {\scriptsize No.\,1};

\foreach \x in {0,...,8}{
  \begin{scope}[shift={(\x+0.5,-2.50)}]
    \draw[carrierbox] (-0.24,0) rectangle (0.24,1.20);
    \draw[carrierbox] (-0.24,0.40)--(0.24,0.40);
    \draw[carrierbox] (-0.24,0.80)--(0.24,0.80);
  \end{scope}
}


\fill (1.5,-2.30) circle (1.8pt);

\fill (2.5,-2.30) circle (1.8pt);
\fill (2.5,-1.90) circle (1.8pt);

\fill (3.5,-1.90) circle (1.8pt);

\fill (4.5,-2.30) circle (1.8pt);
\fill (4.5,-1.90) circle (1.8pt);

\fill (5.5,-2.30) circle (1.8pt);
\fill (5.5,-1.90) circle (1.8pt);
\fill (5.5,-1.50) circle (1.8pt);

\fill (6.5,-1.90) circle (1.8pt);
\fill (6.5,-1.50) circle (1.8pt);

\fill (7.5,-1.50) circle (1.8pt);


\fill[seatone!12] (1.5-0.24,-2.50) rectangle (1.5+0.24,-2.10);
\draw[seatone,line width=0.6pt] (1.5-0.24,-2.50) rectangle (1.5+0.24,-2.10);

\fill[seattwo!12] (2.5-0.24,-2.10) rectangle (2.5+0.24,-1.70);
\draw[seattwo,line width=0.6pt] (2.5-0.24,-2.10) rectangle (2.5+0.24,-1.70);

\fill[seatone!12] (3.5-0.24,-2.50) rectangle (3.5+0.24,-2.10);
\draw[seatone,line width=0.6pt] (3.5-0.24,-2.50) rectangle (3.5+0.24,-2.10);

\fill[seatone!12] (4.5-0.24,-2.50) rectangle (4.5+0.24,-2.10);
\draw[seatone,line width=0.6pt] (4.5-0.24,-2.50) rectangle (4.5+0.24,-2.10);

\fill[seatthree!12] (5.5-0.24,-1.70) rectangle (5.5+0.24,-1.30);
\draw[seatthree,line width=0.6pt] (5.5-0.24,-1.70) rectangle (5.5+0.24,-1.30);

\fill[seatone!12] (6.5-0.24,-2.50) rectangle (6.5+0.24,-2.10);
\draw[seatone,line width=0.6pt] (6.5-0.24,-2.50) rectangle (6.5+0.24,-2.10);

\fill[seattwo!12] (7.5-0.24,-2.10) rectangle (7.5+0.24,-1.70);
\draw[seattwo,line width=0.6pt] (7.5-0.24,-2.10) rectangle (7.5+0.24,-1.70);

\fill[seatthree!12] (8.5-0.24,-1.70) rectangle (8.5+0.24,-1.30);
\draw[seatthree,line width=0.6pt] (8.5-0.24,-1.70) rectangle (8.5+0.24,-1.30);

\fill (1.5,-2.30) circle (1.8pt);
\fill (2.5,-2.30) circle (1.8pt);
\fill (2.5,-1.90) circle (1.8pt);
\fill (3.5,-1.90) circle (1.8pt);
\fill (4.5,-2.30) circle (1.8pt);
\fill (4.5,-1.90) circle (1.8pt);
\fill (5.5,-2.30) circle (1.8pt);
\fill (5.5,-1.90) circle (1.8pt);
\fill (5.5,-1.50) circle (1.8pt);
\fill (6.5,-1.90) circle (1.8pt);
\fill (6.5,-1.50) circle (1.8pt);
\fill (7.5,-1.50) circle (1.8pt);

\foreach \x in {1,...,8}{
  \draw[gray!45,->] (\x,-0.80)--(\x+0.5,-1.25);
}


\end{tikzpicture}
\caption{
Carrier graph and carrier with seat numbers for $\eta=11011000\cdots$. 
The lower part shows the state of the numbered seats of the carrier after it passes each site.
}
\label{fig:carrier-graph-with-seat-numbers}
\end{figure}
                In other words, we define $\mathcal{W}_{k}(\eta, \cdot) : \Z_{\ge 0} \to \{0,1\}$ recursively as follows : $\mathcal{W}_{k}(\eta, 0) := 0$ for any $k \in \N$, and
                    \begin{align}
                        \mathcal{W}_{k}(\eta, x) & =\mathcal{W}_{k}(\eta, x-1) + \eta(x)(1-\mathcal{W}_{k}(\eta, x-1))\prod_{\ell=1}^{k-1}\mathcal{W}_{\ell}(\eta, x-1) \\  & \quad  -(1-\eta(x))\mathcal{W}_{k}(\eta, x-1)\prod_{\ell=1}^{k-1}(1-\mathcal{W}_{\ell}(\eta, x-1)). 
                    \end{align}
                {The carrier with seat numbers may be regarded as a refinement of the usual carrier with finite capacity.  To see this, let us look only at the first $\ell$ seats.  Then the total number of balls in these seats changes according to the same rule as the carrier with capacity $\ell$.  Hence, from the construction of $\mathcal{W}_k$, we obtain}
                    \begin{align}\label{eq:car_l_seat}
                        W_{\ell}\left(\eta, x\right) = \sum_{k = 1}^{\ell} \mathcal{W}_{k}\left(\eta, x\right),
                    \end{align}
                for any $\ell \in \N$ and $x \in \N$, where $W_{\ell}$ is the carrier with capacity $\ell$ defined in Introduction. 
                Then, the seat number configuration $\eta^{\sigma}_{k} \in \Omega, k \in \N, \sigma \in \left\{ \uparrow, \downarrow \right\}$ is defined as 
                    \begin{align}
                        \eta^{\uparrow}_{k}(x) &:= \begin{dcases}
                                            1 \ & \ \text{if } \mathcal{W}_{k}(\eta, x) > \mathcal{W}_{k}(\eta, x - 1), \\
                                            0 \ & \ \text{otherwise}
                        \end{dcases} \\
                        &\ = \eta(x)(1-\mathcal{W}_{k}(\eta, x-1))\prod_{\ell=1}^{k-1}\mathcal{W}_{\ell}(\eta, x-1) \label{def:seatup}, 
                    \end{align}
                and
                    \begin{align}
                        \eta^{\downarrow}_{k}(x) &:= \begin{dcases}
                                            1 \ & \ \text{if } \mathcal{W}_{k}(\eta, x) < \mathcal{W}_{k}(\eta, x - 1), \\
                                            0 \ & \ \text{otherwise}
                        \end{dcases} \\ 
                        &\ = (1-\eta(x))\mathcal{W}_{k}(\eta, x-1)\prod_{\ell=1}^{k-1}(1-\mathcal{W}_{\ell}(\eta, x-1)) \label{def:seatdown}. 
                    \end{align}
                Here, $\eta^{\uparrow}_{k}(x) = 1$ (resp. $\eta^{\downarrow}_{k}(x) = 1$) means that a ball gets into (resp. off) No.$k$ seat at site $x$. We say that $x \in \N$ is a $\left(k,\uparrow\right)$-seat (resp. $\left(k,\downarrow\right)$-seat) if $\eta^{\uparrow}_{k}(x) = 1$ (resp. $\eta^{\downarrow}_{k}(x) = 1$). 
                
                We also define 
                $r\left(\eta, \cdot\right) \in \Omega$ as 
                    \begin{align}\label{def:record}
                        r\left(\eta, x \right) := 1 - \sum_{k \in \N} \sum_{\sigma \in \left\{ \uparrow, \downarrow \right\}} \eta^{\sigma}_{k}(x) { \in \{0,1\}}.
                    \end{align}
                We say that $x \in \N$ is a {\it record} if $r\left(\eta, x \right) = 1$. 
            \begin{remark}\label{rem:record}
                $r\left(\eta, x \right) = 1$ means that the carrier with infinite capacity $W_{\infty}$ {does not change its state at} the site $x$, i.e., $W_{\infty}\left(\eta,x-1\right) = W_{\infty}\left(\eta,x\right) = 0$. In other words, $r\left(\eta, x \right) = 1$ if and only if for any $1 \le z \le x$,
                    \begin{align}
                        \sum_{y = z}^{x} {\left(2\eta\left(y\right) - 1\right)} < 0.
                    \end{align}
                This characterization of records will be used in Section \ref{sec:line}. 
            \end{remark}
            \begin{remark}\label{rem:seat_sol}
                The relation between solitons and the seat number configuration is as follows. Suppose that $\gamma \subset \N$ is a $k$-soliton in $\eta$, and we define 
                    \begin{align}
                        H(\gamma) &:= \left\{ x \in \gamma \ ; \ \eta(x) = 1 \right\} = \left( H_{1} < \dots < H_{k} \right), \\
                        T(\gamma) &:= \left\{ x \in \gamma \ ; \ \eta(x) = 0 \right\} = \left( T_{1} < \dots < T_{k} \right).
                    \end{align}
                Then, for any $1 \le \ell \le k$, at $H_{\ell}$ (resp. $T_{\ell}$), a ball gets into (resp. off) No. $\ell$ seat, i.e., 
                    \begin{align}
                        \eta^{\uparrow}_{\ell}\left(H_{\ell}\right) = \eta^{\downarrow}_{\ell}\left(T_{\ell}\right) = 1,
                    \end{align}
                see \cite[Section 2]{MSSS} for details. Thus, a $k$-soliton is composed of one of each $\left(\ell, \sigma\right)$-seat for $1 \le \ell \le k$ and $\sigma \in \{\uparrow, \downarrow\}$. {See Figure \ref{fig:seat_soliton} for an illustration of this relation in terms of the carrier graph.} Hence, if $\eta \in \Omega_{<\infty}$, then the quantity $\sum_{x \in \N} \eta^{\uparrow}_{k}(x) = \sum_{x \in \N} \eta^{\downarrow}_{k}(x)$ represents the total number of solitons with size $k$ and larger. We also note that a record can be seen as a site that is not a component of any soliton. 
            \end{remark}

\begin{figure}
    \centering
    \begin{tikzpicture}[
    x=0.75cm,y=0.75cm,
    >=Stealth,
    font=\small
]

\colorlet{threeSol}{blue!75!black}
\colorlet{oneSol}{orange!85!black}
\colorlet{seatOne}{red!70!black}
\colorlet{seatTwo}{green!50!black}
\colorlet{seatThree}{blue!70!black}

\begin{scope}[shift={(0,6.8)}]


\foreach \x in {0,...,9}{
  \draw[gray!18] (\x+0.5,-1.05)--(\x+0.5,3.35);
}
\foreach \y in {0,...,3}{
  \draw[gray!18] (0.35,\y)--(9.65,\y);
}

\draw[->,gray!70] (0.35,0)--(9.85,0) node[right] {$x$};
\draw[->,gray!70] (0.5,-0.05)--(0.5,3.45)
  node[above] {$W_\infty(\eta,\cdot)$};

\draw[line width=1.1pt]
  (0.5,0)--(1.5,1)--(2.5,2)--(3.5,3)
  --(4.5,2)--(5.5,1)--(6.5,0);

\draw[line width=1.1pt,gray!60]
  (6.5,0)--(9.5,0);

\foreach \x/\lab in {1/{H_1},2/{H_2},3/{H_3}}{
  \fill (\x,-0.45) circle (2.4pt);
  \node[below=2pt] at (\x,-0.45) {$1$};
  \node[above=2pt] at (\x,-0.45) {$\lab$};
}
\foreach \x/\lab in {4/{T_1},5/{T_2},6/{T_3}}{
  \draw[gray!80,fill=white] (\x,-0.45) circle (2.4pt);
  \node[below=2pt,gray!80] at (\x,-0.45) {$0$};
  \node[above=2pt] at (\x,-0.45) {$\lab$};
}

\foreach \x in {7,8,9}{
  \draw[gray!80,fill=white] (\x,-0.45) circle (2.4pt);
  \node[below=2pt,gray!80] at (\x,-0.45) {$0$};
}

\foreach \x in {1,...,9}{
  \node[below=15pt,gray!80] at (\x,-0.45) {\scriptsize $\x$};
}

\node[fill=white,inner sep=1pt,text=seatOne] at (1,0.68)
  {\scriptsize $(1,\uparrow)$};
\node[fill=white,inner sep=1pt,text=seatTwo] at (2,1.5)
  {\scriptsize $(2,\uparrow)$};
\node[fill=white,inner sep=1pt,text=seatThree] at (3,2.5)
  {\scriptsize $(3,\uparrow)$};

\node[fill=white,inner sep=1pt,text=seatOne] at (4,2.5)
  {\scriptsize $(1,\downarrow)$};
\node[fill=white,inner sep=1pt,text=seatTwo] at (5,1.5)
  {\scriptsize $(2,\downarrow)$};
\node[fill=white,inner sep=1pt,text=seatThree] at (6,0.68)
  {\scriptsize $(3,\downarrow)$};



\end{scope}

\begin{scope}[shift={(0,0)}]


\foreach \x in {0,...,10}{
  \draw[gray!18] (\x+0.5,-1.15)--(\x+0.5,3.35);
}
\foreach \y in {0,...,3}{
  \draw[gray!18] (0.35,\y)--(10.65,\y);
}

\draw[->,gray!70] (0.35,0)--(10.85,0) node[right] {$x$};
\draw[->,gray!70] (0.5,-0.05)--(0.5,3.45)
  node[above] {$W_\infty(\eta,\cdot)$};

\draw[line width=1.1pt,threeSol]
  (0.5,0)--(1.5,1)--(2.5,2);
\draw[line width=1.1pt,oneSol]
  (2.5,2)--(3.5,1)--(4.5,2);
\draw[line width=1.1pt,threeSol]
  (4.5,2)--(5.5,3)--(6.5,2)--(7.5,1)--(8.5,0);
\draw[line width=1.1pt,gray!60]
  (8.5,0)--(10.5,0);

\foreach \x/\lab in {
  1/{H_1},
  2/{H_2},
  4/{H_1},
  5/{H_3}
}{
  \fill (\x,-0.45) circle (2.4pt);
  \node[below=2pt] at (\x,-0.45) {$1$};
  \node[above=2pt] at (\x,-0.45) {$\lab$};
}
\foreach \x/\lab in {
  3/{T_1},
  6/{T_1},
  7/{T_2},
  8/{T_3}
}{
  \draw[gray!80,fill=white] (\x,-0.45) circle (2.4pt);
  \node[below=2pt,gray!80] at (\x,-0.45) {$0$};
  \node[above=2pt] at (\x,-0.45) {$\lab$};
}
\foreach \x in {9,10}{
  \draw[gray!80,fill=white] (\x,-0.45) circle (2.4pt);
  \node[below=2pt,gray!80] at (\x,-0.45) {$0$};
}

\foreach \x in {1,...,10}{
  \node[below=15pt,gray!80] at (\x,-0.45) {\scriptsize $\x$};
}

\node[fill=white,inner sep=1pt,text=seatOne] at (1,0.68)
  {\scriptsize $(1,\uparrow)$};
\node[fill=white,inner sep=1pt,text=seatTwo] at (2,1.5)
  {\scriptsize $(2,\uparrow)$};
\node[fill=white,inner sep=1pt,text=seatThree] at (5,2.5)
  {\scriptsize $(3,\uparrow)$};

\node[fill=white,inner sep=1pt,text=seatOne] at (6,2.5)
  {\scriptsize $(1,\downarrow)$};
\node[fill=white,inner sep=1pt,text=seatTwo] at (7,1.5)
  {\scriptsize $(2,\downarrow)$};
\node[fill=white,inner sep=1pt,text=seatThree] at (8,0.68)
  {\scriptsize $(3,\downarrow)$};

\node[fill=white,inner sep=1pt,text=oneSol] at (3,1.5)
  {\scriptsize $(1,\downarrow)$};
\node[fill=white,inner sep=1pt,text=oneSol] at (4,1.5)
  {\scriptsize $(1,\uparrow)$};



\end{scope}

\end{tikzpicture}

    \caption{The upper figure shows a single $3$-soliton, which consists of one
$(\ell,\uparrow)$-seat and one $(\ell,\downarrow)$-seat for each
$\ell=1,2,3$. The lower figure shows a $3$-soliton and a $1$-soliton during interaction.}
    \label{fig:seat_soliton}
\end{figure}

                We note that $\mathcal{W}_{k}$ can be represented as 
                    \begin{align}\label{eq:car_n_seat}
                        \mathcal{W}_{k}\left(\eta, x \right) = \sum_{y = 1}^{x} \left(\eta^{\uparrow}_{k}\left(y\right) - \eta^{\downarrow}_{k}\left(y\right)\right),
                    \end{align}
                for any $k \in \N$ and $x \in \N$. Also, we recall a  useful lemma proved in \cite{MSSS}. 
                    \begin{lemma}[Lemma 3.1 in \cite{MSSS}]\label{lem:1_half}
                        Suppose that $\eta \in \Omega$. Then, for any $k \in \N$ and $x \in \N$, we have the following. 
                    \begin{enumerate}
                        \item $\eta^{\uparrow}_{k}\left(x\right) = 1$ implies $\sum_{y = 1}^{x} \left(\eta^{\uparrow}_{\ell}\left(y\right) - \eta^{\downarrow}_{\ell}\left(y\right)\right) = 1$ for any $1 \le \ell \le k$.
                        \item $\eta^{\downarrow}_{k}\left(x\right) = 1$ implies $\sum_{y = 1}^{x} \left(\eta^{\uparrow}_{\ell}\left(y\right) - \eta^{\downarrow}_{\ell}\left(y\right)\right) = 0$ for any $1 \le \ell \le k$.
                        \item $r\left(\eta, x\right) = 1$ implies $\sum_{y = 1}^{x} \left(\eta^{\uparrow}_{\ell}\left(y\right) - \eta^{\downarrow}_{\ell}\left(y\right)\right) = 0$ for any $\ell \in \N$.
                    \end{enumerate}
                        
                    \end{lemma}

                {We now introduce some functions which will be useful for counting
solitons.  The idea is to use the seat number
configuration to decompose the original configuration into intervals,
and then to count the number of solitons of each size in each
interval.}  For any $k \in \N \cup \{\infty\}$ and $x \in \Z_{\ge 0}$, we define $\xi_{k}\left( \eta, x \right)$ as $\xi_{k}\left( \eta, 0 \right) := 0$, and
                    \begin{align}
                        \xi_{k}\left( \eta, x \right) &:= x -  \sum_{y = 1}^{x} \sum_{ \ell = 1}^{k} \sum_{\sigma \in \left\{ \uparrow, \downarrow \right\}} \eta^{\sigma}_{\ell}\left(y\right) \\
                        &\ = \sum_{y = 1}^{x} \left( r(y) +  \sum_{ \ell \ge k + 1} \sum_{\sigma \in \left\{ \uparrow, \downarrow \right\}}  \eta^{\sigma}_{\ell}\left(y\right) \right).
                    \end{align}
                {Thus $\xi_k(\eta,x)$ counts the number of sites up to $x$ which remain
after removing sites corresponding to the seats numbered
$1,\ldots,k$. }
                We remark that $\xi_{k}\left( \eta, x \right) - \xi_{k}\left( \eta, x - 1 \right) = 1$ if and only if {$x$ is a record or } $x$ is a $\left(k+\ell, \sigma\right)$-seat with some $\ell \in \N$ and $\sigma \in \{\uparrow, \downarrow\}$, and $\xi_{k}\left( \eta, x \right) - \xi_{k}\left( \eta, x - 1 \right) = 0$ otherwise.
                Then, for any $k \in \N \cup\{\infty\}$ and $i \in \Z_{\ge 0}$, we define $s_{k}\left(\eta, i \right)$ as 
                    \begin{align}\label{def:sk}
                        s_{k}\left(\eta, i \right) &:= \min\left\{ x \in \Z_{\ge 0} \ ; \ \xi_{k}\left( \eta, x \right) = i \right\},
                    \end{align}
                where $\min\emptyset := \infty$. {In other words, the map $s_k(\eta,\cdot)$ is the inverse of the counting
function $\xi_k\left(\eta,\cdot\right)$ and gives the positions of the remaining sites after the above
removing procedure.} Note that $s_{k}\left(\eta, 0 \right) = 0$ for any $k \in \N \cup\{\infty\}$. 
{The intervals $\left[s_{k}\left(\eta, i \right) + 1,  s_{k}\left(\eta, i + 1 \right) \right]$  between two consecutive such sites
will be used to count solitons of each size.}
                For any $k \in \N$ and $i \in \Z_{\ge 0}$, we define $\zeta_{k}\left( \eta, i \right)$ as
                    \begin{align}\label{def:zeta}
                        \zeta_{k}\left( \eta, i \right) &:= \sum_{y = s_{k}\left(\eta, i \right) + 1}^{s_{k}\left(\eta, i + 1 \right)} \left( \eta^{\uparrow}_{k}\left(y\right) - \eta^{\uparrow}_{k+1}\left(y\right) \right).
                    \end{align}
                Thanks to Lemma \ref{lem:1_half}, $\zeta_{k}\left( \eta, i \right)$ can be represented as 
                    \begin{align}\label{eq:down_zeta}
                                \zeta_{k}\left( \eta, i \right) = \sum_{y = s_{k}\left(\eta, i \right) + 1}^{s_{k}\left(\eta, i + 1 \right)} \left( \eta^{\downarrow}_{k}\left(y\right) - \eta^{\downarrow}_{k+1}\left(y\right) \right).
                            \end{align}
                {As explained in Remark \ref{rem:zeta_sol} below, this quantity has the interpretation
of counting $k$-solitons in the above interval.}
                \begin{remark}\label{rem:zeta_sol}

                {
The quantity $\zeta_k(\eta,i)$ can be interpreted as the number of
$k$-solitons contained in the interval $[s_k(\eta,i)+1,\;s_k(\eta,i+1)]$.
Actually, by the definition of $s_k\left(\eta,\cdot\right)$, all sites in $[s_k(\eta,i)+1,\;s_k(\eta,i+1) - 1]$
correspond to seats numbered at most $k$.  Hence, in the interval
$[s_k(\eta,i)+1,\;s_k(\eta,i+1)]$, a seat numbered $k+1$ or larger may
appear only at the right endpoint.  By Remark \ref{rem:seat_sol}, the number of $(k,\uparrow)$-seats counts
solitons of size at least $k$, while the number of
$(k+1,\uparrow)$-seats counts solitons of size at least $k+1$. 
Therefore their difference counts solitons of size $k$, which is  $\zeta_k(\eta,i)$.  See
Figure~\ref{fig:zeta-soliton-example} for an illustration} and \cite[Section 2]{MSSS} for details. {We also note that a record can be seen as a site that is not a component of
any soliton. Thus records separate the soliton decomposition in the following
sense: solitons are identified independently between consecutive records, and
no soliton crosses a record.}
                \end{remark}
        \begin{figure}[t]
\centering
\begin{tikzpicture}[scale=0.8, every node/.style={transform shape}]

\colorlet{soltwo}{blue!75!black}
\colorlet{solthree}{purple!75!black}
\colorlet{solone}{orange!85!black}

\colorlet{kone}{red!70!black}
\colorlet{ktwo}{green!50!black}
\colorlet{kthree}{blue!70!black}
\colorlet{recordc}{gray!70}

\tikzset{
  eventlabel/.style={
    font=\tiny,
    inner sep=0.9pt,
    fill=white,
    fill opacity=0.9,
    text opacity=1,
    rounded corners=1pt
  },
  skmark/.style={
    draw=gray!65,
    fill=white,
    rounded corners=1pt,
    inner sep=0.8pt,
    font=\tiny
  }
}

\foreach \x in {0,...,16}{
  \draw[gray!18] (\x+0.5,-0.90)--(\x+0.5,3.40);
}
\foreach \y in {0,...,3}{
  \draw[gray!18] (0.35,\y)--(16.65,\y);
}

\draw[->,gray!70] (0.35,0)--(16.95,0) node[right] {$x$};
\draw[->,gray!70] (0.5,-0.05)--(0.5,3.52)
  node[above] {$W_\infty(\eta,\cdot)$};

\node[left,gray!80] at (0.35,0) {$0$};
\node[left,gray!80] at (0.35,1) {$1$};
\node[left,gray!80] at (0.35,2) {$2$};
\node[left,gray!80] at (0.35,3) {$3$};

\draw[line width=1.15pt,soltwo]
  (0.5,0)--(1.5,1)--(2.5,2)--(3.5,1)--(4.5,0);

\draw[line width=1.00pt,gray!60]
  (4.5,0)--(5.5,0);

\draw[line width=1.15pt,solthree]
  (5.5,0)--(6.5,1)--(7.5,2)--(8.5,3)--(9.5,2);

\draw[line width=1.15pt,solone]
  (9.5,2)--(10.5,3)--(11.5,2);

\draw[line width=1.15pt,solone]
  (11.5,2)--(12.5,3)--(13.5,2);

\draw[line width=1.15pt,solthree]
  (13.5,2)--(14.5,1)--(15.5,0);

\draw[line width=1.00pt,gray!60]
  (15.5,0)--(16.5,0);

\foreach \p in {(0.5,0),(1.5,1),(2.5,2),(3.5,1),(4.5,0),(5.5,0),
(6.5,1),(7.5,2),(8.5,3),(9.5,2),(10.5,3),(11.5,2),(12.5,3),(13.5,2),
(14.5,1),(15.5,0)}{
  \fill \p circle (2.0pt);
}

\node[eventlabel,text=kone]   at (1.00,0.5) {$ (1,\uparrow) $};
\node[eventlabel,text=ktwo]   at (2.00,1.5) {$ (2,\uparrow) $};
\node[eventlabel,text=kone]   at (3.00,1.5) {$ (1,\downarrow) $};
\node[eventlabel,text=ktwo]   at (4.00,0.5) {$ (2,\downarrow) $};

\node[eventlabel,text=recordc] at (5.00,0.23) {record};

\node[eventlabel,text=kone]   at (6.00,0.5) {$ (1,\uparrow) $};
\node[eventlabel,text=ktwo]   at (7.00,1.5) {$ (2,\uparrow) $};
\node[eventlabel,text=kthree] at (8.00,2.5) {$ (3,\uparrow) $};
\node[eventlabel,text=kone]   at (9.00,2.5) {$ (1,\downarrow) $};

\node[eventlabel,text=kone]   at (10.00,2.5) {$ (1,\uparrow) $};
\node[eventlabel,text=kone]   at (11.00,2.5) {$ (1,\downarrow) $};

\node[eventlabel,text=kone]   at (12.00,2.5) {$ (1,\uparrow) $};
\node[eventlabel,text=kone]   at (13.00,2.5) {$ (1,\downarrow) $};

\node[eventlabel,text=ktwo]   at (14.00,1.5) {$ (2,\downarrow) $};
\node[eventlabel,text=kthree] at (15.00,0.5) {$ (3,\downarrow) $};

\node[eventlabel,text=recordc] at (16.00,0.23) {record};

\node[left] at (0.35,-0.45) {$\eta$};

\foreach \x in {1,2,6,7,8,10,12}{
  \fill (\x,-0.45) circle (2.3pt);
}
\foreach \x in {3,4,5,9,11,13,14,15,16}{
  \draw[gray!80,fill=white] (\x,-0.45) circle (2.2pt);
}

\foreach \x in {1,...,16}{
  \node[below=12pt,gray!80] at (\x,-0.45) {\scriptsize $\x$};
}






\def\ysone{-1.95}
\def\yzone{-2.32}
\def\ystwo{-2.95}
\def\yztwo{-3.32}
\def\ysthree{-3.95}
\def\yzthree{-4.32}

\node[left] at (0.35,\ysone) {$s_1$};
\node[left] at (0.35,\yzone) {$\zeta_1$};

\node[left] at (0.35,\ystwo) {$s_2$};
\node[left] at (0.35,\yztwo) {$\zeta_2$};

\node[left] at (0.35,\ysthree) {$s_3$};
\node[left] at (0.35,\yzthree) {$\zeta_3$};

\draw[gray!25] (0.50,\ysone)--(16.20,\ysone);
\draw[gray!25] (0.50,\ystwo)--(16.20,\ystwo);
\draw[gray!25] (0.50,\ysthree)--(16.20,\ysthree);

\foreach \x/\i in {0.5/0,2/1,4/2,5/3,7/4,8/5,14/6,15/7,16/8}{
  \node[skmark] at (\x,\ysone) {$\i$};
}
\foreach \a/\b in {0.5/2,2/4,4/5,5/7,7/8,8/14,14/15,15/16}{
  \draw[gray!55] (\a,\ysone)--(\b,\ysone);
}
\foreach \x/\v in {1.25/0,3/0,4.5/0,6/0,7.5/0,11/2,14.5/0,15.5/0}{
  \node[font=\scriptsize] at (\x,\yzone) {$\v$};
}

\foreach \x/\i in {0.5/0,5/1,8/2,15/3,16/4}{
  \node[skmark] at (\x,\ystwo) {$\i$};
}
\foreach \a/\b in {0.5/5,5/8,8/15,15/16}{
  \draw[gray!55] (\a,\ystwo)--(\b,\ystwo);
}
\foreach \x/\v in {2.75/1,6.5/0,11.5/0,15.5/0}{
  \node[font=\scriptsize] at (\x,\yztwo) {$\v$};
}

\foreach \x/\i in {0.5/0,5/1,16/2}{
  \node[skmark] at (\x,\ysthree) {$\i$};
}
\foreach \a/\b in {0.5/5,5/16}{
  \draw[gray!55] (\a,\ysthree)--(\b,\ysthree);
}
\foreach \x/\v in {2.75/0,10.5/1}{
  \node[font=\scriptsize] at (\x,\yzthree) {$\v$};
}


\end{tikzpicture}
\caption{
Carrier graph and seats  for $\eta=1100011101010000\cdots$.
There are two $1$-solitons, one $2$-soliton, and one $3$-soliton.
The rows labeled $s_1$, $s_2$, and $s_3$ show the sites $s_k(\eta,i)$, and
the numbers written between consecutive sites give the corresponding values
of $\zeta_k(\eta,i)$.
In particular, the non-zero values shown in the figure are
$\zeta_1(\eta,5)=2$, $\zeta_2(\eta,0)=1$, and $\zeta_3(\eta,1)=1$.
}
\label{fig:zeta-soliton-example}
\end{figure}
                \begin{remark}\label{rem:initial}
                    We note that $\zeta$ can be considered as a map $\zeta : \Omega_{r} \to \Z_{\ge 0}^{\mathbb{\N} \times \Z_{\ge 0}}$, and it is shown in \cite{CS,MSSS} that $\zeta$ is a bijection between $\Omega_{r} \subset \Omega$ and $\bar{\Omega} \subset \Z_{\ge 0}^{\mathbb{N} \times \Z_{\ge 0}}$, where $\Omega_{r}$ and $\bar{\Omega}$ are defined as
                        \begin{align}
                            \Omega_{r} &:= \left\{ \eta \in \Omega \ ; \ \sum_{x \in \N} r(\eta, x) = \infty  \right\}, \\
                            \bar{\Omega} &:= \left\{ \zeta \in  \Z_{\ge 0}^{\mathbb{N} \times \Z_{\ge 0}} \ ; \ \max\left\{ k \in \N \ ; \ \zeta_{k}(i) > 0 \right\} < \infty \ \text{for any } i \right\}.
                        \end{align} 
                \end{remark}
                
                We conclude this subsection by quoting a main result in \cite{MSSS}.  In \cite[Proposition 2.3 and Theorem 2.3]{MSSS}, it is shown that the dynamics of the BBS($\ell$) is linearized in terms of $\zeta$ : 
                    \begin{theorem}[Proposition 2.3 and Theorem 2.3 in \cite{MSSS}]\label{thm:linear}
                        Suppose that $\eta \in \Omega$ and  there exist some $k \in \N$ and $i \in \Z_{\ge 0}$ such that $s_{k}\left(\eta, i + 1 \right) < \infty$. Then we have
                            \begin{align}
                                \zeta_k\left( T_{\ell} \eta , i  \right) = \zeta_k\left( \eta , i - {k \wedge \ell}  \right), 
                            \end{align}
                        with convention $\zeta_k\left( \eta , i \right) = 0$ for any $i < 0$. 
                    \end{theorem}
                    \begin{remark}
                        Note that if $\eta \in \Omega_{r}$, then $s_{k}\left(\eta, i  \right) < \infty$ for any $k \in \N$ and $i \in \Z_{\ge 0}$. Therefore, by combining the bijectivity of $\zeta$ and Theorem \ref{thm:linear}, the initial value problem of the BBS($\ell$) can be solved when the initial configuration is an element of $\Omega_{r}$.
                    \end{remark}

		\subsection{The \texorpdfstring{$k$}{k}-skip map}\label{sec:seat_skip}
		  {We next introduce the $k$-skip map, which is a natural operation from the
viewpoint of the seat number configuration. For any $k \in \N$, we define $\Psi_{k} : \Omega \to \Omega$ as}
		        \begin{align}
		            \Psi_{k}\left(\eta\right)(x) := \eta\left( s_{k}\left(\eta, x\right) \right).
		        \end{align}
            We may call $\Psi_{k}\left(\eta\right)$ the $k$-skipped configuration (of $\eta$).
            In this subsection, we investigate some basic properties of $\left(\Psi_{k}\right)_{k \in \N}$. 
            {The map $\Psi_k$ can be understood in terms of the carrier with seat
numbers as follows.  We remove from $\eta$ the sites at which a ball
enters or leaves one of the seats numbered $1,\ldots,k$, and then read
the remaining sites from left to right.  This gives the configuration
$\Psi_k(\eta)$.  
In terms of the seat number configuration, it shifts the seat numbers by
$k$.  More precisely, the $(k+\ell,\sigma)$-seats in $\eta$ becomes the
$(\ell,\sigma)$-seat in $\Psi_k(\eta)$, as stated in Proposition \ref{prop:seat_semig} below. }
                \begin{proposition} \label{prop:seat_semig}
                    Suppose that $\eta \in \Omega$. For any $k, \ell \in \N$, $\sigma \in \{ \uparrow, \downarrow \}$ and $x \in \N$, we have
                        \begin{align}\label{eq:seat_semig}
                            \Psi_{k}\left( \eta \right)^{\sigma}_{\ell}\left( x \right) &= \eta^{\sigma}_{k+\ell}\left( s_{k}\left({\eta},x\right) \right).
                        \end{align}
                    In particular, 
                        \begin{align}\label{eq:seat_semig_rec}
                            r\left( \Psi_{k}\left( \eta \right), x\right) = r\left( \eta, s_{k}\left({\eta},x\right) \right). 
                        \end{align}
                \end{proposition}

            As a consequence of Proposition \ref{prop:seat_semig}, we obtain the following semigroup property of $\left(\Psi_{k}\right)_{k \in \N}$ : 
                \begin{proposition}\label{prop:semig}
                    Suppose that $\eta \in \Omega$. For any $k,\ell \in \N$ and $x \in \Z_{\ge 0}$, we have
                        \begin{align}
                            \Psi_{k}\left( \Psi_{\ell}\left( \eta \right) \right)\left( x \right) &= \Psi_{k + \ell}\left( \eta \right)\left( x \right).
                        \end{align}
                \end{proposition}
            In addition, we will also show that $\Psi_{k}$ induces a {\it shift operator} of $\left(\zeta_{k}\right)_{k \in \N}$ : 
                \begin{proposition}\label{prop:shift}
                    Suppose that $\eta \in \Omega$. For any $k \in \N$, $\ell \in \Z_{\ge 0}$ and $i \in \Z_{\ge 0}$, we have
                        \begin{align}
                            \zeta_{k}\left( \Psi_{\ell}\left(\eta\right) , i \right) = \zeta_{k + \ell}\left( \eta , i \right). 
                        \end{align}
                \end{proposition}
            \begin{remark}\label{rem:skip_sol}
            {The $k$-skip map can be regarded as an operation which removes, from
each soliton, the first $k$ heads and the first $k$ tails.  As a
result, a soliton of size $k+\ell$ in $\eta$ is viewed as a $\ell$-soliton
in $\Psi_k(\eta)$.  In particular, solitons of size at most $k$
disappear after applying $\Psi_k$.  See
Figure \ref{fig:psi1-soliton-interpretation} for an example in the
case $k=1$. }
            \end{remark}

    \begin{figure}[t]
\centering
\begin{tikzpicture}[
    x=0.72cm,y=0.72cm,
    >=Stealth,
    font=\small
]

\colorlet{soltwo}{blue!75!black}
\colorlet{solthree}{purple!75!black}
\colorlet{solone}{orange!85!black}
\colorlet{recordc}{gray!70}
\colorlet{removebox}{black}

\tikzset{
  lab/.style={
    font=\scriptsize,
    inner sep=1pt,
    fill=white,
    rounded corners=1pt
  },
  hlab/.style={
    font=\scriptsize,
    inner sep=1pt,
    fill=white,
    text=black
  },
  tlab/.style={
    font=\scriptsize,
    inner sep=1pt,
    fill=white,
    text=black
  },
  removed/.style={
    font=\scriptsize,
    inner sep=1pt,
    fill=white,
    text=black,
    circle,
    draw=removebox,
    line width=0.6pt
  }
}

\begin{scope}[shift={(0,6.2)}]


\foreach \x in {0,...,16}{
  \draw[gray!18] (\x+0.5,-0.9)--(\x+0.5,3.35);
}
\foreach \y in {0,...,3}{
  \draw[gray!18] (0.35,\y)--(16.65,\y);
}

\draw[->,gray!70] (0.35,0)--(16.95,0) node[right] {$x$};
\draw[->,gray!70] (0.5,-0.05)--(0.5,3.45)
  node[above] {$W_\infty(\eta,\cdot)$};

\node[left,gray!80] at (0.35,0) {$0$};
\node[left,gray!80] at (0.35,1) {$1$};
\node[left,gray!80] at (0.35,2) {$2$};
\node[left,gray!80] at (0.35,3) {$3$};

\draw[line width=1.15pt,soltwo]
  (0.5,0)--(1.5,1)--(2.5,2)--(3.5,1)--(4.5,0);

\draw[line width=1.00pt,gray!60]
  (4.5,0)--(5.5,0);

\draw[line width=1.15pt,solthree]
  (5.5,0)--(6.5,1)--(7.5,2)--(8.5,3)--(9.5,2);

\draw[line width=1.15pt,solone]
  (9.5,2)--(10.5,3)--(11.5,2);

\draw[line width=1.15pt,solone]
  (11.5,2)--(12.5,3)--(13.5,2);

\draw[line width=1.15pt,solthree]
  (13.5,2)--(14.5,1)--(15.5,0);

\draw[line width=1.00pt,gray!60]
  (15.5,0)--(16.5,0);

\foreach \p in {(0.5,0),(1.5,1),(2.5,2),(3.5,1),(4.5,0),(5.5,0),
(6.5,1),(7.5,2),(8.5,3),(9.5,2),(10.5,3),(11.5,2),(12.5,3),(13.5,2),
(14.5,1),(15.5,0)}{
  \fill \p circle (2.0pt);
}

\node[left] at (0.35,-0.45) {$\eta$};

\foreach \x in {1,2,6,7,8,10,12}{
  \fill (\x,-0.45) circle (2.3pt);
}
\foreach \x in {3,4,5,9,11,13,14,15,16}{
  \draw[gray!80,fill=white] (\x,-0.45) circle (2.2pt);
}
\foreach \x in {1,...,16}{
  \node[below=12pt,gray!80] at (\x,-0.45) {\scriptsize $\x$};
}






\node[removed,text=soltwo] at (1.0,0.5) {$H_1$};
\node[hlab,text=soltwo]    at (2.0,1.5) {$H_2$};
\node[removed,text=soltwo] at (3.0,1.5) {$T_1$};
\node[tlab,text=soltwo]    at (4.0,0.5) {$T_2$};

\node[removed,text=solthree] at (6.0,0.5) {$H_1$};
\node[hlab,text=solthree]    at (7.0,1.5) {$H_2$};
\node[hlab,text=solthree]    at (8.0,2.5) {$H_3$};
\node[removed,text=solthree] at (9.0,2.5) {$T_1$};
\node[tlab,text=solthree]    at (14.0,1.5) {$T_2$};
\node[tlab,text=solthree]    at (15.0,0.5) {$T_3$};

\node[removed,text=solone] at (10.0,2.5) {$H_1$};
\node[removed,text=solone] at (11.0,2.5) {$T_1$};

\node[removed,text=solone] at (12.0,2.5) {$H_1$};
\node[removed,text=solone] at (13.0,2.5) {$T_1$};

\end{scope}


\begin{scope}[shift={(0,0)}]


\foreach \x in {0,...,8}{
  \draw[gray!18] (\x+0.5,-0.9)--(\x+0.5,2.35);
}
\foreach \y in {0,1,2}{
  \draw[gray!18] (0.35,\y)--(8.65,\y);
}

\draw[->,gray!70] (0.35,0)--(8.95,0) node[right] {$x$};
\draw[->,gray!70] (0.5,-0.05)--(0.5,2.45)
  node[above] {$W_\infty(\Psi_1(\eta),\cdot)$};

\node[left,gray!80] at (0.35,0) {$0$};
\node[left,gray!80] at (0.35,1) {$1$};
\node[left,gray!80] at (0.35,2) {$2$};

\draw[line width=1.15pt,soltwo]
  (0.5,0)--(1.5,1)--(2.5,0);

\draw[line width=1.00pt,gray!60]
  (2.5,0)--(3.5,0);

\draw[line width=1.15pt,solthree]
  (3.5,0)--(4.5,1)--(5.5,2)--(6.5,1)--(7.5,0);

\draw[line width=1.00pt,gray!60]
  (7.5,0)--(8.5,0);

\foreach \p in {(0.5,0),(1.5,1),(2.5,0),(3.5,0),(4.5,1),(5.5,2),(6.5,1),(7.5,0)}{
  \fill \p circle (2.0pt);
}

\node[left] at (0.35,-0.45) {$\Psi_1(\eta)$};

\foreach \x in {1,4,5}{
  \fill (\x,-0.45) circle (2.3pt);
}
\foreach \x in {2,3,6,7,8}{
  \draw[gray!80,fill=white] (\x,-0.45) circle (2.2pt);
}
\foreach \x in {1,...,8}{
  \node[below=12pt,gray!80] at (\x,-0.45) {\scriptsize $\x$};
}




\node[hlab,text=soltwo]   at (1.0,0.5) {$H_1$};
\node[tlab,text=soltwo]   at (2.0,0.5) {$T_1$};

\node[hlab,text=solthree] at (4.0,0.5) {$H_1$};
\node[hlab,text=solthree] at (5.0,1.5) {$H_2$};
\node[tlab,text=solthree] at (6.0,1.5) {$T_1$};
\node[tlab,text=solthree] at (7.0,0.5) {$T_2$};


\end{scope}

\end{tikzpicture}
\caption{Solitonic interpretation of the $1$-skip map.
The map $\Psi_1$ removes the first head and the first tail of each soliton.
In this example, the $2$-soliton becomes a $1$-soliton, the
$3$-soliton becomes a $2$-soliton, and the $1$-solitons disappear.
}
\label{fig:psi1-soliton-interpretation}
\end{figure}
            
            {From now on} we will prove Propositions \ref{prop:seat_semig}, \ref{prop:semig} and \ref{prop:shift}.  First, we prepare some lemmas. Then we will give the proofs of the propositions. 
            
            \begin{lemma}\label{lem:sums}
                For any $k, \ell \in \N$, $z, w \in \Z_{\ge 0}, z \le w$ and $\sigma \in \{ \uparrow, \downarrow \}$, we have
                    \begin{align}\label{eq:sums}
                        \sum_{y = z}^{w} \eta^{\sigma}_{k + \ell}\left( s_{k}\left(\eta, y \right) \right) = \sum_{y = s_{k}\left(\eta, z \right)}^{s_{k}\left(\eta, w \right)} \eta^{\sigma}_{k + \ell}\left( y \right).
                    \end{align}
                In particular, we have
                    \begin{align}\label{eq:sums_1}
                        \sum_{y = s_{k}\left(\eta, z \right) + 1}^{s_{k}\left(\eta, z+1 \right) - 1} \eta^{\sigma}_{k + \ell}\left( y \right) = 0.
                    \end{align}
            \end{lemma}
            \begin{proof}[Proof of Lemma \ref{lem:sums}]
                Since all {$m ( > k)$-seats } in the interval $[s_{k}\left(\eta, z \right), s_{k}\left(\eta, w \right)]$ are {precisely } $\left\{ s_{k}\left(\eta, y \right) ; z \le y \le w \right\}$, {and the quantities in \eqref{eq:sums} only count $k+\ell$-seats,} we have \eqref{eq:sums}.
            \end{proof}
            \begin{lemma}\label{lem:carrier_cap}
                For any $k \in \N$ and $x \in \N$, we have
                    \begin{align}
                        \begin{dcases}\label{if:carrier}
                            W_{\ell}\left( \eta, s_{k}\left(\eta, x \right) \right) = \ell \ \text{for any }  1 \le \ell \le k + 1 \ \text{if } \eta\left(s_{k}\left(\eta, x \right) \right) = 1, \\
                            W_{\ell}\left( \eta, s_{k}\left(\eta, x \right) \right) = 0 \ \text{for any }  1 \le \ell \le k + 1 \ \text{if } \eta\left(s_{k}\left(\eta, x \right) \right) = 0.
                        \end{dcases}
                    \end{align}
                In particular, if $W_{\ell}\left( \eta, s_{k}\left(\eta, x \right) \right) > 0$ for some $1 \le \ell \le k + 1$, then  $W_{\ell}\left( \eta, s_{k}\left(\eta, x \right) \right) = \ell$ for any $1 \le \ell \le k + 1$. Also, if $W_{\ell}\left( \eta, s_{k}\left(\eta, x \right) \right) < \ell$ for some $1 \le \ell \le k + 1$, then $W_{\ell}\left( \eta, s_{k}\left(\eta, x \right) \right) = 0$ for any $1 \le \ell \le k + 1$.
            \end{lemma}
            \begin{proof}[Proof of Lemma \ref{lem:carrier_cap}]
                {First we} observe that { for any $1 \le m \le k$, $x \in \N$ and $\sigma \in \{\uparrow, \downarrow\}$, $\eta^{\sigma}_{m}\left(s_{k}\left(\eta, x \right) \right) = 0$. Thus by \eqref{def:record}, for any $x \in \N$ we obtain}
                    \begin{align}\label{eq:carrier_cap}
                        r\left(\eta,s_{k}\left(\eta, x \right) \right) + \sum_{\ell \in \N} \sum_{\sigma \in \{ \uparrow, \downarrow \}} \eta^{\sigma}_{k + \ell}\left(s_{k}\left(\eta, x \right) \right) = 1.
                    \end{align}
                {Now we show \eqref{if:carrier}. If $\eta\left(s_{k}\left(\eta, x \right) \right) = 1$, then from \eqref{eq:carrier_cap}, there exists some $m \in \N$ such that $\eta^{\uparrow}_{k + m}\left(s_{k}\left(\eta, x \right) \right) = 1$. Then, by \eqref{eq:car_n_seat} and Lemma \ref{lem:1_half} (1), for any $1 \le \ell \le k + m$ we have $W_{\ell}\left(\eta, s_{k}\left(\eta, x \right)\right) = \ell$. Next, if $\eta\left(s_{k}\left(\eta, x \right) \right) = 0$ and $r\left(\eta,s_{k}\left(\eta, x \right) \right) = 0$, then from \eqref{eq:carrier_cap}, there exists some $m \in \N$ such that $\eta^{\downarrow}_{k + m}\left(s_{k}\left(\eta, x \right) \right) = 1$. Then, by \eqref{eq:car_n_seat} and Lemma \ref{lem:1_half} (2), for any $1 \le \ell \le k + m$ we have $W_{\ell}\left(\eta, s_{k}\left(\eta, x \right)\right) = 0$. Finally, if $r\left(\eta,s_{k}\left(\eta, x \right) \right) = 1$, then by \eqref{eq:car_n_seat} and Lemma \ref{lem:1_half} (3), for any $\ell \in \N$ we have $W_{\ell}\left(\eta, s_{k}\left(\eta, x \right)\right) = 0$.}
            \end{proof}
            \begin{lemma}\label{lem:equi}
                For any $k \in \N$, $\ell \ge 2$ and $x \in \Z_{\ge 0}$, 
                    \begin{align}\label{equi:up2}
                        \eta^{\uparrow}_{k+\ell}\left( s_{k}\left(\eta, x + 1\right) \right) = 1 \ 
                        \text{if and only if} \ \begin{dcases} W_{k + \ell - 1}\left(\eta, s_{k}\left(\eta, x\right) \right) = k + \ell - 1, \\ \mathcal{W}_{k + \ell}\left(\eta, s_{k}\left(\eta, x\right) \right) = 0, \\ \eta\left( s_{k}\left(\eta, x + 1\right) \right) = 1, \end{dcases}
                    \end{align}
                and 
                    \begin{align}\label{equi:down2}
                        \eta^{\downarrow}_{k+\ell}\left( s_{k}\left(\eta, x+1\right) \right) = 1 \ 
                        \text{if and only if} \ \begin{dcases} W_{k + \ell - 1}\left(\eta, s_{k}\left(\eta, x\right) \right) = 0, \\ \mathcal{W}_{k + \ell}\left(\eta, s_{k}\left(\eta, x\right) \right) = 1, \\ \eta\left( s_{k}\left(\eta, x + 1\right) \right) = 0. \end{dcases}
                    \end{align}
                For the case $\ell = 1$, 
                    \begin{align}\label{equi:up}
                        \eta^{\uparrow}_{k+1}\left( s_{k}\left(\eta, x+1\right) \right) = 1 \ 
                        \text{if and only if} \ \begin{dcases} \mathcal{W}_{k + 1}\left(\eta, s_{k}\left(\eta, x\right) \right) = 0, \\ \eta\left( s_{k}\left(\eta, x + 1\right) \right) = 1, \end{dcases}
                    \end{align}
                and 
                    \begin{align}\label{equi:down}
                        \eta^{\downarrow}_{k+1}\left( s_{k}\left(\eta, x+1\right) \right) = 1 \ 
                        \text{if and only if} \ \begin{dcases} \mathcal{W}_{k + {1}}\left(\eta, s_{k}\left(\eta, x\right) \right) = 1, \\ \eta\left( s_{k}\left(\eta, x+1\right) \right) = 0. \end{dcases}
                    \end{align}
            \end{lemma}

            \begin{proof}[Proof of Lemma \ref{lem:equi}]
                We only prove \eqref{equi:up2} and then \eqref{equi:up}. \eqref{equi:down2} and \eqref{equi:down} can be proved in a similar way. 
                Since the {direction} $(\Leftarrow)$ of \eqref{equi:up2} is clear from Lemma \ref{lem:sums}, we will show the {direction} $(\Rightarrow)$ of \eqref{equi:up2}.
                
                We observe that from Lemma \ref{lem:sums}, for any $m, n \in \N$, $z \in \Z_{\ge 0}$ and $s_{m}\left(\eta,z\right) \le y \le s_{m}\left(\eta,z + 1\right) - 1$ we have 
                    \begin{align}\label{eq:lemequi1}
                        \mathcal{W}_{m + n}\left(\eta, s_{m}\left(\eta, z\right) \right) = \mathcal{W}_{m + n}\left(\eta, y \right).
                    \end{align}
                If $\eta^{\uparrow}_{k+\ell}\left( s_{k}\left({\eta},x+1\right) \right) = 1$, then we have $\eta\left( s_{k}\left({\eta},x+1\right) \right) = 1$, $\mathcal{W}_{k + \ell}\left(\eta, s_{k}\left(\eta,x + 1\right) - 1 \right) = 0$, and
                    \begin{align}\label{eq:lemequi2}
                        \mathcal{W}_{m}\left(\eta, s_{k}\left(\eta,x + 1\right) - 1 \right) = 1,
                    \end{align}
                for any $1 \le m \le k + \ell - 1$. For the case $\ell \ge 2$, from \eqref{eq:lemequi1} and \eqref{eq:lemequi2} we see that $W_{k + 1}\left(s_{k}\left(\eta,x\right)\right) > 0$, and thus from Lemma \ref{lem:carrier_cap}, we have  
                    \begin{align}
                        \eta^{\uparrow}_{k+\ell}\left( s_{k}\left(\eta,x+1\right) \right) = 1 \ 
                        &\text{implies} \ \begin{dcases} W_{k + \ell - 1}\left(\eta, s_{k}\left(\eta, x + 1\right) - 1 \right) = k + \ell - 1, \\ \mathcal{W}_{k + \ell}\left(\eta, s_{k}\left(\eta, x+1\right) - 1 \right) = 0, \\ \eta\left( s_{k}\left(\eta, x+1\right) \right) = 1, \end{dcases} 
                    \end{align}
                and
                    \begin{align}
                        &\begin{dcases} W_{k + \ell - 1}\left(\eta, s_{k}\left(\eta, x + 1\right) - 1 \right) = k + \ell - 1, \\ \mathcal{W}_{k + \ell}\left(\eta, s_{k}\left(\eta, x+1\right) - 1 \right) = 0, \end{dcases} \\ 
                        & \quad  \text{if and only if} \ \begin{dcases} W_{k + \ell - 1}\left(\eta, s_{k}\left(\eta, x\right) \right) = k + \ell - 1, \\ \mathcal{W}_{k + \ell}\left(\eta, s_{k}\left(\eta, x\right) \right) = 0. \end{dcases}
                    \end{align}
                Hence \eqref{equi:up2} is proved. 
                
                Next we will show \eqref{equi:up}. Since the direction $(\Rightarrow)$ is clear from \eqref{eq:lemequi1}, we check the the direction $(\Leftarrow)$. From the assumption $\mathcal{W}_{k+1}\left(\eta, s_{k}\left(\eta,x \right)\right) = 0$ and Lemma \ref{lem:carrier_cap}, we see that $W_{k + 1}\left(\eta, s_{k}\left(\eta, x\right) \right) = 0$. Then, from $\eta\left(s_{k}\left(\eta, x + 1\right)\right) = 1$, Lemmas \ref{lem:sums} and \ref{lem:carrier_cap}, we get $\mathcal{W}_{k+1}\left( \eta, s_{k}\left(\eta, x+1\right) - 1 \right) = 0$ and $W_{k}\left(\eta, s_{k}\left(\eta, x+1\right) - 1 \right) = k$. Hence we have  
                $\eta^{\uparrow}_{k+1}\left(s_{k}\left(\eta, x+1\right)\right) = 1$ and \eqref{equi:up} is proved. 
            \end{proof}
            
            \begin{proof}[Proof of Proposition \ref{prop:seat_semig}]
            
                We use induction on the space variable $x \in \N$. First we consider the case $x = 1$. Observe that in this case either $\eta^{\uparrow}_{k+1}\left( s_{k}\left(\eta,1\right) \right) = 1$ or $r\left(\eta, s_{k}\left(\eta,1\right) \right) = 1$ holds. On the other hand, we have 
                    \begin{align}
                        \mathcal{W}_{1}\left( \Psi_{k}\left( \eta \right) , 1 \right) &=   \Psi_{k}\left( \eta \right)\left( 1 \right) = \eta\left( s_{k}\left({\eta},1\right) \right).
                    \end{align}
                Hence, \eqref{eq:seat_semig} holds for $x = 1$. 
                
                From now on we assume that up to $x$, \eqref{eq:seat_semig} holds for any $k, \ell \in \N$. Then, from \eqref{eq:car_n_seat} and Lemma \ref{lem:sums}, for any $k, \ell \in \N$ and $0 \le y \le x$ we obtain
                    \begin{align}
                        \mathcal{W}_{\ell}\left( \Psi_{k}\left( \eta \right) , y \right) &= \sum_{z = 1}^{y} \left( \Psi_{k}\left( \eta \right)^{\uparrow}_{\ell}\left( z \right) - \Psi_{k}\left( \eta \right)^{\downarrow}_{\ell}\left( z \right) \right) & & \text{{(by \eqref{eq:car_n_seat})}} \\
                        &= \sum_{z = 1}^{y} \left( \eta^{\uparrow}_{k+\ell}\left( s_{k}\left({\eta},z\right) \right) - \eta^{\downarrow}_{k+\ell}\left( s_{k}\left({\eta},z\right) \right) \right) & & \text{{(by the induction hypothesis)}} \\
                        &= \sum_{z = 1}^{s_{k}\left({\eta},y\right)} \left( \eta^{\uparrow}_{k+\ell}\left( z \right) - \eta^{\downarrow}_{k+\ell}\left( z \right) \right), & & \text{{(by Lemma \ref{lem:sums})}} \label{eq:seat_semig_1}
                    \end{align}
                and thus from \eqref{eq:car_l_seat} we get
                    \begin{align}
                        W_{\ell}\left( \Psi_{k}\left( \eta \right) , y \right) &= \sum_{m = 1}^{\ell} \sum_{z = 1}^{y} \left( \Psi_{m}\left( \eta \right)^{\uparrow}_{\ell}\left( z \right) - \Psi_{m}\left( \eta \right)^{\downarrow}_{\ell}\left( z \right) \right) & & \text{{(by \eqref{eq:car_n_seat})}} \\
                        &= \sum_{m = 1}^{\ell} \sum_{z = 1}^{s_{k}\left({\eta},y\right)} \left( \eta^{\uparrow}_{k+m}\left( z \right) - \eta^{\downarrow}_{k+m}\left( z \right) \right) & & \text{{(by \eqref{eq:seat_semig_1})}} \\
                        &= W_{k + \ell}\left( \eta, s_{k}\left({\eta},y\right) \right) - W_{k}\left( \eta, s_{k}\left({\eta},y\right) \right) & & \text{{(by \eqref{eq:car_l_seat})}} \label{eq:seat_semig_2}.
                    \end{align}
                Therefore, from Lemmas \ref{lem:carrier_cap} and \ref{lem:equi}, for $\sigma = \uparrow$ and $\ell \ge 2$ we have
                    \begin{align}
                        &\Psi_{k}\left( \eta \right)^{\uparrow}_{\ell}\left( x + 1 \right) = 1 \ \\ &\text{if and only if} \ { \begin{dcases}
                            W_{\ell - 1}\left( \Psi_{k}\left( \eta \right) , x \right) = \ell - 1, \\
                            \mathcal{W}_{\ell}\left( \Psi_{k}\left( \eta \right) , x \right) = 0, \\
                            \Psi_{k}\left( \eta \right)(x+1) = 1,
                        \end{dcases}} \\
                        &\text{if and only if} \ \begin{dcases}
                            W_{\ell - 1}\left( \Psi_{k}\left( \eta \right) , x \right) = \ell - 1, \\
                            W_{\ell}\left( \Psi_{k}\left( \eta \right) , x \right) = \ell - 1, \\
                            \Psi_{k}\left( \eta \right)(x+1) = 1,
                        \end{dcases} & & \text{{(by \eqref{eq:car_l_seat})}} \\
                        &\text{if and only if} \ \begin{dcases}
                            W_{k + \ell - 1}\left( \eta , s_{k}\left({\eta},x\right) \right) - W_{k}\left( \eta , s_{k}\left({\eta},x\right) \right) = \ell - 1, \\
                            W_{k + \ell}\left( \eta , s_{k}\left({\eta},x\right) \right) - W_{k}\left( \eta , s_{k}\left({\eta},x\right) \right) = \ell - 1, \\
                            \eta\left(s_{k}\left({\eta},x + 1\right)\right) = 1,
                        \end{dcases} & & \text{{(by \eqref{eq:seat_semig_2})}} \\
                        &\text{if and only if} \ \begin{dcases}
                            W_{k + \ell - 1}\left( \eta , s_{k}\left({\eta},x\right) \right) = k + \ell - 1, \\
                            W_{k + \ell}\left( \eta , s_{k}\left({\eta},x\right) \right) = k + \ell - 1, \\
                            \eta\left(s_{k}\left({\eta},x + 1\right)\right) = 1,
                        \end{dcases} & & \text{{(by Lemma \ref{lem:carrier_cap})}} \\
                        &\text{if and only if} \ \eta^{\uparrow}_{k+\ell}\left( s_{k}\left({\eta},x + 1\right) \right) = 1. & & \text{{(by Lemma \ref{lem:equi})}}
                    \end{align}
                For the case $\ell = 1$, we obtain
                    \begin{align}
                        &\Psi_{k}\left( \eta \right)^{\uparrow}_{1}\left( x + 1 \right) = 1 \ \\ 
                        &\text{if and only if} \ \begin{dcases}
                            {W}_{1}\left( \Psi_{k}\left( \eta \right) , x \right) = 0, \\
                            \Psi_{k}\left( \eta \right)(x+1) = 1,
                        \end{dcases} \\
                        &{\text{if and only if} \ \begin{dcases}
                            W_{k+1}\left( \eta , s_{k}\left(\eta,x\right) \right) - W_{k}\left( \eta , s_{k}\left(\eta,x\right) \right) = 0, \\
                            \eta\left(s_{k}\left({\eta},x + 1\right)\right) = 1,
                        \end{dcases}} & & \text{{(by \eqref{eq:seat_semig_2})}} \\
                        &\text{if and only if} \ \begin{dcases}
                            \mathcal{W}_{k + 1}\left( \eta , s_{k}\left({\eta},x\right) \right) = 0, \\
                            \eta\left(s_{k}\left({\eta},x + 1\right)\right) = 1,
                        \end{dcases} & & \text{{(by \eqref{eq:car_l_seat})}}  \\
                        &\text{if and only if} \ \eta^{\uparrow}_{k+1}\left( s_{k}\left({\eta},x + 1\right) \right) = 1.
                    \end{align}
                The case $\sigma = \downarrow$ can be proved in a similar way. 

                {Finally, by \eqref{eq:seat_semig} and \eqref{eq:carrier_cap}, for any $k \in \N$ and $x \in \N$, we get 
                    \begin{align}
                        r\left(\eta,s_{k}\left(\eta, x \right) \right) + \sum_{\ell \in \N} \sum_{\sigma \in \{ \uparrow, \downarrow \}} \Psi_{k}\left(\eta\right)^{\sigma}_{\ell}\left(x \right) = 1.
                    \end{align} 
                Then by \eqref{def:record}, we have \eqref{eq:seat_semig_rec}.}
            \end{proof}
		
		    \begin{proof}[Proof of Proposition \ref{prop:semig}]
                
                We observe that from Proposition \ref{prop:seat_semig}, if $y^{{*}} = s_{k}(\Psi_{\ell}(\eta), x)$, for some $k \in \N \cup \{ \infty \}$, $\ell \in \N$ and $x \in \N$, then $s_{\ell}(\eta, y^{{*}}) = s_{k+\ell}(\eta, x)$. Hence, we obtain
                    \begin{align}
                        \Psi_{k}\left( \Psi_{\ell}\left( \eta \right) \right)\left( x \right) 
                        &= \Psi_{\ell}\left( \eta \right)\left(y^{{*}} \right) \\
                        &= \eta\left( s_{\ell}\left( \eta , y^{{*}} \right) \right) \\
                        &= \eta\left( s_{k + \ell}\left( \eta, x \right) \right) \\
                        &= \Psi_{k+\ell}\left(\eta\right)(x).
                    \end{align}
                
            \end{proof}
            \begin{proof}[Proof of Proposition \ref{prop:shift}]
                    
                From Proposition \ref{prop:seat_semig} and Lemma \ref{lem:sums}, for any $k, \ell \in \N$ and $i \in \Z_{\ge 0}$ we have 
                    \begin{align}
                        \zeta_{k}\left( \Psi_{\ell}\left(\eta\right) , i \right) 
                        &= \sum_{y = s_{k}\left( \Psi_{\ell}\left(\eta\right) , i \right) + 1}^{s_{k}\left( \Psi_{\ell}\left(\eta\right) , i +1 \right)} \left( \Psi_{\ell}\left(\eta\right)^{\uparrow}_{k}(y) - \Psi_{\ell}\left(\eta\right)^{\uparrow}_{k+1}(y) \right) \\
                        &= \left( \sum_{y = s_{k}\left( \Psi_{\ell}\left(\eta\right) , i \right)}^{s_{k}\left( \Psi_{\ell}\left(\eta\right) , i +1 \right)} \Psi_{\ell}\left(\eta\right)^{\uparrow}_{k}(y) \right) - \Psi_{\ell}\left(\eta\right)^{\uparrow}_{k+1}\left(s_{k}\left( \Psi_{\ell}\left(\eta\right) , i +1 \right)\right) \\
                        &= \left( \sum_{y = s_{k}\left( \Psi_{\ell}\left(\eta\right) , i \right)}^{s_{k}\left( \Psi_{\ell}\left(\eta\right) , i +1 \right)} \eta^{\uparrow}_{k + \ell}\left(s_{\ell}\left( \eta , y \right)\right) \right) - \eta^{\uparrow}_{k + \ell + 1}\left(s_{\ell}\left( \eta , s_{k}\left( \Psi_{\ell}\left(\eta\right) , i +1 \right)\right)\right)  \\ 
                        &= \left( \sum_{y = s_{k + \ell}\left( \eta , i \right)}^{s_{k + \ell}\left( \eta , i+1 \right)} \eta^{\uparrow}_{k + \ell}\left(y\right) \right) - \eta^{\uparrow}_{k + \ell + 1}\left(s_{k + \ell}\left( \eta, i + 1 \right)\right) \\
                        &= \zeta_{k+\ell}\left(\eta, i\right).
                    \end{align}
            \end{proof}

    From Theorem \ref{thm:linear} and Proposition \ref{prop:semig}, we see that $\Psi_{k}$ and $T_{\ell}$ are both shift operators of different variables, but in general, they do not commute. We can obtain the following formulae only in the special cases $T_{\ell}, \ell = 1, \infty$.
                \begin{proposition}\label{prop:shift_com}
                    For any $\eta \in \Omega_{r}$ and $k \in \N$, we have
                        \begin{align}
                            \Psi_{k}\left( T_{1} \eta \right) &= T_{1} \Psi_{k}\left( \eta \right) \\
                            \Psi_{k}\left( T_{\infty} \eta \right) &= T_{\infty} \Psi_{k}\left(T_{k} \eta \right) \label{eq:shift_com}.
                        \end{align}
                \end{proposition}
                \begin{proof}[Proof of Proposition \ref{prop:shift_com}]
                    {By Theorem \ref{thm:linear} and  Proposition \ref{prop:shift}}, for any $k, \ell, m \in \N$ and $ i \in \Z_{\ge 0}$, we have
                        \begin{align}
                            \zeta_{m}\left( \Psi_{k}\left( T_{\ell} \eta \right), i \right) 
                            &= {\zeta_{m+k}\left(  T_{\ell} \eta , i \right)} & & \text{{(by Proposition \ref{prop:shift})}} \\
                            &= \zeta_{m + k}\left( \eta, i - \left(m + k\right)\wedge \ell \right), & & \text{{(by Theorem \ref{thm:linear})}}
                        \end{align}
                    {and} 
                        \begin{align}
                            \zeta_{m}\left(  T_{\ell} \Psi_{k}\left( \eta \right), i \right) 
                            &= {\zeta_{m}\left( \Psi_{k}\left( \eta \right), i - m \wedge \ell \right)} & & \text{{(by Theorem \ref{thm:linear})}} \\
                            &= \zeta_{m + k}\left( \eta, i - m \wedge \ell \right). & & \text{{(by Proposition \ref{prop:shift})}}
                        \end{align}
                    Then, we see that
                        \begin{align}
                            \zeta_{m}\left( \Psi_{k}\left( T_{\ell} \eta \right), i \right) &= \zeta_{m + k}\left( \eta, i - \left(m + k\right)\wedge \ell \right) & & \\
                            &= \zeta_{m + k}\left( T_{\left(m + k\right)\wedge \ell - m \wedge \ell }\eta, i - m \wedge \ell \right) & & \text{{(by Theorem \ref{thm:linear})}} \\
                            &= {\zeta_{m}\left( \Psi_{k}\left( T_{\left(m + k\right)\wedge \ell - m \wedge \ell }\eta \right), i - m \wedge \ell \right)}  & & \text{{(by Proposition \ref{prop:shift})}} \\
                            &= \zeta_{m}\left(T_{\ell} \Psi_{k}\left( T_{\left(m + k\right)\wedge \ell - m \wedge \ell }\eta \right), i \right), & & \text{{(by Theorem \ref{thm:linear})}} 
                        \end{align}
                    and that when $\ell = 1, \infty$, the quantity $\left(m + k\right)\wedge \ell - m \wedge \ell$ does not depend on $m \in \N$. For such cases, from the bijectivity of $\zeta$, we have \eqref{eq:shift_com}.
                \end{proof}
            
\section{10-elimination}\label{sec:10}

 {The 10-elimination was introduced in \cite{MIT} to solve the initial
value problem of the periodic BBS, and its relation to rigged
configurations was studied in \cite{KS}. In this section, we recall a
half-line version of this procedure. See also \cite{KNTW} for a related
formulation of the 10-elimination on finite-particle configurations on
the half-line. } Throughout this section, 
{the configuration space is 
    \begin{align}
        \Omega_{< \infty} := \left\{\eta \in \{0,1\}^{\Z_{\ge 0}} \ ; \ \eta(0) = 0, \ \sum_{x \in \N} \eta(x) < \infty \right\}.
    \end{align}}
We will often identify $\eta \in \{0,1\}^{\Z_{\ge 0}}$ with the one-sided sequence $\eta = \eta(0) \eta(1) \eta(2) \cdots$. 
    
    
    \subsection{Definition of the \texorpdfstring{$10$}{10}-elimination}

{We first define the 10-elimination. We say that an interval
$[x+1,x+2j]$, $j\in\mathbb N$, is a maximal 10-block of $\eta$ if for any $1\le y \le j$, 
\begin{align}
    \eta(x + 2y-1) = 1, \quad \eta(x + 2y) = 0
\end{align}
and
\begin{align}
    (\eta(x),\eta(x+1)) \neq (1,0), \quad (\eta(x+2j),\eta(x+2j+1) ) \neq (1,0).
\end{align}
In other words, a maximal 10-block is a maximal consecutive string of
10-pairs. If $[x+1,x+2j]$, $j\in\mathbb N$, is a maximal 10-block of $\eta$, then we will write 
    \begin{align}
        \eta(x+1)\eta(x+2)\dots \eta(x+2j) = (10)^{j}.
    \end{align}
These maximal 10-blocks are disjoint and uniquely determined
by $\eta$. Equivalently, $\eta$ can be uniquely written in the form
\begin{align}
    \eta
  =
  \eta'(0)(10)^{j(0)}\eta'(1)(10)^{j(1)}\eta'(2)(10)^{j(2)}\cdots,
\end{align}
where $\eta'(0) := 0$, $\eta'(x)_
\in\{0,1\}$, $j(x) = j(\eta,x) \in\mathbb Z_{\ge0}$ for $x\in\mathbb Z_{\ge0}$, and the blocks
$(10)^{j(x)}$ with $j(x)\ge1$ are precisely the maximal 10-blocks of
$\eta$. Here $j(x)=0$ means that no 10-pair is inserted between $\eta'(x)$ and
$\eta'(x+1)$. We note that if $\eta'(x)\eta'(x+1)=10$, then $j(x)$ must satisfy $j(x) \ge 1$.

The 10-elimination is a map $\Phi_{1} : \Omega_{<\infty} \to \Omega_{<\infty}$, $\eta \mapsto \Phi_1(\eta)$, and 10-eliminated configuration $\Phi_1(\eta)$ is obtained by
removing all maximal 10-blocks, or equivalently by removing all neighboring
10-pairs simultaneously. In the above notation, this is given by
\begin{align}
    \Phi_1(\eta) = \eta'(0)\eta'(1)\eta'(2) \cdots.
\end{align}
Then, the 10-rigging  $J^{10}_{1}(\eta) = \left( J^{10}_{1}(\eta,x) \right)_{x \in \Z_{\ge 0}}$, $J^{10}_{1}(\eta,x) \in \Z_{\ge 0}$ is defined as
\begin{align}
     J^{10}_{1}(\eta,x) :=
  \begin{dcases}
    j(\eta,x) \ & \ \text{if } j(\eta,x) \ge 1, \ (\eta'(x),\eta'(x+1)) \in\{(0,0),(0,1),(1,1)\},\\
    j(\eta,x) - 1 \ & \ \text{if } j(\eta,x) \ge 2, \  (\eta'(x),\eta'(x+1)) =(1,0), \\
    0 \ & \ \text{otherwise.}
  \end{dcases}
\end{align}
Figure \ref{fig:exof10} illustrates the decomposition of $\eta$, the
10-elimination, and the corresponding 10-rigging.}

\begin{figure}[t]
\centering
\begin{tikzpicture}[
  x=0.55cm,
  y=0.60cm,
  every node/.style={font=\scriptsize}
]

\node[left] at (-1,0) {$x$};
\foreach \x in {0,...,18}{
  \node at (\x,0) {\x};
}

\node[left] at (-1,-1) {$\eta(x)$};
\foreach \x/\b in {
0/0,1/1,2/1,3/0,4/0,5/1,6/1,7/1,8/0,
9/1,10/0,11/1,12/1,13/0,14/0,15/0,16/1,17/0,18/0}{
  \node at (\x,-1) {$\b$};
}
\node at (19,-1) {$\cdots$};

\draw[red!70!black, rounded corners] (1.55,-1.35) rectangle (3.45,-0.65);
\node[red!70!black] at (2.5,-1.85) {$(10)^1$};

\draw[red!70!black, rounded corners] (6.55,-1.35) rectangle (10.45,-0.65);
\node[red!70!black] at (8.5,-1.85) {$(10)^2$};

\draw[red!70!black, rounded corners] (11.55,-1.35) rectangle (13.45,-0.65);
\node[red!70!black] at (12.5,-1.85) {$(10)^1$};

\draw[red!70!black, rounded corners] (15.55,-1.35) rectangle (17.45,-0.65);
\node[red!70!black] at (16.5,-1.85) {$(10)^1$};

\foreach \x in {0,1,4,5,6,11,14,15,18}{
  \draw[blue!70!black, rounded corners] (\x-0.32,-1.33) rectangle (\x+0.32,-0.67);
}

\begin{scope}[shift={(0,-3)}, x=1.10cm, y=0.65cm]

\node[left] at (-1,0) {$x$};
\foreach \x in {0,...,8}{
  \node at (\x,0) {$\x$};
}

\node[left] at (-1,-1) {$\eta'(x)$};
\foreach \x in {0,...,8}{
  \node at (\x,-1) {$\eta'(\x)$};
}
\node at (9,-1) {$\cdots$};

\node[left] at (-1,-2) {$\Phi_1(\eta)(x)$};
\foreach \x/\b in {0/0,1/1,2/0,3/1,4/1,5/1,6/0,7/0,8/0}{
  \node at (\x,-2) {$\b$};
}
\node at (9,-2) {$\cdots$};

\node[left] at (-1,-3) {$J^{10}_1(\eta,x)$};
\node at (0,-3) {$0$};
\node at (1,-3) {$0$};
\node at (2,-3) {$0$};
\node at (3,-3) {$0$};
\node at (4,-3) {$2$};
\node at (5,-3) {$0$};
\node at (6,-3) {$0$};
\node at (7,-3) {$1$};
\node at (8,-3) {$0$};
\node at (9,-3) {$\cdots$};

\end{scope}

\end{tikzpicture}
\caption{One step of the 10-elimination and the corresponding 10-rigging.
The red boxes indicate the maximal 10-blocks removed by $\Phi_1$, and the
blue boxes indicate the remaining symbols. These remaining symbols are
denoted by $X(i)$ and form $\Phi_1(\eta)$. In this example, the only
non-zero riggings are $J^{10}_1(\eta,4)=2$ and
$J^{10}_1(\eta,7)=1$.}
\label{fig:exof10}
\end{figure}

{We then repeat the same construction. First, we apply the above procedure
to $\eta$ and obtain $\Phi_1(\eta)$ and $J^{10}_1(\eta)$. Next, we regard
$\Phi_1(\eta)$ as a new configuration and apply the same procedure to it.
Then we define 
    \begin{align}
        \Phi_2(\eta):=\Phi_1(\Phi_1(\eta)),
  \quad
  J^{10}_2(\eta):=J^{10}_1(\Phi_1(\eta)).
    \end{align}
Continuing in this way, once $\Phi_k(\eta)$ has been constructed, we
apply the 10-elimination to $\Phi_k(\eta)$ and define
    \begin{align}
        \Phi_{k+1}(\eta):=\Phi_1(\Phi_k(\eta)),
  \qquad
  J^{10}_{k+1}(\eta):=J^{10}_1(\Phi_k(\eta)).
    \end{align}
Thus we obtain the sequence 
$(\Phi_k(\eta))_{k\in\mathbb N}$ and the corresponding 10-riggings
$(J^{10}_k(\eta))_{k\in\mathbb N}$. }
Note that since $\eta \in \Omega_{< \infty}$, $\Phi_{k}\left(\eta\right) = 000\dots$ for sufficiently large $k$.

        It is known that BBS($\ell$) can be linearized via the $10$-elimination as follows :  
            \begin{theorem}[Theorem 2 in \cite{MIT}, Theorem 22 in \cite{KNTW}]
                Suppose that $\eta \in \Omega_{<\infty}$. Then for any $k \in \N$, $\ell \in \N \cup \{\infty\}$ and $i \in \Z_{\ge 0}$, we have
                    \begin{align}
                        J^{10}_{k}\left(T_{\ell}\eta, i\right) = J^{10}_{k}\left(\eta, i - k \wedge \ell \right)
                    \end{align}
                with convention $J^{10}_{k}\left(\eta, i\right) = {0}$ for $i < 0$.
            \end{theorem}
        
    \subsection{On the relation to the seat number configuration}
    
        In this subsection, we will show the following. 
            \begin{theorem}\label{thm:10=seat}
                Suppose that $\eta \in \Omega_{< \infty}$. Then, for any $k \in \N$ and $i \in \Z_{\ge 0}$, we have
                    \begin{align}\label{eq:10=seat}
                        \Phi_{k}\left(\eta\right) = \Psi_{k}\left(\eta\right), \quad  J^{10}_{k}(\eta,i) = \zeta_{k}\left(\eta, i\right).
                    \end{align}
            \end{theorem}

        \begin{proof}[Proof of Theorem \ref{thm:10=seat}]
            First we observe that it is sufficient to show \eqref{eq:10=seat} for $k = 1$. Actually, if \eqref{eq:10=seat} holds for $k = 1$, then from Propositions \ref{prop:semig} and \ref{prop:shift}, we have
                \begin{align}
                    \Phi_{2}\left(\eta\right) = \Phi_{1}\left(\Phi_{1}\left(\eta\right)\right) = \Phi_{1}\left(\Psi_{1}\left(\eta\right)\right) = \Psi_{1}\left(\Psi_{1}\left(\eta\right)\right) = \Psi_{2}\left(\eta\right), 
                 \end{align}
            and 
                \begin{align}
                     J^{10}_{2}(\eta,i)  =  J^{10}_{1}(\Phi_{1}\left(\eta\right),i) =  J^{10}_{1}(\Psi_{1}\left(\eta\right),i)  = \zeta_{1}\left(\Psi_{1}\left(\eta\right),  i\right) = \zeta_{2}\left(\eta,  i\right).
                \end{align}
            Hence, by repeating the above argument, \eqref{eq:10=seat} can be proved for any $k \in \N$.

            From now on we will show \eqref{eq:10=seat}. Fix ${\eta} \in \Omega_{< \infty}$. Observe that $\eta$ can be decomposed as follows: 
                    \begin{align}\label{eq:decom_10}
	                    \eta = 0^{\otimes m_0} 1^{\otimes n_1} ... 0^{\otimes m_{l}} 1^{\otimes n_l} 0^{\otimes m_{l}}, \quad m_{l} = \infty,
	                \end{align}
            where $z^{\otimes r}, ~ z=0,1, ~ r \in \Z_{\ge 0}\cup\{ \infty \}$ is the $r$ successive $z$'s  and $n_1, ..., n_l \in \N$ satisfy 
					\begin{align}
						                    \sum_{i=1}^{l} n_i = \sum_{x \in \Z_{\ge 0}} \eta(x).
					\end{align}
			Note that $z^{\otimes 0} = \emptyset$. By using the above decomposition, we can see that the $10$-elimination removes the rightmost $``1"$  (resp. the leftmost $``0"$) of all consecutive $1$'s (resp. $0$'s) except for the origin, and thus $\Phi_{1}(\eta)$ is expressed as
				\begin{align}
					\Phi_{1}(\eta) = 0^{\otimes m_0} 1^{\otimes n_1 - 1} 0^{\otimes m_1 - 1} ... 0^{\otimes m_{l} - 1} 1^{\otimes n_l - 1} 0^{\otimes \infty}.
				\end{align}
			On the other hand, $\Psi_{1}(\eta)$ skips the leftmost $``1"$(resp. $``0"$) of all consecutive $1$'s (resp. $0$'s) except for the origin. Therefore we obtain
    			\begin{align}
    				\Psi_{1}(\eta) &= 0^{\otimes m_0} 1^{\otimes n_1 - 1} 0^{\otimes m_1 - 1} ... 0^{\otimes m_{l} - 1} 1^{\otimes n_l - 1} 0^{\otimes \infty} \\
    				    &= \Phi_{1}(\eta).
    			\end{align}
    		Finally we will show $ J^{10}_{1}(\eta,i)  = \zeta_{1}\left(\eta, i\right)$. 
    		Assume that $ J^{10}_{1}(\eta,i) = m, m \in \Z_{\ge 0}$. Then, by considering the meanings of $\Phi_{1}(\eta)$ and $\Psi_{1}(\eta)$ via the decomposition \eqref{eq:decom_10} again, we have the following.
    		    \begin{itemize}
    		        \item For the case $\eta\left( s_{{1}}\left({\eta},i\right) \right) = \eta\left( s_{{1}}\left({\eta},i + 1\right) \right) = 1$. Then, we obtain
    		            \begin{align}
    		                \eta\left( s_{{1}}\left({\eta},i\right) - 1  \right)\eta\left( s_{{1}}\left({\eta},i\right)   \right) \dots \eta\left( s_{{1}}\left({\eta},i + 1\right) - 1 \right)\eta\left( s_{{1}}\left({\eta},i + 1\right) \right) = 1(10)^{m}1,
    		            \end{align}
    		        where we use the convention $X(10)^{0}Y = XY$ for $X,Y \in \{0,1 \}$. Since all $``0"$s in
    		        $\eta\left( s_{{1}}\left(i\right) + 1  \right) \dots \eta\left( s_{{1}}\left(i + 1\right) \right)$ are $(1,\downarrow)$-seats, from \eqref{eq:down_zeta} we have
    		            \begin{align}
    		                \zeta_{1}\left(\eta, i\right) = \sum_{y = s_{1}({\eta},i) + 1}^{s_{{1}}({\eta},i+1)} \eta^{\downarrow}_{1}(y) -  \eta^{\downarrow}_{2}(y) = m.
    		            \end{align}
    		            
                    \item For the case $\eta\left( s_{{1}}\left({\eta},i\right) \right) = 0,  \eta\left( s_{{1}}\left({\eta},i + 1\right) \right) = 1$. Then, we obtain
    		            \begin{align}
    		                  \eta\left( s_{{1}}\left({\eta},i\right)   \right) \dots \eta\left( s_{{1}}\left({\eta},i + 1\right) - 1 \right)\eta\left( s_{{1}}\left({\eta},i + 1\right) \right) = 0(10)^{m}1.
    		            \end{align}
                    Since all $``0"$s in $\eta\left( s_{{1}}\left({\eta},i\right) + 1  \right) \dots \eta\left( s_{{1}}\left({\eta},i + 1\right) \right)$ are $(1,\downarrow)$-seats, from \eqref{eq:down_zeta} we have
    		            \begin{align}
    		                \zeta_{1}\left(\eta, i\right) = \sum_{y = s_{1}({\eta},i) + 1}^{s_{{1}}({\eta},i+1)} \eta^{\downarrow}_{1}(y) -  \eta^{\downarrow}_{2}(y) = m.
    		            \end{align}
                    
                    \item For the case $\eta\left( s_{{1}}\left({\eta},i\right) \right) = 0,  \eta\left( s_{{1}}\left({\eta},i + 1\right) \right) = 0$. Then, we obtain
    		            \begin{align}
    		                  \eta\left( s_{{1}}\left({\eta},i\right)   \right) \dots \eta\left( s_{{1}}\left({\eta},i + 1\right) - 1 \right)\eta\left( s_{{1}}\left({\eta},i + 1\right) \right) = 0(10)^{m}0.
    		            \end{align}
                    Since all $``1"$s in $\eta\left( s_{{1}}\left({\eta},i\right) + 1  \right) \dots \eta\left( s_{{1}}\left({\eta},i + 1\right)\right)$ are $(1,\uparrow)$-seats, we have
    		            \begin{align}
    		                \zeta_{1}\left(\eta, i\right) = \sum_{y = s_{1}({\eta},i) + 1}^{s_{{1}}({\eta},i+1)} \eta^{\uparrow}_{1}(y) -  \eta^{\uparrow}_{2}(y) = m.
    		            \end{align}
                    
                    \item For the case $\eta\left( s_{{1}}\left({\eta},i\right) \right) = 1,  \eta\left( s_{{1}}\left({\eta},i + 1\right) \right) = 0$. Then, we obtain
    		            \begin{align}
    		                  \eta\left( s_{{1}}\left({\eta},i\right)   \right) \dots \eta\left( s_{{1}}\left({\eta},i + 1\right) - 1 \right)\eta\left( s_{{1}}\left({\eta},i + 1\right) \right) = 1(10)^{m+1}0.
    		            \end{align}
                    Since all $``0"$s in $\eta\left( s_{{1}}\left({\eta},i\right) + 1  \right) \dots \eta\left( s_{{1}}\left({\eta},i + 1\right) - 1 \right)$ are $(1,\downarrow)$-seats and $s_{{1}}\left({\eta},i + 1\right)$ is a $\left(2,\downarrow\right)$-seat, from \eqref{eq:down_zeta} we have
    		            \begin{align}
    		                \zeta_{1}\left(\eta, i\right) = \sum_{y = s_{1}({\eta},i) + 1}^{s_{{1}}({\eta},i+1)} \eta^{\downarrow}_{1}(y) -  \eta^{\downarrow}_{2}(y) = m.
    		            \end{align} 
                    
    		    \end{itemize}
    		From the above, we have \eqref{eq:10=seat} for $k = 1$. 
        \end{proof}
    
    \begin{remark}
        We note that both the $1$-skip map and the $10$-elimination remove the same $0$'s, but they may remove different $1$'s, see Figure \ref{fig:diff_1_10}. In other words, the $10$-elimination may remove $1$ located at $x$ such that $\eta^{\uparrow}_{\ell}\left(x \right) = 1$ for some $\ell \ge 2$. From this point of view and Remark \ref{rem:seat_sol}, the $1$-skip map seems to be the more natural operation for solitons. 
    \end{remark}

        \begin{figure}[t]
\centering
\begin{tikzpicture}[scale=0.7, every node/.style={transform shape},
  x=0.88cm,
  y=0.90cm,
  digit/.style={font=\normalsize},
  lab/.style={font=\normalsize},
  title/.style={font=\bfseries\large},
  line/.style={black!70, line width=0.4pt}
]

\newcommand{\keepdigit}[3]{%
  \node[digit] at (#1,#2) {$#3$};
}
\newcommand{\removedone}[2]{%
  \node[digit, blue!70!black] at (#1,#2) {$1$};
  \draw[red!80!black, line width=0.5pt] (#1-0.17,#2-0.20) -- (#1+0.17,#2+0.20);
  \draw[red!80!black, line width=0.5pt] (#1-0.17,#2+0.20) -- (#1+0.17,#2-0.20);
}
\newcommand{\removedzero}[2]{%
  \node[digit, gray!75!black] at (#1,#2) {$0$};
  \draw[gray!75!black, line width=0.5pt] (#1-0.17,#2-0.20) -- (#1+0.17,#2+0.20);
  \draw[gray!75!black, line width=0.5pt] (#1-0.17,#2+0.20) -- (#1+0.17,#2-0.20);
}

\node[lab] at (-1.9,0.0) {$x$};
\foreach \x in {0,...,18}{
  \node[digit] at (\x,0.0) {\x};
}
\node[digit] at (19.1,0.0) {$\cdots$};
\draw[line] (-2.3,-0.45) -- (19.5,-0.45);

\draw[rounded corners=5pt, red!75!black, line width=0.6pt]
  (-2.55,-1.20) rectangle (19.55,-5.05);

\node[
  draw=red!75!black,
  rounded corners=4pt,
  line width=0.6pt,
  fill=white,
  inner xsep=8pt,
  inner ysep=5pt,
  anchor=west
] at (-2.45,-1.20) {\color{red!75!black}\title 10-elimination};

\node[lab] at (-1.65,-2.10) {$\eta(x)$};
\node[lab] at (-1.65,-4.20) {$\Phi_{1}(\eta)$};

\draw[line] (-2.25,-2.55) -- (19.35,-2.55);
\draw[line] (-2.25,-3.65) -- (19.35,-3.65);

\foreach \x/\a in {
0/0,1/1,2/1,3/0,4/0,5/1,6/1,7/1,8/0,
9/1,10/0,11/1,12/1,13/0,14/0,15/0,16/1,17/0,18/0}{
  \keepdigit{\x}{-2.10}{\a}
}
\node[digit] at (19.1,-2.10) {$\cdots$};

\foreach \x/\a in {0/0,1/1,4/0,5/1,6/1,11/1,14/0,15/0,18/0}{
  \keepdigit{\x}{-3.15}{\a}
}
\foreach \x in {2,7,9,12,16}{
  \removedone{\x}{-3.15}
}
\foreach \x in {3,8,10,13,17}{
  \removedzero{\x}{-3.15}
}
\node[digit] at (19.1,-3.15) {$\cdots$};

\foreach \i/\a in {0/0,1/1,2/0,3/1,4/1,5/1,6/0,7/0,8/0}{
  \keepdigit{\i}{-4.20}{\a}
}
\foreach \i in {9,...,18}{
  \keepdigit{\i}{-4.20}{0}
}
\node[digit] at (19.1,-4.20) {$\cdots$};

\draw[rounded corners=5pt, blue!75!black, line width=0.6pt]
  (-2.55,-6.05) rectangle (19.55,-9.90);

\node[
  draw=blue!75!black,
  rounded corners=4pt,
  line width=0.6pt,
  fill=white,
  inner xsep=8pt,
  inner ysep=5pt,
  anchor=west
] at (-2.45,-6.05) {\color{blue!75!black}\title 1-skip map};

\node[lab] at (-1.65,-6.95) {$\eta(x)$};
\node[lab] at (-1.65,-9.05) {$\Psi_{1}(\eta)$};

\draw[line] (-2.25,-7.40) -- (19.35,-7.40);
\draw[line] (-2.25,-8.50) -- (19.35,-8.50);

\foreach \x/\a in {
0/0,1/1,2/1,3/0,4/0,5/1,6/1,7/1,8/0,
9/1,10/0,11/1,12/1,13/0,14/0,15/0,16/1,17/0,18/0}{
  \keepdigit{\x}{-6.95}{\a}
}
\node[digit] at (19.1,-6.95) {$\cdots$};

\foreach \x/\a in {0/0,2/1,4/0,6/1,7/1,12/1,14/0,15/0,18/0}{
  \keepdigit{\x}{-8.00}{\a}
}
\foreach \x in {1,5,9,11,16}{
  \removedone{\x}{-8.00}
}
\foreach \x in {3,8,10,13,17}{
  \removedzero{\x}{-8.00}
}
\node[digit] at (19.1,-8.00) {$\cdots$};

\foreach \i/\a in {0/0,1/1,2/0,3/1,4/1,5/1,6/0,7/0,8/0}{
  \keepdigit{\i}{-9.05}{\a}
}
\foreach \i in {9,...,18}{
  \keepdigit{\i}{-9.05}{0}
}
\node[digit] at (19.1,-9.05) {$\cdots$};





\end{tikzpicture}
\caption{Difference between the $1$-skip map and the $10$-elimination. 
For the above configuration, the removed $0$'s are the same, while the
removed $1$'s are different. However, we have
$\Phi_{1}(\eta)=\Psi_{1}(\eta)$.}
\label{fig:diff_1_10}
\end{figure}

\section{BBS on the whole-line}\label{sec:line}

    {In this section, we extend the seat number configuration to the BBS on the
whole-line. We also show that the results obtained in the previous sections
have natural counterparts in this setting. This extension is needed for the
application in the next section, where we consider the BBS with random initial
configurations on $\Z$. }

Since many notions are natural extensions of those introduced in Section \ref{sec:seat}, the same symbols will be used. We note that the relationships between the seat number configuration and solitons in the half-line case hold for the whole-line case as well. 
    
	Throughout this section, $\Omega$ is given by
	    \begin{align}
	        \Omega = \left\{ \eta \in \{0,1 \}^{\Z} \ ; \ {\lim_{x \to  \infty}} \frac{1}{x} \sum_{y = -x}^{0} \eta(y) < \frac{1}{2} \right\}.
	    \end{align}
    Since we mainly consider the BBS$(\infty)$ in the rest of this paper, we
write $T:=T_\infty$ for simplicity of notation. {It was shown in \cite[Section 2.7]{CKST} that the time evolution $T$ is well-defined on
$\Omega$. We note that, compared with the reversible setting considered in
\cite{CKST}, we impose the density condition only on the left-hand side.
This is sufficient for our purpose, since in this paper we only consider the
forward time evolution $T$, and do not consider the inverse dynamics
$T^{-1}$.}
    
	We also note that since $\{0,1 \}^{\N}$ can be considered as a subset of $\Omega$ {by the identification,}
	    \begin{align}
	        \{0,1 \}^{\N} \cong \left\{ \eta \in \Omega \ ; \ \eta(x) = 0 \  \text{for any } x \le 0 \right\}.
	    \end{align}
	{In this sense, the half-line case is included in the whole-line case described
below.}

	\subsection{Seat number configuration on the whole-line}
	    
	    First, we introduce the carrier with capacity $\ell$ to define the one-step time evolution of the BBS($\ell$) on the whole-line. For any $\eta \in \Omega$, we define $s_{\infty}(\eta,x)$ recursively as follows : 
	        \begin{align}\label{def:s_k_1}
	            s_{\infty}(\eta,0) &:= \max\left\{ x \le 0 \ ; \ \sum_{y = z}^{x}\left(2\eta(y) - 1\right) < 0 \  \text{for any } z\le x \right\},
            \end{align}
        and 
            \begin{align}
	            s_{\infty}(\eta,i) &:= \min\left\{ x > s_{\infty}(\eta,i - 1) \ ; \ \sum_{y = z}^{x}\left(2\eta(y) - 1\right) < 0 \  \text{for any } z \le x \right\} \label{def:s_k_2}, \\
	            s_{\infty}(\eta,-i) &:= \max\left\{ x < s_{\infty}(\eta,-i + 1) \ ; \ \sum_{y = z}^{x}\left(2\eta(y) - 1\right) < 0 \  \text{for any } z \le x \right\}\label{def:s_k_3}
	        \end{align}
        for any $i \in \N$, with convention $\min \emptyset = \infty$ and $\max \emptyset = - \infty$. {We call a site \(x\in\mathbb{Z}\) a record of \(\eta\) if
\(x=s_\infty(\eta,i)\) for some \(i\in\mathbb{Z}\).
Thus the sequence \((s_\infty(\eta,i))_{i\in\mathbb{Z}}\) enumerates the records
of \(\eta\) from left to right.}
        We note that from Remark \ref{rem:record}, the definition of {a record} is consistent with the half-line case. 
        Observe that if $\eta \in \Omega$, then $s_{\infty}(\eta,-i) > - \infty$ for any $i \in \Z_{\ge 0}$. Then, it is not difficult to check that by changing the starting point from $0$ to $s_{\infty}(\eta,-i)$, $W_{\ell}\left( \eta^{(i)}, \cdot \right)$ can be defined on $[s_{\infty}(\eta,-i), \infty) \cap \Z$ by using the same construction described in Introduction, where
            \begin{align}\label{def:eta_i}
                \eta^{(i)}(x) := 
                \begin{dcases}
                    \eta(x) \ & \ x \ge s_{\infty}(\eta,-i), \\
                    0 \ & \ \text{otherwise.}
                \end{dcases}
            \end{align}
        Then, for any $i, j \in \N$, $i \le j$, we see that
            \begin{align}\label{eq:cut_s}
                s_{\infty}(\eta^{(j)},-i) = s_{\infty}(\eta,-i).
            \end{align}
        As a result, for any $i \in \Z_{\ge 0}$, $W_{\ell}\left( \eta^{(i)}, \cdot \right)$ and $W_{\ell}\left( \eta^{(i+1)}, \cdot \right)$ are consistent, i.e., for any $x \ge s_{\infty}(\eta,- i)$, 
            \begin{align}
                W_{\ell}\left( \eta^{(i)}, x \right) = W_{\ell}\left( \eta^{(i+1)}, x \right). 
            \end{align}
        Hence, $W_{\ell}\left( \eta, x \right) := \lim_{i \to \infty} W_{\ell}\left( \eta^{(i)}, x \right)$ is well-defined for any $x \in \Z$. {We note that by the construction of the carrier for the whole-line case, a record is a site at which the carrier with infinite capacity
passes through while staying at level $0$, that is, $x \in \Z$ is a record in $\eta$ if and only if $W_{\infty}\left(\eta,x-1\right) = W_{\infty}\left(\eta,x\right) = 0$. This characterization of records is also consistent with the half-line case, see Remark \ref{rem:record}.} 

        Similarly, for any $k \in \N$ and $i \in \Z_{\ge 0}$, $\mathcal{W}_{k}\left(\eta^{(i)}, \cdot \right) : [s_{\infty}(\eta,-i), \infty) \cap \Z \to \{0,1\}$ can be defined by changing the starting point from $0$ to $s_{\infty}(\eta,-i)$, and one can check the consistency of $\mathcal{W}_{k}\left(\eta^{(i)}, x\right)$ and $\mathcal{W}_{k}\left(\eta^{(i+1)}, x \right)$ for any $x \ge s_{\infty}(\eta, -i)$. 
        Hence, $\mathcal{W}_{k}\left( \eta, x \right) := \lim_{i \to \infty} \mathcal{W}_{k}\left( \eta^{(i)}, x \right)$ is also well-defined for any $k \in \N$ and $x \in \Z$. In particular, from \eqref{eq:car_l_seat}, for any $\ell \in \N$ and $x \in \Z$, we have 
            \begin{align}\label{eq:car_l_seat_w}
                W_{\ell}\left(\eta, x \right) = \sum_{k = 1}^{\ell} \mathcal{W}_{{k}}\left(\eta, x\right). 
            \end{align}
        For later use, we state the above procedure as a lemma. 
            \begin{lemma}\label{lem:indep_c}
                Suppose that $\eta \in \Omega$. Then, for any $k \in \N$, $i \in \Z_{\ge 0}$ and $x \ge s_{\infty}(\eta,-i)$, we have
                    \begin{align}
                        \mathcal{W}_{k}\left( \eta^{(i)}, x \right) = \mathcal{W}_{k}\left(\eta^{(i+1)}, x \right).
                    \end{align}
                In other words, the value of $\mathcal{W}_{k}\left( \eta, \cdot \right)$ on $[s_{\infty}(\eta,-i), \infty)$ is independent of $\eta(x), x \le s_{\infty}(\eta,-i) - 1$. 
            \end{lemma}

        Now, we can define the seat number configuration $\eta^{\sigma}_{k} \in \{0,1\}^{\Z}$, $k \in \N$, $\sigma \in \{\uparrow, \downarrow\}$ by the same equations \eqref{def:seatup} and \eqref{def:seatdown}. Also, we define $r\left(\eta, \cdot\right) \in \{0,1\}^{\Z}$ as
            \begin{align}
                r\left(\eta, x\right) := 1 - \sum_{k \in \N} \sum_{\sigma \in \{\uparrow,\downarrow\}}  \eta^{\sigma}_{k}\left(x\right). 
            \end{align}
        Then, from \eqref{eq:car_l_seat}, \eqref{eq:car_n_seat} and Lemma \ref{lem:indep_c}, we have 
            \begin{align}\label{eq:car_n_seat_w}
                \mathcal{W}_{k}\left(\eta, x \right) = \sum_{y = s_{\infty}\left(\eta, -i \right) + 1}^{x} \left( \eta^{\uparrow}_{k}\left(y\right) - \eta^{\downarrow}_{k}\left(y\right) \right),
            \end{align}
        for any $k \in \N$,  $i \in \Z_{\ge 0}$ and $x \ge s_{\infty}\left(\eta, -i \right)$. 
        Also, thanks to Lemma \ref{lem:indep_c}, as in the half-line case (Lemma \ref{lem:1_half}), we obtain the following basic property of the seat number configuration. 
            \begin{lemma}\label{lem:1_whole}
                Suppose that $\eta \in \Omega$. Then, for any $k \in \N$, $i \le 0$ and $x \ge s_{\infty}\left(\eta, i\right)$, we have the following. 
                    \begin{enumerate}
                        \item $\eta^{\uparrow}_{k}\left(x\right) = 1$ implies $\sum_{y = s_{\infty}\left(\eta, i\right) + 1}^{x} \left(\eta^{\uparrow}_{\ell}\left(y\right) - \eta^{\downarrow}_{\ell}\left(y\right)\right) = 1$ for any $1 \le \ell \le k$.
                        \item $\eta^{\downarrow}_{k}\left(x\right) = 1$ implies $\sum_{y = s_{\infty}\left(\eta, i\right) + 1}^{x} \left(\eta^{\uparrow}_{\ell}\left(y\right) - \eta^{\downarrow}_{\ell}\left(y\right)\right) = 0$ for any $1 \le \ell \le k$.
                        \item $r\left(\eta, x\right) = 1$ implies $\sum_{y = s_{\infty}\left(\eta, i\right) + 1}^{x} \left(\eta^{\uparrow}_{\ell}\left(y\right) - \eta^{\downarrow}_{\ell}\left(y\right)\right) = 0$ for any $\ell \in \N$.
                    \end{enumerate}
            \end{lemma}
        For later use, we note the relationship between the seat number configuration and the notion of {\it slots} and corresponding slot configuration introduced in \cite{FNRW}. {First we remark that if \(|s_\infty(\eta,i)|<\infty\) for every \(i\in\mathbb{Z}\), 
        then the records \(s_\infty(\eta,i)\), \(i\in\mathbb{Z}\), divide the
configuration into finite intervals. As explained in Remark \ref{rem:zeta_sol},
the soliton decomposition can be performed separately in each interval between
two consecutive records. More precisely, for each \(i\in\mathbb{Z}\), we apply
the Takahashi--Satsuma algorithm to the finite word,
    \begin{align}
        \eta|_{[s_\infty(\eta,i),\,s_\infty(\eta,i+1)-1]} = \eta\left(s_\infty(\eta,i)\right) \dots \eta\left(s_\infty(\eta,i+1)-1\right).
    \end{align}
Combining the solitons obtained from all these intervals, we obtain the
solitons in \(\eta\). Then, for each $k \in \N$, we call a site $x \in \Z$ a $k$-slot if $x$ is a record, or if $x$ is a $\ell$-th head or  $\ell$-th tail of some $m$-soliton, where $k + 1 \le \ell \le m$.}        

        As in the half-line case, we can obtain the following. 
            \begin{proposition}\label{prop:seat_slot_w}
                Suppose that $\eta \in \Omega$ and $|s_{\infty}\left(\eta, i\right)| < \infty$ for any $i \in \Z$. Then, for any $k \in \N$ and $x \in \Z$, 
                \begin{align}
                    \eta^{\uparrow}_{k}\left(x\right) + \eta^{\downarrow}_{k}\left(x\right) = 1 \ \text{if and only if } x \text{ is a } (k-1)\text{-slot but not a } k\text{-slot}.
                \end{align}
            \end{proposition}
            \begin{proof}[Proof of Proposition \ref{prop:seat_slot_w}]
                This is a direct consequence of \cite[Proposition 2.3]{MSSS} and Lemma \ref{lem:indep_c}.
            \end{proof}
        
        Then, for any $k \in \N\cup\{\infty\}$, we define $\xi_{k}\left(\eta, \cdot \right) : \Z \to \Z $ and $s_{k}\left( \eta, \cdot \right) : \Z \to \Z \cup \{\infty\}$ as
            \begin{align}
                \xi_{k}\left( \eta, x \right) - \xi_{k}\left( \eta, x - 1 \right) &:= r\left(\eta,x\right) + \sum_{\ell \in \N} \sum_{\sigma \in \{\uparrow, \downarrow\}} \eta^{\sigma}_{k + \ell}\left( x \right)  \\
                &= 1 - \sum_{\ell = 1}^{k} \sum_{\sigma \in \{\uparrow, \downarrow\}} \eta^{\sigma}_{\ell}\left( x \right), \\
                \xi_{k}\left( \eta, s_{\infty}\left(\eta, 0\right) \right) &:= 0,
            \end{align}
        and 
            \begin{align}\label{def:s_k}
                s_{k}\left(\eta, x \right) := \min \left\{ y \in \Z \ ; \ \xi_{k}\left(\eta, y \right) = x \right\}.
            \end{align}
        Note that $s_{\infty}$ defined via \eqref{def:s_k} coincides with $s_{\infty}$ defined via \eqref{def:s_k_1}-\eqref{def:s_k_3}. 
        For later use, we note that thanks to Lemma \ref{lem:indep_c}, Lemma \ref{lem:sums} also holds for the whole-line case. 
            \begin{lemma}\label{lem:sums_w}
                Suppose that $\eta \in \Omega$. Then, for any $k,\ell \in \N$, $z, w \in \Z$, $z \le w$ and $\sigma \in \{{\uparrow,\downarrow}\}$, we have
                    \begin{align}
                        \sum_{y = z}^{w} \eta^{\sigma}_{k + \ell}\left( s_{k}\left(\eta, y \right) \right) = \sum_{y = s_{k}\left(\eta, z \right)}^{s_{k}\left(\eta, w \right)} \eta^{\sigma}_{k + \ell}\left( y \right).         
                    \end{align}
             \end{lemma}
        
        Finally, for any $k \in \N$ and $i \in \Z$, we define $\zeta_{k}(\eta, \cdot) : \Z \to \Z \cup \{\infty\}$ by \eqref{def:zeta}. Thanks to Proposition \ref{prop:seat_slot_w}, we see that our $\zeta$ coincides with the slot decomposition introduced in \cite{FNRW}, see also \cite[Proposition 2.3]{MSSS} for the half-line case. 
        
        The dynamics of the BBS $(\infty)$ can be linearized through $\zeta$, but an offset is required. {The following result is the whole-line version of \cite[Theorem 2.1]{MSSS}.}
            \begin{theorem}\label{thm:linear_whole}
                Suppose that $\eta \in \Omega$ and $s_{k}\left( \eta, i + 1 \right) < \infty$ for some $k \in \N$ and $i \in \Z$. Then, we have  
                    \begin{align}
                        \zeta_{k}\left(T\eta, i +  k + o_{k}\left(\eta\right) \right) = \zeta_{k}\left(\eta, i \right),
                    \end{align}
                where the offset $o_{k}(\eta)$ is given by 
                    \begin{align}
                        o_{k}\left(\eta\right) 
                        &:= s_{\infty}\left(\eta, 0\right) - s_{\infty}\left(T\eta, 0\right) + 2\sum_{y = s_{\infty}\left(\eta, 0\right)+ 1}^{0} \sum_{\ell = 1}^{k}  \eta^{\downarrow}_{\ell}\left(y\right) - 2\sum_{y = s_{\infty}\left(T\eta, 0\right)+1}^{0}\sum_{\ell = 1}^{k}  T\eta^{\uparrow}_{\ell}\left(y\right).
                    \end{align}
            \end{theorem}
            \begin{remark}\label{rem:offset_ex}
                {If $\eta(x)=0$ for all $x\leq 0$, then the offset in the above formula is
$0$. In this case there is no soliton on the left of the origin, and no soliton crosses the origin from the left. Hence the formula
coincides with the formula in the half-line case. For a general whole-line configuration, however, solitons may pass through the
origin. When this happens, the corresponding $(k,\sigma)$-seats are moved
from the left-hand side of the origin to the right-hand side. Hence the
enumeration of the seat number configuration with respect to the origin
changes, and the offset term is needed to describe this change. We also note that by \eqref{eq:car_l_seat_w} and \eqref{eq:car_n_seat_w}, if $\mu$ is an invariant measure for the BBS, then the expectation of $o_{k}(\eta)$ with respect to $\mu$ is 
    \begin{align}
        \E_{\mu}\left[ o_{k}\left(\eta\right) \right] = -2\E_{\mu}\left[ \sum_{\ell = 1}^{k} \mathcal{W}_{\ell}\left(0\right) \right] = -2\E_{\mu}\left[  W_{k}\left(0\right) \right].
    \end{align}
The relationship between the expectation of the offset and the effective velocity of $k$-solitons under certain invariant measures will be explained in Remark \ref{rem:offset_vel} and Appendix \ref{app:eff_velo}.}
                
            \end{remark}
            \begin{remark}
                Since we compute the offset with respect to the origin while \cite{FNRW} computes the offset with respect to the tagged $k$-slots, the values of the offsets are different. The offset in \cite{FNRW} is independent of $\left(\zeta_{\ell}\right)_{\ell \le k}$, but the formula \cite[$(3.1)$]{FNRW} is not very easy to compute. On the other hand, our offset $o_{k}$ may depend on $\zeta_{\ell}$ for some $\ell \le k$, but a simple formula is obtained. 
            \end{remark}

             The proof of Theorem \ref{thm:linear_whole} will be given in Section \ref{subsec:linear_whole}. 
             \begin{remark}\label{rem:two_bi}
                 Note that $\zeta$ can be considered as a map $\zeta : \Omega_{r} \to \Z_{\ge 0}^{\mathbb{\N} \times \Z}$, and it is shown in \cite{FNRW} that $\zeta$ is a bijection between $\Omega_{r} \subset \Omega$ and $\bar{\Omega} \subset \Z_{\ge 0}^{\mathbb{N} \times \Z}$, where $\Omega_{r}$ and $\bar{\Omega}$ are defined as
                        \begin{align}
                            \Omega_{r} &:= \left\{ \eta \in \Omega \ ; \ \left| s_{\infty}\left(\eta, i\right) \right| < \infty \text{ for any } i, \ s_{\infty}\left(\eta, 0\right) = 0 \right\}, \label{def:omega_r} \\
                            \bar{\Omega} &:= \left\{ {\bar{\zeta}} \in  \Z_{\ge 0}^{\mathbb{N} \times \Z} \ ; \ \sum_{k \in \N} {\bar{\zeta}}_{k}\left(i\right) < \infty \ \text{for any } i \right\}.
                        \end{align} 
                \end{remark}

    \subsection{Proof of Theorem \ref{thm:linear_whole}}\label{subsec:linear_whole}
        
        To show Theorem \ref{thm:linear_whole}, we need the following property of the seat number configuration. 
    
            \begin{proposition}\label{prop:flip}
                For any $\eta \in \Omega$, $k \in \N$ and $x \in \Z$, we have
                    \begin{align}\label{eq:flip_w}
                        \eta^{\downarrow}_{k}(x) =  T\eta^{\uparrow}_{k}(x).
                    \end{align}
                In addition, $\eta^{\uparrow}_{k}(x) = 1$ implies
                    \begin{align}
                        r\left(T\eta, x \right) +  \sum_{\ell \ge k} T\eta^{\downarrow}_{\ell}(x) = 1.
                    \end{align}
            \end{proposition}
            \begin{proof}[Proof of Proposition \ref{prop:flip}]
                First, we note that for the half-line case, the assertions of this proposition have been proven in \cite[Proposition 3.1]{MSSS}. In the proof of \cite[Proposition 3.1]{MSSS}, the boundary condition of the function $\widetilde{\mathcal{W}}_{k}\left(\eta, x\right) := 1 - \mathcal{W}_{k}\left(T\eta, x\right)$ is given by $\widetilde{\mathcal{W}}_{k}\left(\eta, 0\right) = 1$ for any $k \in \N$, but one can check that the proof does not depend on the boundary condition. Instead, the following condition $\sum_{y = 1}^{x} \left( 1 - 2T\eta\left(y\right) \right) \le 0, \ x \in \N$ is essential, and this condition trivially holds for the half-line case. 
                For the whole-line case, the following inequality
                    \begin{align}
                        \sum_{y = s_{\infty}\left(\eta, i\right) + 1}^{x} \left( 1 - 2T\eta\left(y\right) \right) \le 0,
                    \end{align}
                holds for any $i \in \Z$ and $x \ge s_{\infty}\left(\eta, i\right)$. Therefore, by following the strategy of \cite[Proposition 3.1]{MSSS}, we obtain \eqref{prop:flip} for $\eta^{(i)}$. Hence, from Lemma \ref{lem:indep_c}, by taking the limit $i \to \infty$, we have \eqref{prop:flip} for any $\eta \in \Omega$. 
            \end{proof}
            \begin{proof}[Proof of Theorem \ref{thm:linear_whole}]

                {The argument is almost the same as that in the proof of
\cite[Theorem~2.1]{MSSS}. For this reason, we only give a sketch of the proof.}
                
                First we note that from Lemma \ref{lem:1_whole}, $\zeta_{k}\left(\eta, i\right)$ can be represented as
                    \begin{align}
                        \zeta_{k}\left(\eta, i\right) &= \left| \left\{ x \in \Z \ ; \ \eta^{\uparrow}_{k}\left(x\right) = 1, \ \xi_{k}\left(\eta,x\right) = i \right\}\right| \\
                        & \quad - \left| \left\{ x \in \Z \ ; \ \eta^{\uparrow}_{k+1}\left(x\right) = 1, \ \xi_{k}\left(\eta,x\right) = i +1 \right\}\right| \\
                        &= \left| \left\{ x \in \Z \ ; \ \eta^{\downarrow}_{k}\left(x\right) = 1, \ \xi_{k}\left(\eta,x\right) = i \right\}\right| \\
                        & \quad - \left| \left\{ x \in \Z \ ; \ \eta^{\downarrow}_{k+1}\left(x\right) = 1, \ \xi_{k}\left(\eta,x\right) = i +1 \right\}\right|.
                    \end{align}
                    
                Now we show that for any $k \in \N$, the following quantity
                    \begin{align}
                        \xi_{k}\left(T\eta,x\right) - \xi_{k}\left(\eta,x\right) - W_{k}\left(T\eta, x\right) - W_{k}\left(\eta, x\right)
                    \end{align}
                is independent of $x \in \Z$ and equal to $o_{k}\left(\eta\right)$.  
                {We only consider the case $x \ge s_{\infty}\left(\eta, 0\right)$. The case $x \le s_{\infty}\left(\eta, 0\right) - 1$ can be shown in a similar way.} From \eqref{eq:car_l_seat_w}, \eqref{eq:car_n_seat_w} and proposition \ref{prop:flip}, we have
                    {\begin{align}
                        &\xi_{k}\left(T\eta,x\right) - \xi_{k}\left(\eta,x\right) - W_{k}\left(T\eta, x\right) - W_{k}\left(\eta, x\right) \\
                        &= \left(x - s_{\infty}\left(T\eta, 0\right) -  \sum_{\ell = 1}^{k} \sum_{y = s_{\infty}\left(T\eta, 0\right) + 1}^{x} \sum_{\sigma \in \{\uparrow, \downarrow\}} T\eta^{\sigma}_{\ell}\left(y\right) \right) \\
                        & \quad - \left(x - s_{\infty}\left(\eta, 0\right) -  \sum_{\ell = 1}^{k}\sum_{y = s_{\infty}\left(\eta, 0\right) + 1}^{x} \sum_{\sigma \in \{\uparrow, \downarrow\}} \eta^{\sigma}_{\ell}\left(y\right) \right) & & \text{{(by the def. of $\xi_{k}$)}} \\
                        & \quad - \sum_{\ell = 1}^{k} \sum_{y = s_{\infty}\left(T\eta, 0\right) + 1}^{x}\left( T\eta^{\uparrow}_{\ell}\left(y\right) -T\eta^{\downarrow}_{\ell}\left(y\right) \right) \\
                        & \quad - \sum_{\ell = 1}^{k} \sum_{y = s_{\infty}\left(\eta, 0\right) + 1}^{x}\left( \eta^{\uparrow}_{\ell}\left(y\right) -\eta^{\downarrow}_{\ell}\left(y\right) \right) & & \text{{(by \eqref{eq:car_l_seat_w}, \eqref{eq:car_n_seat_w})}} \\
                        &= s_{\infty}\left(\eta, 0\right) - s_{\infty}\left(T\eta, 0\right) \\
                        & \quad + 2 \sum_{\ell = 1}^{k} \left( \sum_{y = s_{\infty}\left(\eta, 0\right) + 1}^{x} \eta^{\downarrow}_{\ell}\left(y\right) -  \sum_{y = s_{\infty}\left(T\eta, 0\right) + 1}^{x}  T\eta^{\uparrow}_{\ell}\left(y\right) \right) & & \text{{(by Proposition \ref{prop:flip})}} \\
                        &= s_{\infty}\left(\eta, 0\right) - s_{\infty}\left(T\eta, 0\right) \\
                        & \quad + 2 \sum_{\ell = 1}^{k} \left( \sum_{y = s_{\infty}\left(\eta, 0\right) + 1}^{0} \eta^{\downarrow}_{\ell}\left(y\right) -  \sum_{y = s_{\infty}\left(T\eta, 0\right) + 1}^{0}  T\eta^{\uparrow}_{\ell}\left(y\right) \right) & & \text{{(by Proposition \ref{prop:flip})}} \\
                        &= o_{k}\left(\eta\right).
                    \end{align}
                We note that by setting $x = s_{\infty}(T\eta,0) , s_{\infty}(\eta,0)$, $o_{k}\left(\eta\right)$ can be represented as 
                    \begin{align}
                        o_{k}\left(\eta\right) 
                        &= - \xi_{k}\left(\eta, s_{\infty}(T\eta,0) \right) - W_{k}\left(\eta, s_{\infty}(T\eta,0) \right) \\
                        &= \xi_{k}\left(T\eta, s_{\infty}(\eta,0) \right) - W_{k}\left(T\eta, s_{\infty}(\eta,0) \right).
                    \end{align}}
                    
                In addition, from \eqref{eq:car_l_seat_w}, \eqref{eq:car_n_seat_w}, Lemma \ref{lem:1_whole} and Proposition \ref{prop:flip}, if $x = s_{k - 1}\left(\eta, i\right)$ for some $i \in \Z$ and $r\left(\eta, x\right) = 0$, then we obtain 
                    \begin{align}
                        W_{k}\left(\eta, x\right) + W_{k}\left(T\eta, x\right) = k.
                    \end{align}
                {Actually,  if $\eta\left(x\right) = 1$, then by Proposition \ref{prop:flip}, $x$ is a record or $(\ell,\downarrow)$-seat in $T\eta$ where $\ell \ge k - 1$. Hence by  \eqref{eq:car_l_seat_w}, \eqref{eq:car_n_seat_w} and Lemma \ref{lem:1_whole}, $W_{k}\left(\eta, x\right) = k$ and $W_{k}\left(T\eta, x\right) = 0$. Similarly, if $\eta\left(x\right) = 0$, then $W_{k}\left(\eta, x\right) = 0$ and $W_{k}\left(T\eta, x\right) = k$.}
                
                {From the above, if $x = s_{k - 1}\left(\eta, i\right)$ for some $i \in \Z$ and $r\left(\eta, x\right) = 0$, then we get}
                    \begin{align}\label{eq:Txi_xi}
                        \xi_{k}\left(T\eta,x\right) - \xi_{k}\left(\eta,x\right) = k + o_{k}\left(\eta\right).
                    \end{align}
                {Therefore} we have
                    \begin{align}
                        &\zeta_{k}\left(\eta, i \right) \\
                        &= \left| \left\{ x \in \Z \ ; \ \eta^{\downarrow}_{k}\left(x\right) = 1, \ \xi_{k}\left(\eta,x\right) = i  \right\}\right| - \left| \left\{ x \in \Z \ ; \ \eta^{\downarrow}_{k+1}\left(x\right) = 1, \ \xi_{k}\left(\eta,x\right) = i + 1\right\}\right|  \\
                        &= \left| \left\{ x \in \Z \ ; \ T\eta^{\uparrow}_{k}\left(x\right) = 1, \ \xi_{k}\left(T\eta,x\right) = i + k + o_{k}\left(\eta\right)  \right\}\right| \\
                        & \quad - \left| \left\{ x \in \Z \ ; \ T\eta^{\uparrow}_{k+1}\left(x\right) = 1, \ \xi_{k}\left(T\eta,x\right) = i + k + o_{k}\left(\eta\right) + 1  \right\}\right| \quad \text{{(by Prop. \ref{prop:flip} and \eqref{eq:Txi_xi})}} \\
                        &= \zeta_{k}\left(T\eta, i + k + o_{k}\left(\eta\right)\right),
                    \end{align}
                and thus Theorem \ref{thm:linear_whole} is proved.

            \end{proof}
        
    \subsection{The \texorpdfstring{$k$}{k}-skip map on the whole-line}
    
        In this subsection, we will show how the propositions stated in section \ref{sec:seat_skip} can be generalized.
        For the whole-line case, we define the $k$-skip map $\Psi_{k} : \Omega \to \Omega$ as
            \begin{align}
                \Psi_{k}\left( \eta \right)(x) := \eta\left(s_{k}\left(\eta, x + \xi_{k}\left(\eta, 0\right) \right) \right).
            \end{align}
        {We note that the shift by $\xi_k(\eta,0)$ in the definition of $\Psi_k$
is needed because the origin is used as the reference point. Indeed,
$\xi_k(\eta,0)$ counts the number of sites which remain between the
$0$-th record $s_\infty(\eta,0)$ and the origin after removing the sites
corresponding to the seats numbered $1,\ldots,k$. Hence the $0$-th record
is moved closer to the origin after the $k$-skip procedure. This is the reason why the term $\xi_k(\eta,0)$ appears in the definition of
$\Psi_k$, see Figure~\ref{fig:shift-xi0} for example.} 


\begin{figure}[t]
\centering
\begin{tikzpicture}[
    x=.9cm,y=1cm,
    lab/.style={font=\small},
    removed/.style={red, very thick},
    rec/.style={fill=black, draw=black},
    kept/.style={fill=blue!70!black, draw=blue!70!black}
]

\def\ytop{1.5}
\draw[gray] (-7.5,\ytop) -- (1.5,\ytop);
\node[lab,left] at (-7.8,\ytop) {$\eta$};

\foreach \x in {-7,-6,-5,-4,-3,-2,-1,0,1}{
    \draw (\x,\ytop) circle (2pt);
}

\node[lab,below] at (-7,\ytop-0.28) {$s_\infty(\eta,0)$};

\node[lab,below] at (0,\ytop-0.28) {$0$};

\foreach \x in {-6,-4,-2}{
    \draw[removed] (\x-0.12,\ytop-0.12) -- (\x+0.12,\ytop+0.12);
    \draw[removed] (\x-0.12,\ytop+0.12) -- (\x+0.12,\ytop-0.12);
}

\foreach \x in {-5,-3,-1,0}{
    \fill[kept] (\x,\ytop) circle (2.2pt);
}


\def\ybot{0}
\draw[gray] (-4.5,\ybot) -- (1.5,\ybot);
\node[lab,left] at (-4.8,\ybot) {$\Psi_1(\eta)$};

\foreach \x in {-4,-3,-2,-1,0,1}{
    \draw (\x,\ybot) circle (2pt);
}

\node[lab,below] at (-4,\ybot-0.28) {$-\xi_1(\eta,0)$};
\node[lab,below] at (0,\ybot-0.28) {$0$};


\end{tikzpicture}
\caption{A schematic explanation of the shift by \(\xi_1(\eta,0)\) in the
definition of \(\Psi_1\). In this example, $\eta(s_{\infty}(\eta,0)) \dots \eta(0) = 01100111 \dots,$ and the removed sites are \((1,\sigma)\)-seats lying between $s_{\infty}(\eta,0)$ and the origin.}
\label{fig:shift-xi0}
\end{figure}

        Recall that $\eta^{(i)}$ is defined in \eqref{def:eta_i}. As we observed in Lemma \ref{lem:indep_c}, we also have the following. 
            \begin{lemma}\label{lem:indep_p}
                Suppose that $\eta \in \Omega$. Then, for any $k \in \N$, $i \in \Z_{\ge 0}$ and 
                
                \noindent $x \ge \xi_{k}\left(\eta, s_{\infty}\left(\eta, -i \right)\right) - \xi_{k}\left(\eta, 0\right)$, we have
                    \begin{align}\label{eq:indep_p}
                        \Psi_{k}\left(\eta^{(i)}\right)(x) = \Psi_{k}\left(\eta^{(i+1)}\right)(x).
                    \end{align}
                In particular, for any $x \in \Z$ we have
                    \begin{align}
                        \Psi_{k}\left(\eta\right)(x) = \lim_{i \to \infty} \Psi_{k}\left(\eta^{(i)}\right)(x).
                    \end{align}
            \end{lemma}
            \begin{proof}
                From \eqref{eq:cut_s}, we see that $\xi_{k}\left(\eta^{(i)}, x \right) = \xi_{k}\left(\eta^{(i+1)}, x \right)$ for any $x \ge s_{\infty}\left(\eta, -i \right)$. Hence, we get $s_{k}\left(\eta^{(i)}, x \right) = s_{k}\left(\eta^{(i+1)}, x \right)$ for any $x \ge \xi_{k}\left(\eta, s_{\infty}\left(\eta, -i \right)\right)$. Thus we obtain \eqref{eq:indep_p}.
            \end{proof}
        Now, we generalize Propositions \ref{prop:seat_semig}, \ref{prop:semig} and \ref{prop:shift} for the whole-line case as follows :  
        
            \begin{proposition}
                Suppose that $\eta\in\Omega$. Then, for any $k, \ell \in \N$, $\sigma \in \{\uparrow, \downarrow \}$ and $x \in \Z$, we have  
                    \begin{align}\label{eq:whole_1}
                        \Psi_{k}\left( \eta \right)^{\sigma}_{\ell}(x) = \eta^{\sigma}_{k+\ell}\left(s_{k}\left(\eta, x + \xi_{k}\left(\eta, 0\right) \right) \right).
                    \end{align}
                In addition, we have
                    \begin{align}
                        \Psi_{k}\left( \Psi_{\ell}\left( \eta \right) \right)(x) = \Psi_{k+\ell}\left( \eta \right)(x) \label{whole_2}, 
                    \end{align}
                and
                    \begin{align}
                        \zeta_{k}\left(\Psi_{\ell}\left( \eta \right), i \right) = \zeta_{k+\ell}\left(\eta, i \right) \label{whole_3}.
                    \end{align}
            \end{proposition}
                \begin{proof}
                    First we show \eqref{eq:whole_1}. Thanks to Lemma \ref{lem:indep_p}, it is sufficient to show that
                        \begin{align}\label{eq:whole_p1}
                            \Psi_{k}\left( \eta^{(i)} \right)^{\sigma}_{\ell}(x) = \left(\eta^{(i)}\right)^{\sigma}_{k+\ell}\left(s_{k}\left(\eta, x + \xi_{k}\left(\eta, 0\right) \right) \right)
                        \end{align}
                    for any $i \in \Z_{\ge 0}$,  $k, \ell \in \N$, $\sigma \in \{\uparrow, \downarrow \}$ and $x \in \Z$. We fix $i \in \Z_{\ge 0}$ and define $\tilde{\eta}^{(i)} \in \Omega$ as
                        \begin{align}
                            \tilde{\eta}^{(i)}\left( x \right) := \eta^{(i)}\left( x + s_{\infty}\left(\eta, -i \right) \right).
                        \end{align}
                    Since $r\left(\tilde{\eta}^{(i)}, x \right) = 1$ for any $x \le 0$ and $\tilde{\eta}^{(i)}$ can be regarded as an element of $\{0,1\}^{\N}$, from Proposition \ref{prop:seat_semig}, we get 
                        \begin{align}
                            \Psi_{k}\left( \tilde{\eta}^{(i)} \right)^{\sigma}_{\ell}(x) = \left(\tilde{\eta}^{(i)}\right)^{\sigma}_{k+\ell}\left(s_{k}\left(\tilde{\eta}^{(i)}, x \right) \right),
                        \end{align}
                    for any $k, \ell \in \N$, $\sigma \in \{\uparrow, \downarrow \}$ and $x \in \Z$. 
                    On the other hand,  by direct computation, for any $x \in \Z$ we obtain
                        \begin{align}
                            \xi_{k}\left(\tilde{\eta}^{(i)}, x\right)
                            = \xi_{k}\left(\eta^{(i)}, x + s_{\infty}\left(\eta, -i \right)\right) - \xi_{k}\left(\eta^{(i)},  s_{\infty}\left(\eta, -i \right)\right),
                        \end{align}
                    and thus we have
                        \begin{align}
                            s_{k}\left(\tilde{\eta}^{(i)}, x\right) = s_{k}\left( \eta^{(i)}, x + \xi_{k}\left(\eta^{(i)},  s_{\infty}\left(\eta, -i \right)\right) \right) - s_{\infty}\left(\eta, -i \right).
                        \end{align}
                    From the above, we obtain
                        \begin{align}
                            \Psi_{k}\left( \tilde{\eta}^{(i)} \right)(x) &= \tilde{\eta}^{(i)}\left(s_{k}\left(\tilde{\eta}^{(i)}, x \right) \right) \\
                            &= \eta^{(i)}\left(s_{k}\left(\tilde{\eta}^{(i)}, x \right) + s_{\infty}\left(\eta, -i \right) \right) \\
                            &= \eta^{(i)}\left(  s_{k}\left(\eta^{(i)}, x + \xi_{k}\left(\eta^{(i)}, s_{\infty}\left(\eta, -i \right)\right) \right)\right) \\
                            &= \Psi_{k}\left( \eta^{(i)} \right)\left(x + \xi_{k}\left(\eta^{(i)}, s_{\infty}\left(\eta, -i \right)\right) -  \xi_{k}\left(\eta^{(i)}, 0\right) \right) \label{eq:whole_p2},
                        \end{align}
                    for any $k, \ell \in \N$ and $x \in \Z$. In particular, we have 
                        \begin{align}
                            &\left(\eta^{(i)}\right)^{\sigma}_{k + \ell}\left(  s_{k}\left(\eta^{(i)}, x + \xi_{k}\left(\eta^{(i)}, s_{\infty}\left(\eta, -i \right)\right) \right)\right) \\
                            &= \left(\tilde{\eta}^{(i)}\right)^{\sigma}_{k + \ell}\left(s_{k}\left(\tilde{\eta}^{(i)}, x \right) \right) = \Psi_{k}\left( \tilde{\eta}^{(i)} \right)^{\sigma}_{\ell}(x) \\
                            &= \Psi_{k}\left( \eta^{(i)} \right)^{\sigma}_{\ell}\left(x + \xi_{k}\left(\eta^{(i)}, s_{\infty}\left(\eta, -i \right)\right) -  \xi_{k}\left(\eta^{(i)}, 0\right) \right),
                        \end{align}
                    for any $k, \ell \in \N$, $\sigma \in \{\uparrow, \downarrow \}$ and $x \in \Z$. Therefore we have \eqref{eq:whole_p1}. For later use, we note some equations derived from \eqref{eq:whole_1}. We observe that from \eqref{eq:whole_1}, if $y = s_{k}\left( \Psi_{\ell}\left(\eta\right), x \right)$ for some $k \in \N \cup \{\infty\}$, $\ell \in \N$ and $x \in \Z$, then $s_{\ell}(\eta, y + \xi_{\ell}(\eta, 0)) = s_{k + \ell}(\eta, x)$. Hence, for any $k \in \N \cup \{\infty\}$, $\ell \in \N$ and $x \in \Z$, we have 
                        \begin{align}
                            &s_{k}\left( \Psi_{\ell}\left(\eta\right), x \right) - s_{k}\left( \Psi_{\ell}\left(\eta\right), x - 1 \right) \\ 
                            &= \sum_{y = s_{k}\left( \Psi_{\ell}\left(\eta\right), x - 1 \right) + 1}^{s_{k}\left( \Psi_{\ell}\left(\eta\right), x \right)} r\left(\Psi_{\ell}\left(\eta\right), y \right) + \left(  \sum_{h \in \N} \sum_{\sigma \in \{ \uparrow, \downarrow \}} \Psi_{\ell}\left(\eta\right)^{\sigma}_{h}\left(y\right) \right) \\
                            &= \sum_{y = s_{k}\left( \Psi_{\ell}\left(\eta\right), x - 1 \right) + 1}^{s_{k}\left( \Psi_{\ell}\left(\eta\right), x \right)} r\left(\eta, s_{\ell}\left(\eta, y + \xi_{\ell}\left(\eta, 0\right)\right) \right) + \left(  \sum_{h \in \N} \sum_{\sigma \in \{ \uparrow, \downarrow \}} \eta^{\sigma}_{\ell + h}\left(s_{\ell}\left(\eta, y + \xi_{\ell}\left(\eta, 0\right)\right)\right) \right) \\
                            &= \sum_{y = s_{k + \ell}\left( \eta, x - 1 \right) + 1}^{s_{k + \ell}\left( \eta, x \right)} r\left(\eta, y \right) + \left(  \sum_{h \in \N} \sum_{\sigma \in \{ \uparrow, \downarrow \}} \eta^{\sigma}_{\ell + h}\left(y\right) \right) \\
                            &= \xi_{\ell}\left(\eta, s_{k + \ell}\left( \eta, x \right) \right) - \xi_{\ell}\left(\eta, s_{k + \ell}\left( \eta, x - 1 \right) \right),
                        \end{align}
                    where at the third equality we use the fact that $(s_{k+h}(\eta, x))_{x \in \Z} \subset (s_{k}(\eta, x))_{x \in \Z}$ for any $k \in \N$ and $h \in \N \cup \{ \infty \}$. By using the same computation, for any $k \in \N \cup \{\infty\}$ and $\ell \in \N$, we obtain 
                        \begin{align}
                            s_{k}\left( \Psi_{\ell}\left(\eta\right), 0 \right) &= s_{\infty}\left( \Psi_{\ell}\left(\eta\right), 0 \right) \\
                            &= - \sum_{y = s_{\infty}\left( \Psi_{\ell}\left(\eta\right), 0 \right) + 1}^{0}  \sum_{h \in \N} \sum_{\sigma \in \{ \uparrow, \downarrow \}} \Psi_{\ell}\left(\eta\right)^{\sigma}_{h}\left(y\right) \\
                            &= - \sum_{y = s_{\infty}\left( \eta, 0 \right) + 1}^{0}   \sum_{h \in \N} \sum_{\sigma \in \{ \uparrow, \downarrow \}} \eta^{\sigma}_{\ell + h}\left(y\right)  \\
                            &= - \xi_{\ell}\left(\eta, 0\right), 
                        \end{align}
                    and for any $k, \ell \in \N$, we also get
                        \begin{align}
                            \xi_{k}\left(\Psi_{\ell}\left(\eta\right), 0\right) 
                            &= \sum_{y = s_{\infty}\left( \Psi_{\ell}\left(\eta\right), 0 \right) + 1}^{0}   \sum_{h \in \N} \sum_{\sigma \in \{ \uparrow, \downarrow \}} \Psi_{\ell}\left(\eta\right)^{\sigma}_{k + h}\left(y\right) \\
                            &= \sum_{y = s_{\infty}\left( \eta, 0 \right) + 1}^{0}  \sum_{h \in \N} \sum_{\sigma \in \{ \uparrow, \downarrow \}} \eta^{\sigma}_{k + \ell + h}\left(y\right) \\
                            &= \xi_{k + \ell}\left(\eta, 0\right) \label{eq:whole_p2***}.
                        \end{align}
                    In particular, for any $k \in \N \cup \{\infty\}$, $\ell \in \N$ and $x \in \Z$, we have 
                        \begin{align}\label{eq:whole_p2**}
                            s_{k}\left( \Psi_{\ell}\left(\eta\right), x \right) = \xi_{\ell}\left(\eta, s_{k + \ell}\left( \eta, x \right) \right) - \xi_{\ell}\left(\eta, 0\right).
                        \end{align}

                    Next we show \eqref{whole_2}. From \eqref{eq:whole_p2***} and \eqref{eq:whole_p2**}, for any $k, \ell \in \N$ and $x \in \Z$ we have 
                        \begin{align}
                            \Psi_{k}\left(\Psi_{\ell}\left(\eta\right)\right)\left(x\right) 
                            &= \Psi_{\ell}\left(\eta\right)\left( s_{k}\left( \Psi_{\ell}\left(\eta\right), x + \xi_{k}\left(\Psi_{\ell}\left(\eta\right), 0\right) \right) \right) \\
                            &= \Psi_{\ell}\left(\eta\right)\left( \xi_{\ell}\left(\eta, s_{k+\ell}\left(\eta, x + \xi_{k + \ell}\left(\eta, 0\right) \right)\right) - \xi_{\ell}\left(\eta, 0\right) \right) \\
                            &= \eta\left(s_{\ell}\left( \eta, \xi_{\ell}\left(\eta, s_{k+\ell}\left(\eta, x + \xi_{k + \ell}\left(\eta, 0\right) \right)\right) \right)\right) \\
                            &= \eta\left(s_{k+\ell}\left(\eta, x + \xi_{k + \ell}\left(\eta, 0\right)  \right)\right) \\
                            &= \Psi_{k + \ell}\left(\eta\right)(x),
                        \end{align}
                    and thus we obtain \eqref{whole_2}.

                    Finally we show \eqref{whole_3}. From \eqref{eq:whole_1} and \eqref{eq:whole_p2**}, for any $k, \ell \in \N$ and $i \in \Z$, we have
                        \begin{align}
                            &\zeta_{k}\left( \Psi_{\ell}\left(\eta\right), i \right) \\ &= \sum_{y = s_{k}\left(\Psi_{\ell}\left(\eta\right), i \right) + 1}^{s_{k}\left(\Psi_{\ell}\left(\eta\right), i + 1 \right)} \left( \Psi_{\ell}\left(\eta\right)^{\uparrow}_{k}\left(y\right) - \Psi_{\ell}\left(\eta\right)^{\uparrow}_{k+1}\left(y\right) \right) \\
                            &= \left( \sum_{y = s_{k}\left(\Psi_{\ell}\left(\eta\right), i \right) }^{s_{k}\left(\Psi_{\ell}\left(\eta\right), i + 1 \right)}  \Psi_{\ell}\left(\eta\right)^{\uparrow}_{k}\left(y\right) \right) - \Psi_{\ell}\left(\eta\right)^{\uparrow}_{k+1}\left(s_{k}\left(\Psi_{\ell}\left(\eta\right), i + 1 \right)\right)  \\
                            &= \left(  \sum_{y = \xi_{\ell}\left(\eta, s_{k + \ell}\left( \eta, i \right) \right) - \xi_{\ell}\left(\eta, 0\right) }^{\xi_{\ell}\left(\eta, s_{k + \ell}\left( \eta, i + 1 \right) \right) - \xi_{\ell}\left(\eta, 0\right)} \eta^{\uparrow}_{k+\ell}\left(s_{\ell}\left(\eta, y + \xi_{\ell}\left(\eta, 0\right)\right)\right)  \right) - \eta^{\uparrow}_{k+\ell}\left(s_{\ell}\left(\eta, \xi_{\ell}\left(\eta, s_{k + \ell}\left( \eta, i + 1 \right) \right) \right)\right) \\
                            &= \left( \sum_{y = s_{k + \ell}\left( \eta, i \right) }^{s_{k + \ell}\left( \eta, i + 1 \right)} \eta^{\uparrow}_{k+\ell}\left(y\right) \right) - \eta^{\uparrow}_{k+\ell}\left(s_{k + \ell}\left(\eta, i + 1 \right)\right) \\
                            &= \zeta_{k + \ell}\left( \eta, i \right),
                        \end{align}
                    and thus we obtain \eqref{whole_3}.
                
                \end{proof}                

        We conclude this section by describing the relation between $T$ and $\Psi_{k}$. {For any $y \in \Z$, we denote by $\tau_{y} : \{0,1\}^{\Z} \to \{0,1\}^{\Z}$ the spatial shift by $y$, i.e., for any $\eta \in \{0,1\}^{\Z}$ and $x \in \Z$, 
            \begin{align}\label{def:op_shift}
                \left(\tau_{y}\eta\right)\left(x\right) := \eta\left(x+y\right).
            \end{align}}
            \begin{proposition}
                Suppose that $\eta \in \Omega$. Then, for any $k \in \N$, we have 
                    {\begin{align}\label{eq:T_Psi_k}
                        \left(T \circ \Psi_{k}\right)\left(\eta\right) = \left(\tau_{\left({k - W_{k}\left(\eta,0\right) -  W_{k}\left(T\eta,0\right)}\right)} \circ \Psi_{k} \circ T \right)\left(\eta\right),
                    \end{align}
                i.e.,} for any $x \in \Z$,
                    \begin{align}
                        T\Psi_{k}\left(\eta\right)\left(x \right) = \Psi_{k}\left(T\eta\right)\left(x {+ k - W_{k}\left(\eta,0\right) -  W_{k}\left(T\eta,0\right)} \right).
                    \end{align}
            \end{proposition}
            \begin{proof}

                Suppose that $T\Psi_{{k}}\left(\eta\right)\left(x \right) = 1$. Then, there exists $\ell \in \N$ such that 
                    \begin{align}
                        T\Psi_{{k}}\left(\eta\right)^{\uparrow}_{\ell}\left(x \right) = 1. 
                    \end{align}
                From \eqref{eq:flip_w} and \eqref{eq:whole_1}, we have 
                    \begin{align}
                        \eta^{\downarrow}_{\ell + {k}}\left( s_{{k}}\left(\eta, x + \xi_{{k}}\left( \eta,0 \right) \right) \right) = \Psi_{{k}}\left(\eta\right)^{\downarrow}_{\ell}\left(x \right) = 1.
                    \end{align}
                Again by using \eqref{eq:flip_w}, we obtain 
                    \begin{align}
                        T\eta^{\uparrow}_{\ell+{k}}\left( s_{{k}}\left(\eta, x  + \xi_{{k}}\left( \eta,0 \right) \right) \right) = 1.
                    \end{align}
                On the other hand, from \eqref{eq:Txi_xi} we get
                    \begin{align}
                        &\xi_{{k}}\left(T\eta, s_{{k}}\left(\eta, x + \xi_{{k}}\left( \eta,0 \right) \right)\right) \\
                        &= x + \xi_{{k}}\left( \eta,0 \right) + 1 + o_{{k}}\left(\eta\right) \\
                        &= x  { + k} { - } W_{{k}}\left(\eta,0\right) { - }  W_{{k}}\left(T\eta,0\right) + \xi_{{k}}\left( T\eta,0 \right).
                    \end{align}
                Since the site $s_{{k}}\left(\eta, x  + \xi_{{k}}\left( \eta,0 \right) \right)$ is a $\left(\ell + {k},\uparrow\right)$-seat in $T\eta$,  we have
                    \begin{align}
                        s_{{k}}\left(T\eta, x { + k - W_{k}\left(\eta,0\right) -  W_{k}\left(T\eta,0\right)} + \xi_{{k}}\left( T\eta,0 \right) \right) = s_{{k}}\left(\eta, x + \xi_{{k}}\left( \eta,0 \right) \right). 
                    \end{align}
                Hence we have 
                    \begin{align}
                        &\Psi_{{k}}\left(T\eta\right)^{\uparrow}_{\ell}\left(x { + k - W_{k}\left(\eta,0\right) -  W_{k}\left(T\eta,0\right)} \right) \\
                        &= T\eta^{\uparrow}_{\ell+{k}}\left(s_{{k}}\left(T\eta, x  { + k - W_{k}\left(\eta,0\right) -  W_{k}\left(T\eta,0\right)} +  \xi_{{k}}\left( T\eta,0 \right) \right) \right) \\
                        &= T\eta^{\uparrow}_{\ell+{k}}\left( s_{{k}}\left(\eta, x  + \xi_{{k}}\left( \eta,0 \right) \right) \right) \\
                        &= 1,
                    \end{align}
                and thus we see that $T\Psi_{{k}}\left(\eta\right)\left(x \right) = 1$ implies $\Psi_{{k}}\left(T\eta\right)\left(x { + k - W_{k}\left(\eta,0\right) -  W_{k}\left(T\eta,0\right)} \right) = 1$. 
                By the same computation, we can also show that $\Psi_{{k}}\left(T\eta\right)\left(x { + k - W_{k}\left(\eta,0\right) -  W_{k}\left(T\eta,0\right)} \right) = 1$ implies $T\Psi_{{k}}\left(\eta\right)\left(x \right) = 1$.

            \end{proof}
                
\section{{The \texorpdfstring{$k$}{k}-skip map and invariant measures}}\label{sec:dis}

            {The purpose of this section is to apply the $k$-skip map to the BBS with random initial distributions. We consider a class of invariant measures introduced in
\cite{FG}.  To the best of our knowledge, this class is the most general
known family of invariant measures for the BBS, and it is also well suited
to the $k$-skip map. We will derive the distribution of the $k$-skipped configuration under such invariant measures.

The $k$-skip map is also useful for computing statistical
quantities of the randomized BBS.  In particular, it leads
to a recursive relation for the mean size of excursions, 
and this relation gives explicit formulae for quantities related to the
stationary current and the effective velocity of solitons.}
             Throughout this section, we restrict the state space $\Omega_{*} \subset \Omega$, defined as
                \begin{align}
                    \Omega_{*} &:= \left\{ \eta \in \Omega \ ; \ |s_{\infty}\left(\eta, i \right) | < \infty \text{ for any } i \in \Z \right\}.
                \end{align}
            
        \subsection{Excursion}\label{subsubsec:ex}
            First, we introduce the notion of {\it excursion}, which will be used to define a class of invariant measures of the BBS. For any $n \in \Z_{\ge 0}$, we say that a sequence $\left(e_{j}\right)_{j = 0}^{2n}, e_{j} \in \{0,1\}$ is an excursion with length $2n + 1$ if 
                \begin{align}
                    e_{0} = 0, \quad \sum_{j = 1}^{m} \left(2e_{j} - 1\right) > 0 \ \text{for any } 1 \le m < 2n, \quad \sum_{j = 1}^{2n} \left(2e_{j} - 1\right) = 0.
                \end{align}
            We denote by $\mathcal{E}_{n}$ the set of all excursions with length $2n + 1$, and denote by $\mathcal{E} := \bigcup_{n \in \Z_{\ge 0}} \mathcal{E}_{n}$ the set of all excursions.
            There is a natural injection $\iota : \mathcal{E} \to \{0,1\}^{\Z}$ given by 
                \begin{align}
                    \iota\left(\e\right)(x) := \begin{dcases}
                        \e_{x} \ &  {0} \le  x \le |\e| {-1}, \\
                        0 \ & \text{otherwise},
                    \end{dcases}
                \end{align}
            {where $\e_{x}$ is the $x$-th binary coordinate in the excursion $\e$,} and $\Omega_{1} := \iota\left(\mathcal{E}\right)$ is written as 
                \begin{align}
                    \Omega_{1} = \left\{ \eta \in \{0,1\}^{\Z} \ ; \ \eta\left(x\right) = 0 \text{ for any } x < 0, x \ge s_{\infty}\left(\eta, 1\right) \right\}.
                \end{align}
            Observe that for any $\e \in \mathcal{E}$ and $k \in \N$, we have $\Psi_{k}\left(\iota\left(\e\right)\right) \in \Omega_{1}$. Hence, the following map  
                \begin{align}\label{def:tilde_skip}
                    \widetilde{\Psi}_{k}\left(\e\right) := \iota^{-1}\left(\Psi_{k}\left(\iota\left(\e\right)\right)\right),
                \end{align}
            is well-defined for any $\e$, and we call $\widetilde{\Psi}_{k} : \mathcal{E} \to \mathcal{E}$ the $k$-skip map for excursions. 
            Also, we extend the notion of $\zeta$ for excursions. For any $\e \in \mathcal{E}$ and $k \in \N$ we define 
                \begin{align}\label{def:zeta_excursion}
                    \zeta_{k}\left(\e\right) 
                    &:= \sum_{i = 0}^{\xi_{k}\left(\iota\left(\e\right), s_{\infty}\left(\iota\left(\e\right), 1\right)\right) - 1} \zeta_{k}\left( \iota\left(\e\right), i \right) \\
                    &= \sum_{x = 1}^{s_{\infty}\left(\iota\left(\e\right), 1\right)} \sum_{\sigma \in \{\uparrow, \downarrow\} } \left(\left(\iota\left(\e\right)\right)^{\sigma}_{k}\left(x\right) - \left(\iota\left(\e\right)\right)^{\sigma}_{k+1}\left(x\right) \right).
                \end{align}
            {By Remark \ref{rem:zeta_sol}, $\zeta_{k}\left(\e\right)$ is the number of $k$-solitons in the excursion $\e$. }
            We note that by using $\zeta_{k}$, $k \in \N$, the length of excursion $\e$ can be written as follows: 
                \begin{align}
                    |\e| &= s_{\infty}\left(\iota\left(\e\right), 1\right) \\
                    &= 1 +  \sum_{x = 1}^{s_{\infty}\left(\iota\left(\e\right), 1\right)} \sum_{k \in \N} \sum_{\sigma \in \{\uparrow, \downarrow\} } \left(\iota\left(\e\right)\right)^{\sigma}_{k}\left(x\right) \\ 
                    &= 1 + 2 \sum_{k \in \N} \sum_{i = 0}^{\xi_{k}\left(\iota\left(\e\right), s_{\infty}\left(\iota\left(\e\right), 1\right)\right) - 1} k  \zeta_{k}\left( \iota\left(\e\right), i \right) \\
                    &= 1 + 2 \sum_{k \in \N} {k} \zeta_{k}\left(\e\right) \label{eq:length_z}.
                \end{align}
            In addition, from Proposition \ref{prop:shift}, we see that $\widetilde{\Psi}_{k}$ is a shift operator of $\zeta$ : 
                \begin{proposition}\label{prop:shift_excursion}
                    For any $\e \in \mathcal{E}$ and $k, \ell \in \N$, we have 
                        \begin{align}\label{eq:shift_excursion}
                            \zeta_{k}\left(\widetilde{\Psi}_{\ell}\left(\e\right)\right) = \zeta_{k + \ell}\left(\e\right).
                        \end{align}
                \end{proposition}
            
            Now we introduce a family of probability measures on $\mathcal{E}$ via $\zeta_{k} : \mathcal{E} \to \Z_{\ge 0}$. For any $\a = \left(\a_{k}\right)_{k \in \N}$, we define 
                \begin{align}
                    \mathcal{A} &:= \left\{ \boldsymbol{\a} = \left(\a_{k}\right)_{k \in \N} \subset [0,1)^{\N} \ ; \ Z_{\a} := \sum_{\e \in \mathcal{E}} \prod_{k \in \N} \a_{k}^{\zeta_{k}\left(\e\right)} < \infty \right\}. 
                \end{align}
            Then, for any $\a \in \mathcal{A}$, we define a canonical probability measure $\nu_{\a}$ on $\mathcal{E}$ as
                \begin{align}
                    \nu_{\boldsymbol{\a}}\left(\e\right) := \frac{1}{Z_{\boldsymbol{\a}}} \prod_{k \in \N} \a_{k}^{\zeta_{k}\left(\e\right)}. 
                \end{align}
            We denote by $\mathcal{A}_{+} \subset \mathcal{A}$ the set of all $\a \in \mathcal{A}$ such that the expectation of $|\e|$ under $\nu_{\boldsymbol{\a}}$ is finite, i.e., 
                \begin{align}
                    \mathcal{A}^{+} := \left\{ \boldsymbol{\a} = \left(\a_{k}\right)_{k \in \N} \in \mathcal{A} \ ; \ \bar{\e}\left(\boldsymbol{\a}\right)  < \infty \right\}, 
                \end{align}
            where 
                \begin{align}\label{def:exp_excursion}
                    \bar{\e}\left(\boldsymbol{\a}\right) := 1 + 2 \sum_{\e \in \mathcal{E}} \sum_{k \in \N} k \zeta_{k}\left(\e\right) \nu_{\a}\left(\e\right).
                \end{align}
            Next, we introduce a shift operator on $\mathcal{A}$. For any $\boldsymbol{\a} \in \mathcal{A}$, we define $\theta\boldsymbol{\a} = \left( \left(\theta\a\right)_{k}\right)_{k \in \N} \in \mathcal{A}$ as 
                \begin{align}\label{def:shift_para}
                    \left(\theta\a\right)_{k} := \frac{\a_{k + 1}}{\left( 1- \a_{1} \right)^{2k}}.
                \end{align}
            We note that the fact $\theta : \mathcal{A} \to \mathcal{A}$ can be checked by \cite[Lemma 3.4]{FG}. In addition, one can also show that $\theta|_{\mathcal{A}^{+}} : \mathcal{A}^{+} \to \mathcal{A}^{+}$, see Remark \ref{rem:theta_A+} below. 
            
            Now we compute the distribution of $\Psi_{k}\left(\e\right)$ when the distribution of $\e$ is given by $\nu_{\a}$. 
                \begin{lemma}\label{lem:shift_excursion_m}
                    Suppose that $\a \in \mathcal{A}$. Then, for any $k \in \N$ and $\e' \in \mathcal{E}$ we have
                        \begin{align}\label{eq:shift_excursion_m}
                            \nu_{\boldsymbol{\a}}\left(\left\{ \e \in \mathcal{E} \ ; \  \Psi_{k}\left(\e\right) = \e' \right\}\right) = \nu_{\theta^{k}\boldsymbol{\a}}\left( \e' \right).
                        \end{align}
                \end{lemma}
                \begin{proof}[Proof of Lemma \ref{lem:shift_excursion_m}]
                    
                    From Proposition \ref{prop:shift_excursion} and the uniformity of $\nu_{\a}$, it is sufficient to show that 
                        \begin{align}
                            \nu_{\boldsymbol{\a}}\left( \zeta_{k}\left( \Psi_{1}\left(\e\right) \right) = n_{k} \text{ for any } k\in \N \right) = \nu_{\theta\boldsymbol{\a}}\left( \zeta_{k}\left( \e \right) = n_{k} \text{ for any } k\in \N \right)
                        \end{align}
                    for $\mathbf{n} = \left(n_{k}\right)_{k \in \Z_{\ge 0}} \subset \left(\Z_{\ge 0}\right)^{\N}$ such that $n_{k} = 0$ for sufficiently large $k$. For such $\mathbf{n}$, we define 
                        \begin{align}
                            \mathcal{E}\left(\mathbf{n}\right) := \left\{ \e \in \mathcal{E} \ ; \ \zeta_{k}\left(\e\right) = n_{k} \text{ for any } k \in \N \right\}.
                        \end{align}
                    It is known that $\left| \mathcal{E}\left(\mathbf{n}\right) \right|$ is given via the so-called {\it Fermionic formula}, 
                        \begin{align}
                            \left| \mathcal{E}\left(\mathbf{n}\right) \right| = \prod_{k = 1}^{\infty} \begin{pmatrix} 2 \sum_{\ell \ge k + 1} (\ell - k) n_{\ell} + n_{k} \\ n_{k} \end{pmatrix},
                        \end{align}
                    see \cite{KTT} for details. Then, from Proposition \ref{prop:shift_excursion} we have
                        \begin{align}
                            &\nu_{\boldsymbol{\a}}\left( \zeta_{k}\left( \Psi_{1}\left(\e\right) \right) = n_{k} \text{ for any } k\in \N \right)  \\
                            &= \sum_{n = 0}^{\infty} \nu_{\boldsymbol{\a}}\left( \zeta_{k+1}\left(\e \right) = n_{k} \text{ for any } k\in \N, \zeta_{1}\left(\e \right) = n \right) \\
                             &= \frac{1}{Z_{\boldsymbol{\a}}} \prod_{k \in \N} \a_{k+1}^{n_{k}} \prod_{m = 2}^{\infty} \begin{pmatrix} 2 \sum_{\ell \ge m + 1} (\ell-m) n_{\ell - 1} + n_{m - 1} \\ n_{m - 1} \end{pmatrix} \sum_{n = 0}^{\infty} \begin{pmatrix} 2 \sum_{\ell \ge 2} (\ell - 1) n_{\ell - 1} + n \\ n \end{pmatrix} \a_{1}^{n} \\
                             &= \frac{1}{Z_{\boldsymbol{\a}}} \prod_{k \in \N} \a_{k+1}^{n_{k}} \prod_{m = 1}^{\infty} \begin{pmatrix} 2 \sum_{\ell \ge m + 1} (\ell-m) n_{\ell} + n_{m} \\ n_{m} \end{pmatrix} \left(\frac{1}{1 - \a_1}\right)^{2\sum_{\ell \ge 2} (\ell - 1)n_{\ell - 1} + 1} \\
                             &= \frac{1}{Z_{\boldsymbol{\a}}(1 - \a_1)} \prod_{k \in \N} \left(\theta \a\right)_{k}^{n_{k}} \prod_{m = 1}^{\infty} \begin{pmatrix} 2 \sum_{\ell \ge m + 1} (\ell-m) n_{\ell} + n_{m} \\ n_{m} \end{pmatrix} \\
                             &= \nu_{\theta\boldsymbol{\a}}\left( \zeta_{k}\left( \e \right) = n_{k} \text{ for any } k\in \N \right),
                        \end{align}
                    where we use the equation $Z_{\theta\a} = Z_{\a}(1 - \a_1)$ \cite[(3.22), (3.29)]{FG}, and  
                        \begin{align}
                            \sum_{n = 0}^{\infty} \begin{pmatrix} x + n \\ n \end{pmatrix} y^{n} = \left( \frac{1}{1 - y} \right)^{x + 1},
                        \end{align}
                    for any $x \in \Z_{\ge 0}$ and $y \in (0,1)$. 
                \end{proof}

            We recall that $\Omega_{r}$ is defined in \eqref{def:omega_r}.
            We observe that $\eta \in \Omega_{r}$ can be constructed from excursions as follows. For any $\left( \e_{i} \right)_{i \in \Z} \in \mathcal{E}^{\Z}$, we define $I\left(\left( \e_{i} \right)_{i \in \Z}\right)$ as 
                \begin{align}
                    &I\left( \left( \e_{i} \right)_{i \in \Z} \right)\left(x\right) \\
                    &\quad := 
                    \begin{dcases}
                        \iota\left(\e_{0}\right)(x) \  & \ \text{if } 0 \le x \le |\e_{0}| - 1, \\
                        \iota\left(\e_{-1}\right){\left(x + |\e_{-1}| \right)} \  & \ \text{if } { -} |\e_{-1}| \le x \le - 1, \\
                        \iota\left(\e_{i}\right){\left(x - \sum_{m = 0}^{i - 1} |\e_{m}| \right)} \ & \ \text{if } i \ge 1 \text{ and } \sum_{m = 0}^{i - 1} |\e_{m}| \le x \le \sum_{m = 0}^{i} |\e_{m}| - 1, \\
                        \iota\left(\e_{i}\right){\left(x + \sum_{m = i}^{-1} |\e_{m}| \right)} \ & \ \text{if } i \le - 2 \text{ and } - \sum_{m = i}^{-1} |\e_{m}| \le x \le - \sum_{m = i + 1}^{- 1} |\e_{m}| - 1,
                    \end{dcases}  
                \end{align}
            then we can check that $I$ is injective and $I\left(\mathcal{E}^{\Z}\right) = \Omega_{r}$. For later use, we prepare the following lemma.  
                \begin{lemma}\label{lem:exskip_etaskip}
                    Suppose that $\left( \e_{i} \right)_{i \in \Z} \in \mathcal{E}^{\Z}$. Then for any $k \in \N$, we have
                        \begin{align}\label{eq:exskip_etaskip}
                        \Psi_{k}\left( I\left( \left( \e_{i} \right)_{i \in \Z} \right) \right) = I\left( \left( \tilde{\Psi}_{k}\left(\e_{i}\right) \right)_{i \in \Z} \right).
                    \end{align}
                \end{lemma}
                \begin{proof}[Proof of Lemma \ref{lem:exskip_etaskip}]
                    Thanks to \eqref{whole_2}, it is sufficient to consider the case $k = 1$. 
                    This is a direct consequence of Lemma \ref{lem:indep_c}, \eqref{eq:whole_1}, \eqref{def:tilde_skip} and 
                    \begin{align}
                    s_{\infty}\left( I\left( \left( \e_{i} \right)_{i \in \Z} \right), i \right) = \begin{dcases}
                        \ 0 \ & \ \text{if } i = 0,  \\
                        \ \sum_{m = 0}^{i - 1}|\e_{m}| \ & \ \text{if } i \ge 1, \\
                        \ - \sum_{m = i}^{-1}|\e_{m}| \ & \ \text{if } i \le -1.
                    \end{dcases}
                \end{align}
                \end{proof}

            We denote by $\hat{\mu}_{\boldsymbol{\a}}$, $\boldsymbol{\a} \in \mathcal{A}^{+}$, the probability measure on $\Omega_{r}$ induced by the product probability measure $\prod_{\e_{i} \in \Z}\nu_{\boldsymbol{\a}}\left(\e_i\right), e_i \in \mathcal{E}$ on $\mathcal{E}^{\Z}$ via the map $I$. Then, for any $\boldsymbol{\a} \in \mathcal{A}^{+}$, we define a probability measure $\mu_{\boldsymbol{\a}}$ on $\Omega_{*}$ as
                \begin{align}\label{def:inv_palm}
                    \int_{\Omega_{*}} d\mu_{\boldsymbol{\a}}\left(\eta\right) f(\eta) := \frac{1}{\bar{\e}\left(\boldsymbol{\a}\right)} \int_{\Omega_{r}} d\hat{\mu}_{\boldsymbol{\a}}\left(\eta\right) \sum_{y = 0}^{s_{\infty}\left(\eta, 1\right) - 1} f \left(\tau_{y}\eta\right), 
                \end{align}
            for any local function $f : \{0,1\}^{\Z} \to \R$, where $\tau_{y}, y \in \Z$ is a spatial shift operator defined {in \eqref{def:op_shift}.}  Then, the following result is shown by \cite{FG}.
                \begin{theorem*}[Theorem 4.5 in \cite{FG}]\label{thm:FG}
                    Suppose that $\a \in \mathcal{A}^{+}$. Then, $\mu_{\a}$ is a shift-stationary invariant measure of the BBS($\infty$), i.e., for any $y \in \Z$ and local function $f : \{0,1\}^{\Z} \to \R$, 
                        \begin{align}
                            \int_{\Omega_{*}} d\mu_{\boldsymbol{\a}}\left(\eta\right) f\left( \tau_{y} \eta \right) = \int_{\Omega_{*}} d\mu_{\boldsymbol{\a}}\left(\eta\right) f\left( \eta \right), \quad \int_{\Omega_{*}} d\mu_{\boldsymbol{\a}}\left(\eta\right) f\left( T\eta \right) = \int_{\Omega_{*}} d\mu_{\boldsymbol{\a}}\left(\eta\right) f\left( \eta\right).
                        \end{align}
                \end{theorem*}
    \begin{remark}
        From the definition of $\mu_{\boldsymbol{\a}}$, the density of balls under $\mu_{\boldsymbol{\a}}$, denoted by  $\rho\left(\boldsymbol{\a}\right)$, is given by 
            \begin{align}\label{eq:density_ball}
                \rho\left(\boldsymbol{\a}\right) := \int_{\Omega} d\mu_{\boldsymbol{\a}}\left(\eta\right) \eta\left(0\right) =\frac{\bar{\e}\left(\boldsymbol{\a}\right) - 1}{2 \bar{\e}\left(\boldsymbol{\a}\right)}. 
            \end{align}
    \end{remark}
    \begin{remark}\label{rem:Markov}
        The class $\mu_{\boldsymbol{\a}}$, $\boldsymbol{\a} \in \mathcal{A}^{+}$ contains some important probability measures on $\{0,1\}^{\Z}$. Actually, 
        by \cite[Lemma 3.7, Corollary 4.8]{FG}, it is shown that by choosing $\boldsymbol{\a}$ as $\a_{k} := a b^{k-1}$ , $k \in \N$ with some $0 < a, b < 1$ such that $\sqrt{a} + \sqrt{b} < 1$, then $\mu_{\boldsymbol{\a}}$ becomes a two-sided homogeneous Markov distribution on $\{0,1\}^{\Z}$, where the transition matrix is given by \eqref{eq:transition}, defined in Section \ref{subsec:Markov} below. In particular, if $a = b$, then $\mu_{\boldsymbol{\a}}$ is a product homogeneous Bernoulli measure with density $\rho = \frac{1 - \sqrt{1 - 4a}}{2}$. 
            
    \end{remark}
            {The \(\boldsymbol{\alpha}\)-parametrization is natural from the viewpoint of
canonical measures. On the other hand, for some explicit computations, it is
more convenient to use another parametrization introduced in \cite{FG}. We
therefore recall the \(\mathbf{q}\)-parametrization below.}
         We define parameter sets $\mathcal{Q}, \mathcal{Q}^{+}$ as
                \begin{align}
                    \mathcal{Q} &:= \left\{ \q = \left( q_k \right)_{k \in \N} \subset  [0,1)^{\N} \ ; \  \sum_{k \in \N} q_k < \infty \right\}, \\
                    \mathcal{Q}^{+} &:= \left\{ \q = \left( q_k \right)_{k \in \N} \in \mathcal{Q} \ ; \  \sum_{k \in \N} k q_k < \infty \right\}. 
                \end{align}
            For any ${\q} \in \mathcal{Q}$, we define a probability measure $\fa_{{\q}}$ on $\mathcal{E}$ as 
                \begin{align}
                    \fa_{\q}\left(\e\right) := \prod_{k \in \N} q_{k}^{\zeta_{k}\left(\e\right)} \left(1 - q_k\right)^{2\sum_{\ell \in \N} \ell \zeta_{k+ \ell}\left(\e\right) }. 
                \end{align}
            Then, the one-to-one correspondence between $\nu_{\boldsymbol{\a}}$ and $\fa_\q$ has been shown by \cite{FG}. To state their result, for any $\boldsymbol{\a} \in \mathcal{A}$ and $\q \in \mathcal{Q}$, we define $\q(\boldsymbol{\a})$ and $\boldsymbol{\a}(\q)$ as 
                \begin{align}
                    q\left(\boldsymbol{\a}\right)_{k} &:= \begin{dcases}
                        \a_1 \ & \ k = 1, \\
                        \frac{\a_{k}}{\prod_{\ell = 1}^{k-1}\left(1 - q_\ell(\a)\right)^{2(k-\ell)}} \ & \ k \ge 2,
                    \end{dcases}\label{def:atoq} \\
                    \a\left(\q\right)_k &:= q_{k}\prod_{\ell = 1}^{k-1}\left(1 - q_\ell\right)^{2(k-\ell)}.\label{def:qtoa}
                \end{align}
                \begin{theorem*}[Theorem 3.1 in \cite{FG}]
                    The maps \eqref{def:atoq} and \eqref{def:qtoa} are the inverse of each other. Moreover, we have 
                        \begin{align}
                            \q\left(\mathcal{A}\right) = \mathcal{Q}, \quad \boldsymbol{\a}\left(\mathcal{Q}\right) = \mathcal{A},
                        \end{align}
                    and 
                        \begin{align}\label{eq:Q+A+}
                            \q\left(\mathcal{A}^{+}\right) = \mathcal{Q}^{+}, \quad \boldsymbol{\a}\left(\mathcal{Q}^{+}\right) = \mathcal{A}^{+}.
                        \end{align}
                    In particular, for any $\q \in \mathcal{Q}$, we have 
                        \begin{align}\label{eq:rel_aq}
                            \fa_\q = \nu_{\boldsymbol{\a}\left(\q\right)}.
                        \end{align}
                \end{theorem*}
        Lemma \ref{lem:shift_excursion_m} can be written in terms of $\fa_{\q}$. 
        We define a shift operator $\tilde{\theta} : \mathcal{Q} \to \mathcal{Q}$ as $\left(\tilde{\theta}{\q}\right)_{k} := q_{k+1}$, $k \in \N$.
            \begin{lemma}\label{lem:shift_q}
                Suppose that $q \in \mathcal{Q}$. Then, for any $k \in \N$ and $\e' \in \mathcal{E}$, we have 
                    \begin{align}\label{eq:shift_q}
                        \fa_\q\left( \left\{ \e \in \mathcal{E} \ ; \ \Psi_{k}\left(\e\right) = \e' \right\} \right) = \fa_{\tilde{\theta}^{k}\q}\left(\e'\right).
                    \end{align}
            \end{lemma}
            \begin{proof}[Proof of Lemma \ref{lem:shift_q}]
                We observe that for any $k \in \N$,
                    \begin{align}
                        \a\left(\tilde{\theta}\q\right)_k = q_{k+1}\prod_{\ell = 1}^{k-1}\left(1 - q_{\ell+1}\right)^{2(k - \ell)}
                        = \frac{\a\left(\q\right)_{k+1}}{\left(1 - \a\left(\q\right)_1\right)^{2k}}
                        = \left(\theta\a\left(\q\right)\right)_k.
                    \end{align}
                Hence we have 
                    \begin{align}\label{eq:theta-tilde}
                        \boldsymbol{\a}\left(\tilde{\theta}\q\right) = \theta\boldsymbol{\a}\left(\q\right).
                    \end{align}
                By using this relation $k$ times, we get $\boldsymbol{\a}\left(\tilde{\theta}^{k}\q\right) = \theta^{k}\boldsymbol{\a}\left(\q\right)$ for any $k \in \N$. Thus from \eqref{eq:shift_excursion_m} and \eqref{eq:rel_aq}, we obtain 
                    \begin{align}
                        \fa_\q\left( \left\{ \e \in \mathcal{E} \ ; \ \Psi_{k}\left(\e\right) = \e' \right\} \right) 
                        &= \nu_{\boldsymbol{\a}\left(\q\right)}\left( \left\{ \e \in \mathcal{E} \ ; \ \Psi_{k}\left(\e\right) = \e' \right\} \right) \\
                        &= \nu_{\theta^k\boldsymbol{\a}\left(\q\right)}\left(\e'\right) \\
                        &= \nu_{\boldsymbol{\a}\left(\tilde{\theta}^k\q\right)}\left(\e'\right) \\
                        &= \fa_{\tilde{\theta}^k\q}\left(\e'\right).
                    \end{align}
            \end{proof}
        \begin{remark}\label{rem:theta_A+}
            It is clear that $\tilde{\theta}|_{\mathcal{Q}^{+}} : \mathcal{Q}^{+} \to \mathcal{Q}^{+}$. Hence, from \cite[Theorem 3.1]{FG} and \eqref{eq:theta-tilde}, if $\boldsymbol{\a} \in \mathcal{A}^{+}$, then $\tilde{\theta}\q\left(\boldsymbol{\a}\right) \in \mathcal{Q}^{+}$, and $\theta \boldsymbol{\a} = \theta \boldsymbol{\a}\left( \q\left(\boldsymbol{\a}\right) \right) = \boldsymbol{\a}\left( \tilde{\theta}\q\left(\boldsymbol{\a}\right) \right) \in \mathcal{A}^{+}$.
        \end{remark}
                
        We note that $\fa_{\q}$ can be extended over $\mathcal{E}^{\Z}$ and $\Omega_{*}$ in the same way as we did for $\nu_{\boldsymbol{\a}}$. 
        We denote by $\hat{\phi}_{\q}$, $\q \in \mathcal{Q}^{+}$, the probability measure on $\Omega_{r}$ induced by the product probability measure $\prod_{\e_{i} \in \Z}\fa_{\q}\left(\e_i\right), e_i \in \mathcal{E}$ on $\mathcal{E}^{\Z}$ via the map $I$. In addition, for any $\q \in \mathcal{Q}^{+}$, we define a probability measure $\phi_{\q}$ on $\Omega_{*}$ as
                \begin{align}
                    \int_{\Omega_{*}} d\phi_{\q}\left(\eta\right) f(\eta) := \frac{1}{\bar{\e}\left(\q\right)} \int_{\Omega_{r}} d\hat{\phi}_{\q}\left(\eta\right) \sum_{y = 0}^{s_{\infty}\left(\eta, 1\right) - 1} \tau_{y} f \left(\eta\right), 
                \end{align}
            for any local function $f : \{0,1\}^{\Z} \to \R$, where $\bar{\e}\left(\q\right) := \bar{\e}\left(\boldsymbol{\a}\left(\q\right)\right).$ We note that from \eqref{eq:rel_aq}, we have 
                \begin{align}
                    \phi_{\q} = \mu_{\boldsymbol{\a}(\q)},\label{eq:rel_aq_O}
                \end{align}
            for any $\q \in \mathcal{Q}^{+}$. 

            \begin{remark}\label{rem:indep}
                It is known that under $\hat{\mu}_{\boldsymbol{\a}}$, $\a \in \mathcal{A}$, the sequence $\left(\zeta_{k}\left(i\right)\right)_{k \in \N, i \in \Z}$ is i.i.d. geometric with $\hat{\mu}_{\boldsymbol{\a}}\left(\zeta_{k}\left(i\right) = 1\right) = q\left(\boldsymbol{\a}\right)_{k}$ for each $k$ and independent over $k$. See \cite[Theorem 4.3 a)]{FG} for details. 
            \end{remark}
            
        \subsection{Distribution of \texorpdfstring{$k$}{k}-skipped configuration}\label{subsubsec:main}
            
            In this subsection we consider the distribution of $\Psi_{k}\left(\eta\right)$ under $\mu_{\boldsymbol{\a}}$, $\boldsymbol{\a} \in \mathcal{A}$. First we prepare the following lemma.  
                \begin{lemma}\label{lem:indep_psi}
                    Suppose that $\boldsymbol{\a} \in \mathcal{A}$ and $k \in \N$. Then, under $\hat{\mu}_{\boldsymbol{\a}}$, $\Psi_{k}\left(\eta\right)$ is independent of  $\left(\zeta_{\ell}\left(\eta, i\right)\right)_{1 \le \ell \le k, i \in \Z}$.
                \end{lemma}
                \begin{proof}[Proof of Lemma \ref{lem:indep_psi}]
                    By Remark \ref{rem:two_bi}, $\eta \in \Omega_{r}$ is a function of $\left(\zeta_{\ell}\left(\eta, i\right)\right)_{\ell \in \N, i \in \Z}$. Combining this fact with \eqref{whole_3}, we see that if $\eta \in \Omega_{r}$, then $\Psi_{k}\left(\eta\right)$ is a function of $\left(\zeta_{k+\ell}\left(\eta, i\right)\right)_{\ell \in \N, i \in \Z}$. Thus, from Remark \ref{rem:indep}, the claim of this lemma is shown. 
                \end{proof}

            Now we consider the expectation of  $f\left(\Psi_{k}\left(\eta\right)\right)$, where $f$ is a local function on $\{0,1\}^{\Z}$. For any $k \in \N$, we define a cutoff operator $C_{k} : \mathcal{A} \to \mathcal{A}$ as 
                \begin{align}
                    \left(C_{k}\a\right)_{\ell} := 
                    \begin{dcases}
                        \a_\ell \ & \ 1 \le \ell \le k, \\
                        0 \ & \ \text{otherwise},
                    \end{dcases}
                \end{align}
            for any $\boldsymbol{\a} \in \mathcal{A}$. We recall that $\e\left(\boldsymbol{\a}\right)$ and $\theta : \mathcal{A} \to \mathcal{A}$ are defined in \eqref{def:exp_excursion} and \eqref{def:shift_para}, respectively. 
            
            {Before stating the main result of this subsection, let us explain its
meaning from the viewpoint of solitons. As explained in Remark
\ref{rem:skip_sol}, the $k$-skip map removes the first $k$ heads and the
first $k$ tails from each soliton. Therefore, a soliton of size $k+\ell$
in the original configuration is regarded as a soliton of size $\ell$ in
the $k$-skipped configuration, while solitons of size at most $k$
disappear. Thus, Theorem \ref{thm:skip_stat} below gives a probabilistic
description of the operation of lowering the height of each soliton by
$k$ under the invariant measures introduced in \cite{FG}.}
                \begin{theorem}\label{thm:skip_stat}
                    Suppose that $\boldsymbol{\a} \in \mathcal{A}^{+}$. Then, for any $k \in \N$ and local function $f : \{0,1\}^{\Z} \to \R$, we have
                        \begin{align}\label{eq:skip_stat_0}
                            \int_{\Omega_{*}} d\mu_{\boldsymbol{\a}}\left(\eta | s_{\infty}\left(\eta,0\right) = 0 \right) f\left( \Psi_{k}\left(\eta\right)\right) = \int_{\Omega_{*}} d\mu_{\theta^{k}\boldsymbol{\a}}\left(\eta | s_{\infty}\left(\eta,0\right) = 0 \right) f\left( \eta\right),
                        \end{align}
                    and 
                        \begin{align}
                            &\int_{\Omega_{*}} d\mu_{\boldsymbol{\a}}\left(\eta\right) f\left( \Psi_{k}\left(\eta\right)\right) \\ 
                            &= \frac{\bar{\e}\left(\theta^{k}\boldsymbol{\a}\right)\bar{\e}\left(C_{k}\boldsymbol{\a}\right)}{\bar{\e}\left(\boldsymbol{\a}\right)}\int_{\Omega_{*}} d\mu_{\theta^{k}\boldsymbol{\a}}\left(\eta\right) f\left( \eta\right) \\
                            & \quad + \frac{\bar{\e}\left(\theta^{k}\boldsymbol{\a}\right)\left(\bar{\e}\left(\boldsymbol{\a}\right) - \bar{\e}\left(\theta^{k}\boldsymbol{\a}\right)\bar{\e}\left(C_{k}\boldsymbol{\a}\right)\right)}{\bar{\e}\left(\boldsymbol{\a}\right)\left(\bar{\e}\left(\theta^{k}\boldsymbol{\a}\right) - \bar{\e}\left(\theta^{k+1}\boldsymbol{\a}\right)\right)}\int_{\Omega_{*}} d\mu_{\theta^{k}\boldsymbol{\a}}\left(\eta\right) \left(\sum_{\sigma \in \{\uparrow, \downarrow\}} \eta^{\sigma}_{1}\left(1\right)\right) f \left(\eta\right).
                            \label{eq:skip_stat}
                        \end{align}
                        
                \end{theorem}
    \begin{proof}[Proof of Theorem \ref{thm:skip_stat}]

        First we show \eqref{eq:skip_stat_0}. Since 
                        \begin{align}
                            \int_{\Omega_{*}} d\mu_{\boldsymbol{\a}}\left(\eta | s_{\infty}\left(\eta,0\right) = 0 \right) f(\eta) = \int_{\Omega_{r}} d\hat{\mu}_{\boldsymbol{\a}}\left(\eta\right)  f \left(\eta\right),
                        \end{align}
        by using \eqref{eq:shift_excursion_m} and \eqref{eq:exskip_etaskip}, we have \eqref{eq:skip_stat_0}.

        Next we show \eqref{eq:skip_stat}. We fix a local function $f$. We observe that for any $\eta \in \Omega_{r}$, $k \in \N$,  $0 \le y \le s_{\infty}\left(\eta, 1\right) - 1$ and $x \in \Z$,
            \begin{align}
                \xi_{k}\left( \tau_{y}\eta, x \right) - \xi_{k}\left( \tau_{y}\eta, x - 1 \right)
                &= r\left(\tau_{y}\eta,x\right) +  \sum_{\ell \in \N} \sum_{\sigma \in \{ \uparrow, \downarrow \}} \tau_{y}\eta^{\sigma}_{k + \ell}\left(x\right) \\
                &= r\left(\eta,x + y\right) +  \sum_{\ell \in \N} \sum_{\sigma \in \{ \uparrow, \downarrow \}} \eta^{\sigma}_{k + \ell}\left(x + y\right) \\
                &= \xi_{k}\left(\eta, x + y \right) - \xi_{k}\left(\eta, x + y - 1 \right),
            \end{align}
        and $s_{\infty}\left(\tau_y \eta, 0\right) = -y$. Hence, we get 
            \begin{align}
                \xi_{k}\left( \tau_{y}\eta, x \right) = \xi_{k}\left( \eta, x + y \right),
            \end{align}
        and 
            \begin{align}
                s_{k}\left(\tau_y \eta, x\right) = s_{k}\left( \eta, x\right) - y.
            \end{align}
        Thus we obtain 
            \begin{align}
                \Psi_{k}\left(\tau_y \eta\right)\left(x\right) 
                &= \tau_y \eta\left(s_{k}\left(\tau_y \eta, x + \xi_{k}\left(\tau_y\eta, 0\right)\right)\right) \\
                &= \eta\left(s_{k}\left(\eta, x + \xi_{k}\left(\eta, y\right)\right)\right) \label{eq:skip_stat_2}.
            \end{align}
        By using \eqref{eq:skip_stat_2}, we have 
            \begin{align}
                &\int_{\Omega_{*}} d\mu_{\boldsymbol{\a}}\left(\eta\right) f\left( \Psi_{1}\left(\eta\right)\right) \\
                &= \frac{1}{\bar{\e}\left(\boldsymbol{\a}\right)} \int_{\Omega_{r}} d\hat{\mu}_{\boldsymbol{\a}}\left(\eta\right) \sum_{y = 0}^{s_{\infty}\left(\eta, 1\right) - 1} f \left(\eta\left( s_{1}\left(\eta, \cdot + \xi_{1}\left( \eta, y \right)\right) \right)\right) \\
                &= \frac{1}{\bar{\e}\left(\boldsymbol{\a}\right)} \int_{\Omega_{r}} d\hat{\mu}_{\boldsymbol{\a}}\left(\eta\right) \sum_{z = 0}^{\xi_{1}\left( \eta, s_{\infty}\left(\eta, 1\right) \right) - 1} \sum_{y = s_{1}\left(\eta, z \right)}^{s_{1}\left(\eta, z + 1 \right) - 1} f \left(\eta\left( s_{1}\left(\eta, \cdot + \xi_{1}\left(\eta, y\right) \right) \right)\right)  \\
                &= \frac{1}{\bar{\e}\left(\boldsymbol{\a}\right)} \int_{\Omega_{r}} d\hat{\mu}_{\boldsymbol{\a}}\left(\eta\right) \sum_{z = 0}^{\xi_{1}\left( \eta, s_{\infty}\left(\eta, 1\right) \right) - 1} \left( s_{1}\left(\eta, z + 1 \right) - s_{1}\left(\eta, z \right) \right) f \left( \tau_z \Psi_{1}\left(\eta\right) \right),
            \end{align}
        where at the last line we use the equality $\xi_{1}\left(\eta, y\right) = \xi_{1}\left(\eta, s_{1}\left(\eta, x\right)\right) = x$ for any $x \in \Z$ and $s_{1}\left(\eta, x\right) \le y \le s_{1}\left(\eta, x + 1\right) - 1$. We then observe that from \eqref{def:zeta} and \eqref{eq:whole_1},
            \begin{align}
                s_{1}\left(\eta, z + 1 \right) - s_{1}\left(\eta, z \right) &= 1 + \sum_{y = s_{1}\left(\eta, z \right) + 1}^{s_{1}\left(\eta, z + 1 \right) - 1} \sum_{\sigma \in \{\uparrow, \downarrow\}} \eta^{\sigma}_{1}\left(y\right) \\
                &= 1 + 2\zeta_{1}\left(\eta,z\right) + \sum_{\sigma \in \{\uparrow, \downarrow\}} \eta^{\sigma}_{2}\left(s_{1}\left(\eta, z + 1 \right)\right) \\
                &= 1 + 2\zeta_{1}\left(\eta,z\right) + \sum_{\sigma \in \{\uparrow, \downarrow\}} \Psi_{1}\left(\eta\right)^{\sigma}_{1}\left(z+1\right).
            \end{align}
        In addition, if $\eta \in \Omega_{r}$, then from \eqref{eq:whole_p2**} with $k = \infty$, $\ell = 1$ and $x = 1$, $\xi_{1}\left( \eta, s_{\infty}\left(\eta, 1\right) \right) = s_{\infty}\left(\Psi_{1}\left(\eta\right), 1\right)$. 
        By using the above observation, Remark \ref{rem:theta_A+}, Lemma \ref{lem:indep_psi} and \eqref{eq:skip_stat_0}, we get 
            \begin{align}
                &\bar{\e}\left(\boldsymbol{\a}\right)\int_{\Omega_{*}} d\mu_{\boldsymbol{\a}}\left(\eta\right) f\left( \Psi_{1}\left(\eta\right)\right) \\
                &= \int_{\Omega_{r}} d\hat{\mu}_{\boldsymbol{\a}}\left(\eta\right) \sum_{z = 0}^{\xi_{1}\left( \eta, s_{\infty}\left(\eta, 1\right) \right) - 1} \left( s_{1}\left(\eta, z + 1 \right) - s_{1}\left(\eta, z \right) \right) f \left( \tau_z \Psi_{1}\left(\eta\right) \right) \\
                &= \left(\frac{2\a_1}{1-\a_1} + 1 \right) \int_{\Omega_{r}} d\hat{\mu}_{\boldsymbol{\a}}\left(\eta\right) \sum_{z = 0}^{s_{\infty}\left(\Psi_{1}\left(\eta\right), 1\right) - 1} f \left( \tau_z \Psi_{1}\left(\eta\right) \right) \\
                & \quad + \int_{\Omega_{r}} d\hat{\mu}_{\boldsymbol{\a}}\left(\eta\right) \sum_{z = 0}^{s_{\infty}\left(\Psi_{1}\left(\eta\right), 1\right) - 1} \left( \sum_{\sigma \in \{\uparrow, \downarrow\}} \Psi_{1}\left(\eta\right)^{\sigma}_{1}\left(z+1\right) \right) f \left( \tau_z \Psi_{1}\left(\eta\right) \right) \\
                &= \bar{\e}\left(\theta \boldsymbol{\a}\right)\left(\frac{2\a_1}{1-\a_1} + 1 \right) \int_{\Omega_{*}} d\mu_{\theta\boldsymbol{\a}}\left(\eta\right) f \left(\eta\right) \\
                &\quad + \bar{\e}\left(\theta \boldsymbol{\a}\right) \int_{\Omega_{*}} d\mu_{\theta\boldsymbol{\a}}\left(\eta\right) \left(\sum_{\sigma \in \{\uparrow, \downarrow\}} \eta^{\sigma}_{1}\left(1\right)\right) f \left(\eta\right) \label{eq:skip_stat_3}.
            \end{align}
        By a similar argument used above, we also get 
            \begin{align}
                &\bar{\e}\left(\boldsymbol{\a}\right) \int_{\Omega_{*}} d\mu_{\boldsymbol{\a}}\left(\eta\right) \left(\sum_{\sigma \in \{\uparrow, \downarrow\}} \eta^{\sigma}_{1}\left(1\right)\right) f \left(\Psi_{1}\left(\eta\right)\right) \\
                &=\int_{\Omega_{r}} d\hat{\mu}_{\boldsymbol{\a}}\left(\eta\right) \sum_{z = 0}^{s_{\infty}\left(\Psi_{1}\left(\eta\right), 1\right) - 1} \sum_{y = s_{1}\left(\eta, z \right)}^{s_{1}\left(\eta, z + 1 \right) - 1} \left( \sum_{\sigma \in \{\uparrow, \downarrow\}} \eta^{\sigma}_{1}\left(y+1\right) \right) f \left(\eta\left( s_{1}\left(\eta, \cdot + \xi_{1}\left(\eta, y\right) \right) \right)\right) \\
                &= \int_{\Omega_{r}} d\hat{\mu}_{\boldsymbol{\a}}\left(\eta\right) \sum_{z = 0}^{s_{\infty}\left(\Psi_{1}\left(\eta\right), 1\right) - 1} \left( s_{1}\left(\eta, z + 1 \right) - s_{1}\left(\eta, z \right) - 1 \right) f \left( \tau_z \Psi_{1}\left(\eta\right) \right) \\
                &= \frac{2\bar{\e}\left(\theta \boldsymbol{\a}\right) \a_1}{1 - \a_1} \int_{\Omega_{*}} d\mu_{\theta\boldsymbol{\a}}\left(\eta\right) f \left(\eta\right) + \bar{\e}\left(\theta \boldsymbol{\a}\right) \int_{\Omega_{*}} d\mu_{\theta\boldsymbol{\a}}\left(\eta\right) \left(\sum_{\sigma \in \{\uparrow, \downarrow\}} \eta^{\sigma}_{1}\left(1\right)\right) f \left(\eta\right) \label{eq:skip_stat_4}.
            \end{align}
        Now for any $k \in \N$ and $0 \le \ell \le k$, we define $a_{k-\ell, \ell}$, $b_{k-\ell, \ell}$ as 
            \begin{align}
                a_{k-\ell, \ell} &:=  \bar{\e}\left(\theta^{\ell}\boldsymbol{\a}\right)\int_{\Omega_{*}} d\mu_{\theta^{\ell}\boldsymbol{\a}}\left(\eta\right) f\left( \Psi_{k-\ell}\left(\eta\right)\right), \\
                b_{k-\ell, \ell} &:=  \bar{\e}\left(\theta^{\ell}\boldsymbol{\a}\right)\int_{\Omega_{*}} d\mu_{\theta^{\ell}\boldsymbol{\a}}\left(\eta\right) \left(\sum_{\sigma \in \{\uparrow, \downarrow\}} \eta^{\sigma}_{1}\left(1\right)\right) f\left( \Psi_{k-\ell}\left(\eta\right)\right).
            \end{align}
        From \eqref{whole_2}, \eqref{eq:skip_stat_3} and \eqref{eq:skip_stat_4}, for any $0 \le \ell \le k - 1$, we have 
            \begin{align}
                a_{k-\ell, \ell} = \left(\frac{2\left(\theta^{\ell}\a\right)_1}{1-\left(\theta^{\ell}\a\right)_1} + 1 \right) a_{k-\ell - 1, \ell+1} + b_{k-\ell-1, \ell+1} \label{eq:skip_stat_5},
            \end{align}
        and 
            \begin{align}
                b_{k-\ell, \ell} = \frac{2\left(\theta^{\ell}\a\right)_1}{1-\left(\theta^{\ell}\a\right)_1} a_{k-\ell - 1, \ell+1} + b_{k-\ell-1, \ell+1} \label{eq:skip_stat_6}.
            \end{align}
        From \eqref{eq:skip_stat_5} and \eqref{eq:skip_stat_6}, we see that for any $0 \le \ell \le k$, there exists some positive constants $c_{\ell}, d_{\ell}$ such that 
            \begin{align}
                a_{k-\ell, \ell} = c_{\ell} a_{0, k} + d_{\ell} b_{0, k} \label{eq:skip_stat_7}.
            \end{align}
        We note that $c_{\ell}, d_{\ell}$ are independent of $f$. 
        Thus, by choosing $f \equiv 1$, we have 
            \begin{align}
                \bar{\e}\left(\theta^{\ell}\boldsymbol{\a}\right) = c_{\ell} \bar{\e}\left(\theta^{k}\boldsymbol{\a}\right)  + d_{\ell}\left(\bar{\e}\left(\theta^{k}\boldsymbol{\a}\right) - \bar{\e}\left(\theta^{k+1}\boldsymbol{\a}\right) \right) \label{eq:skip_stat_8},
            \end{align}
        where $b_{0, k}$ with $f \equiv 1$ can be computed by using \eqref{def:zeta_excursion}, \eqref{eq:length_z}, \eqref{eq:shift_excursion} and \eqref{eq:shift_excursion_m} as follows : for each $\sigma \in \{\uparrow, \downarrow\}$,
            \begin{align}
                \int_{\Omega_{r}} d\hat{\mu}_{\theta^{k}\boldsymbol{\a}}\left(\eta\right) \sum_{z = 0}^{s_{\infty}\left(\eta, 1\right) - 1}  \eta^{\sigma}_{1}\left(z+1\right)
                &= \int_{\mathcal{E}} d\nu_{\theta^{k}\boldsymbol{\a}}\left(\e\right) \sum_{k \in \N} \zeta_{k}\left(\e\right) \\
                &= \frac{1}{2}\int_{\mathcal{E}} d\nu_{\theta^{k}\boldsymbol{\a}}\left(\e\right) \left( \left| \e \right| - \left| \Psi_{1}\left(\e\right) \right| \right) \\
                &= \frac{\bar{\e}\left(\theta^{k}\boldsymbol{\a}\right) - \bar{\e}\left(\theta^{k+1}\boldsymbol{\a}\right)}{2}. \label{eq:number_sol}
            \end{align}
        On the other hand, from \eqref{eq:skip_stat_6}, we get 
            \begin{align}\label{eq:skip_stat_13}
                b_{k-\ell, \ell} = b_{0, k} +  2\sum_{h = \ell}^{k-1} \frac{\left(\theta^{h}\a\right)_1}{1-\left(\theta^{h}\a\right)_1} a_{k-h-1,h+1},
            \end{align}
        and substituting this for \eqref{eq:skip_stat_5}, we obtain 
            \begin{align}
                a_{k-\ell,\ell} = a_{k-\ell - 1,\ell+1} + b_{0, k} + 2\sum_{h = \ell}^{k-1} \frac{\left(\theta^{h}\a\right)_1}{1-\left(\theta^{h}\a\right)_1} a_{k-h-1,h+1} \label{eq:skip_stat_9}.
            \end{align}
        By substituting \eqref{eq:skip_stat_7} for \eqref{eq:skip_stat_9}, the following recursive relations for $c_{\ell}, d_{\ell}$ are derived : 
            \begin{align}\label{eq:skip_stat_10}
                c_{\ell} &= c_{\ell+1} + 2\sum_{h = \ell}^{k-1} \frac{\left(\theta^{\ell}\a\right)_1}{1-\left(\theta^{\ell}\a\right)_1} c_{h+1}, \\
                d_{\ell} &= 1 + d_{\ell+1} + 2\sum_{h = \ell}^{k-1} \frac{\left(\theta^{\ell}\a\right)_1}{1-\left(\theta^{\ell}\a\right)_1} d_{h+1}. \label{eq:skip_stat_12}
            \end{align}
        In particular, the sequences $c_{\ell},d_{\ell}$, $0 \le \ell \le k$ satisfy
            \begin{align}
                \begin{dcases}\label{eq:system_c}
                    c_{\ell-1} + c_{\ell+1} - 2\left(\frac{\left(\theta^{\ell-1}\a\right)_1}{1-\left(\theta^{\ell-1}\a\right)_1} + 1\right) c_{\ell} = 0, \quad 1\le \ell \le k - 1, \\
                    c_{k} = 1, \quad c_{k-1} = \frac{2\left(\theta^{k-1}\a\right)_1}{1-\left(\theta^{k-1}\a\right)_1} + 1,
                \end{dcases}
            \end{align}
        and 
            \begin{align}
                \begin{dcases}\label{eq:system_d}
                    d_{\ell-1} + d_{\ell+1} - 2\left(\frac{\left(\theta^{\ell-1}\a\right)_1}{1-\left(\theta^{\ell-1}\a\right)_1} + 1\right) d_{\ell} = 0, \quad 1\le \ell \le k - 1, \\
                    d_{k} = 0, \quad d_{k-1} = 1.
                \end{dcases}
            \end{align}
        Now we claim that 
            \begin{align}
                c_{\ell} = \bar{\e}\left(\theta^{\ell} C_{k}\boldsymbol{\a}\right) \label{eq:skip_stat_11},
            \end{align}
        for any $0 \le \ell \le k$. From  \eqref{eq:whole_p2**}, \eqref{eq:length_z} and \eqref{eq:shift_excursion}, we get 
            \begin{align}\label{eq:skip_stat_14}
                \left| \Psi_{\ell-1}\left(\e\right) \right| + \left| \Psi_{\ell+1}\left(\e\right) \right| - 2\left| \Psi_{\ell}\left(\e\right) \right| = 2\sum_{i = 0}^{\left| \Psi_{\ell}\left(\e\right) \right|-1} \zeta_{\ell}\left( \iota\left(\e\right),i \right).
            \end{align}
        By integrating both sides with respect to $\nu_{C_{k}\boldsymbol{\a}}$ and using Remark \ref{rem:indep}, we see that the sequence $\bar{\e}\left(\theta^{\ell} C_{k}\boldsymbol{\a}\right)$, $0 \le \ell \le k$ also satisfies 
            \begin{align}\label{eq:system_Ce}
                \begin{dcases}
                    \bar{\e}\left(\theta^{\ell-1} C_{k}\boldsymbol{\a}\right) + \bar{\e}\left(\theta^{\ell+1} C_{k}\boldsymbol{\a}\right) - 2\left(\frac{\left(\theta^{\ell-1}\a\right)_1}{1-\left(\theta^{\ell-1}\a\right)_1} + 1\right) \bar{\e}\left(\theta^{\ell} C_{k}\boldsymbol{\a}\right) = 0, \quad 1\le \ell \le k - 1, \\
                    \bar{\e}\left(\theta^{k} C_{k}\boldsymbol{\a}\right) = 1, \quad \bar{\e}\left(\theta^{k-1} C_{k}\boldsymbol{\a}\right) = \frac{2\left(\theta^{k-1}\a\right)_1}{1-\left(\theta^{k-1}\a\right)_1} + 1.
                \end{dcases}
            \end{align}
        Since the above system is the same as \eqref{eq:system_c}, we have \eqref{eq:skip_stat_11}. Combining \eqref{eq:skip_stat_7}, \eqref{eq:skip_stat_8} and \eqref{eq:skip_stat_11} with $\ell = 0$ yields \eqref{eq:skip_stat}.
        
                \end{proof}

        The coefficients appearing in \eqref{eq:skip_stat} are expectations of the size of the excursion with various
        parameters. 
        Thanks to \eqref{eq:skip_stat_14}, we can find formulae of $\bar{\e}\left(\boldsymbol{\a}\right)$ via continued fractions. For any $\boldsymbol{\a} \in \mathcal{A}^{+}$ and $k \in \N$, we define a continued fraction $F\left(\boldsymbol{\a}\right)$ as 
            \begin{align}
                F\left(\boldsymbol{\a}\right) := \frac{2}{1 - \a_1} - \cfrac{1}{\cfrac{2}{1-\left(\theta\a\right)_1} - \cfrac{1}{\cfrac{2}{1-\left(\theta^2\a\right)_1} - \cdots }}.
            \end{align}
        Then we have the following. 
            \begin{proposition}\label{prop:frac}
                Suppose that $\boldsymbol{\a} \in \mathcal{A}^{+}$. Then we have  
                    \begin{align}\label{eq:frac_F}
                        \bar{\e}\left(\boldsymbol{\a}\right) = \prod_{\ell = 0}^{\infty} F\left(\theta^{\ell} \boldsymbol{\a} \right).
                    \end{align}
            \end{proposition}
            \begin{proof}[Proof of Proposition \ref{prop:frac}]
                First we show that if $\boldsymbol{\a} \in \mathcal{A}^{+}$, then 
        \begin{align}\label{eq:app_1}
            \lim_{\ell \to \infty} \bar{\e}\left(\theta^{\ell}\boldsymbol{\a}\right) = 1.
        \end{align}
    We recall that $\q\left(\boldsymbol{\a}\right)$, $\boldsymbol{\a} \in \mathcal{A}$ is defined in \eqref{def:atoq}, and from \eqref{eq:theta-tilde}, for any $\ell \in \N$ we have 
        \begin{align}
            \left(\theta^{\ell-1}\a\right)_{1} = q\left(\boldsymbol{\a}\right)_{\ell}. 
        \end{align}
    From the above, \eqref{eq:length_z}, \eqref{eq:Q+A+}, Remark \ref{rem:indep}, we see that if $\boldsymbol{\a} \in \mathcal{A}^{+}$, then 
        \begin{align}
            \bar{\e}\left(\theta^{\ell}\boldsymbol{\a}\right) - 1 &=  2 \sum_{h = \ell + 1}^{\infty} \frac{\left(h - \ell\right) q\left(\boldsymbol{\a}\right)_h \bar{\e}\left(\theta^{h}\boldsymbol{\a}\right)}{1 - q\left(\boldsymbol{\a}\right)_h}  \\
            &\le 2 \bar{\e}\left(\boldsymbol{\a}\right) \sum_{h = \ell + 1}^{\infty} \frac{\left(h - \ell\right) q\left(\boldsymbol{\a}\right)_h}{1 - q\left(\boldsymbol{\a}\right)_h} \\
            & \to 0, \quad \ell \to \infty,
        \end{align}
    where at the last inequality we use the fact that $\bar{\e}\left(\theta^{\ell}\boldsymbol{\a}\right)$, $\ell \in \N$ is a decreasing sequence. Hence we have \eqref{eq:app_1}. 
    
    Now we derive \eqref{eq:frac_F}. By integrating both sides of \eqref{eq:skip_stat_14} with respect to $\nu_{\boldsymbol{\a}}$ and then dividing both sides by $\bar{\e}\left(\theta^{\ell}\boldsymbol{\a}\right)$ we have 
        \begin{align}
            \frac{\bar{\e}\left(\theta^{\ell}\boldsymbol{\a}\right)}{\bar{\e}\left(\theta^{\ell+1}\boldsymbol{\a}\right)} &= \frac{2}{1 - \left(\theta^{\ell}\a\right)_1} - \frac{\bar{\e}\left(\theta^{\ell+2}\boldsymbol{\a}\right)}{\bar{\e}\left(\theta^{\ell+1}\boldsymbol{\a}\right)} \\
            &= \frac{2}{1 - \left(\theta^{\ell}\a\right)_1} - \cfrac{1}{\cfrac{2}{1-\left(\theta^{\ell+1}\a\right)_1} - \frac{\bar{\e}\left(\theta^{\ell+3}\boldsymbol{\a}\right)}{\bar{\e}\left(\theta^{\ell+2}\boldsymbol{\a}\right)}} \\
            &=\dots = F\left(\theta^{\ell} \boldsymbol{\a}\right).
        \end{align}
    Thanks to \eqref{eq:app_1} and the monotonicity of $\bar{\e}\left(\theta^{\ell}\boldsymbol{\a}\right)$, $\ell \in \N$, we get
        \begin{align}
            \bar{\e}\left(\boldsymbol{\a}\right) = \prod_{\ell = 0}^{\infty} \frac{\bar{\e}\left(\theta^{\ell}\boldsymbol{\a}\right)}{\bar{\e}\left(\theta^{\ell+1}\boldsymbol{\a}\right)} = \prod_{\ell = 0}^{\infty} F\left(\theta^{\ell} \boldsymbol{\a}\right). 
        \end{align}
            \end{proof}

        By a similar strategy used in the proof of Theorem \ref{thm:skip_stat}, we can compute the expectation of $f\left(\Psi_{k}\left(\eta\right)\right)$ with the condition on $\eta\left(0\right)$. {The following proposition will be used to compute the correlations between carriers with seat numbers in Section \ref{subsec:expectation}.}
            \begin{proposition}\label{prop:skip_stat}
                Suppose that $\boldsymbol{\a} \in \mathcal{A}^{+}$. Then, for any $k \in \N$ and local function $f : \{0,1\}^{\Z} \to \R$, we have
                        \begin{align}
                            &\int_{\Omega_{*}} d\mu_{\boldsymbol{\a}}\left(\eta\right) \eta\left(0\right) f\left( \Psi_{k}\left(\eta\right)\right) \\ 
                            &= \frac{\bar{\e}\left(\theta^{k}\boldsymbol{\a}\right)\left(\bar{\e}\left(C_{k}\boldsymbol{\a}\right) - 1\right)}{2\bar{\e}\left(\boldsymbol{\a}\right)}\int_{\Omega_{*}} d\mu_{\theta^{k}\boldsymbol{\a}}\left(\eta\right) f\left( \eta\right) + \frac{\bar{\e}\left(\theta^{k}\boldsymbol{\a}\right)}{\bar{\e}\left(\boldsymbol{\a}\right)}\int_{\Omega_{*}} d\mu_{\theta^{k}\boldsymbol{\a}}\left(\eta\right) \eta\left(0\right) f\left( \eta\right) \\
                            & \quad + \frac{\bar{\e}\left(\theta^{k}\boldsymbol{\a}\right)\left(\bar{\e}\left(\boldsymbol{\a}\right) - \bar{\e}\left(\theta^{k}\boldsymbol{\a}\right)\bar{\e}\left(C_{k}\boldsymbol{\a}\right) + k\right)}{2\bar{\e}\left(\boldsymbol{\a}\right)\left(\bar{\e}\left(\theta^{k}\boldsymbol{\a}\right) - \bar{\e}\left(\theta^{k+1}\boldsymbol{\a}\right)\right)}\int_{\Omega_{*}} d\mu_{\theta^{k}\boldsymbol{\a}}\left(\eta\right) \eta^{\uparrow}_{1}\left(1\right) f \left(\eta\right) \\
                            & \quad + \frac{\bar{\e}\left(\theta^{k}\boldsymbol{\a}\right)\left(\bar{\e}\left(\boldsymbol{\a}\right) - \bar{\e}\left(\theta^{k}\boldsymbol{\a}\right)\bar{\e}\left(C_{k}\boldsymbol{\a}\right) - k\right)}{2\bar{\e}\left(\boldsymbol{\a}\right)\left(\bar{\e}\left(\theta^{k}\boldsymbol{\a}\right) - \bar{\e}\left(\theta^{k+1}\boldsymbol{\a}\right)\right)}\int_{\Omega_{*}} d\mu_{\theta^{k}\boldsymbol{\a}}\left(\eta\right) \eta^{\downarrow}_{1}\left(1\right) f \left(\eta\right).
                            \label{eq:skip_stat_prop}
                        \end{align}
            \end{proposition}
        \begin{proof}[Proof of Proposition \ref{prop:skip_stat}]
            We fix a local function $f$. 
            First we observe that for any $\eta \in \Omega_{r}$ and $z \in \Z$, 
                \begin{align}
                    \sum_{y = s_{1}\left(\eta,z\right)}^{s_{1}\left(\eta,z+1\right)-1} \eta\left(y\right) 
                    &= \eta\left(s_{1}\left(\eta,z\right)\right) + \sum_{y = s_{1}\left(\eta,z\right) + 1}^{s_{1}\left(\eta,z+1\right)-1} \eta^{\uparrow}_{1}\left(y\right) \\
                    &= \Psi_{1}\left(\eta\right)\left(z\right) + \zeta_{1}\left(\eta,z\right) +  \Psi_{1}\left(\eta\right)^{\uparrow}_{1}\left(z+1\right).
                \end{align}
            Hence, by a similar computation used in the proof of Theorem \ref{thm:skip_stat}, for any $\a \in \mathcal{A}^{+}$ and $k \in \N$, we have 
                \begin{align}
                    &\bar{\e}\left(\boldsymbol{\a}\right)\int_{\Omega_{*}} d\mu_{\boldsymbol{\a}}\left(\eta\right) \eta\left(0\right) f\left( \Psi_{k}\left(\eta\right)\right) \\ 
                    &= \int_{\Omega_{r}} d\hat{\mu}_{\boldsymbol{\a}}\left(\eta\right) \sum_{z = 0}^{s_{\infty}\left(\Psi_{1}\left(\eta\right),1\right)-1} f\left(\tau_{z} \Psi_{k}\left(\eta\right)\right) \sum_{y = s_{1}\left(\eta,z\right)}^{s_{1}\left(\eta,z+1\right)-1} \eta\left(y\right) \\
                    &= \frac{\a_1 \bar{\e}\left(\theta\boldsymbol{\a}\right)}{1 - \a_1} \int_{\Omega_{*}} d\mu_{\theta\boldsymbol{\a}}\left(\eta\right) f\left( \Psi_{k-1}\left(\eta\right)\right) + \bar{\e}\left(\theta\boldsymbol{\a}\right)\int_{\Omega_{*}} d\mu_{\theta\boldsymbol{\a}}\left(\eta\right) \eta\left(0\right)f\left( \Psi_{k-1}\left(\eta\right)\right) \\
                    & \quad + \bar{\e}\left(\theta\boldsymbol{\a}\right)\int_{\Omega_{*}} d\mu_{\theta\boldsymbol{\a}}\left(\eta\right) \eta^{\uparrow}_{1}\left(1\right)f\left( \Psi_{k-1}\left(\eta\right)\right) . \label{eq:skip_stat_prop_0}
                \end{align}
            Also, for any $\a \in \mathcal{A}^{+}$, $k \in \N$ and $\sigma \in \{\uparrow, \downarrow\}$, we get 
                \begin{align}
                    &\bar{\e}\left(\boldsymbol{\a}\right)\int_{\Omega_{*}} d\mu_{\boldsymbol{\a}}\left(\eta\right) \eta^{\sigma}_{1}\left(1\right) f\left( \Psi_{k}\left(\eta\right)\right) \\
                    &= \int_{\Omega_{r}} d\hat{\mu}_{\boldsymbol{\a}}\left(\eta\right) \sum_{z = 0}^{s_{\infty}\left(\Psi_{1}\left(\eta\right),1\right)-1} f\left(\tau_{z} \Psi_{k}\left(\eta\right)\right) \sum_{y = s_{1}\left(\eta,z\right)}^{s_{1}\left(\eta,z+1\right)-1} \eta^{\sigma}_{1}\left(y+1\right) \\
                    &= \int_{\Omega_{r}} d\hat{\mu}_{\boldsymbol{\a}}\left(\eta\right) \sum_{z = 0}^{s_{\infty}\left(\Psi_{1}\left(\eta\right),1\right)-1} f\left(\tau_{z} \Psi_{k}\left(\eta\right)\right) \left( \zeta_{1}\left(\eta,z\right) + \Psi_{1}\left(\eta\right)^{\sigma}_{1}\left(z+1\right) \right) \\
                    &= \frac{\a_1 \bar{\e}\left(\theta\boldsymbol{\a}\right)}{1 - \a_1} \int_{\Omega_{*}} d\mu_{\theta\boldsymbol{\a}}\left(\eta\right) f\left( \Psi_{k-1}\left(\eta\right)\right) + \bar{\e}\left(\theta\boldsymbol{\a}\right)\int_{\Omega_{*}} d\mu_{\theta\boldsymbol{\a}}\left(\eta\right) \eta^{\sigma}_{1}\left(1\right)f\left( \Psi_{k-1}\left(\eta\right)\right) \label{eq:skip_stat_prop_1}.
                \end{align}
        From \eqref{eq:skip_stat_7}, \eqref{eq:skip_stat_10}, \eqref{eq:skip_stat_12} and \eqref{eq:skip_stat_prop_1}, for any $\a \in \mathcal{A}^{+}$, $k \in \N$ and $0 \le \ell \le k-1$, we have 
            \begin{align}
                &\bar{\e}\left(\theta^{\ell}\boldsymbol{\a}\right)\int_{\Omega_{*}} d\mu_{\theta^{\ell}\boldsymbol{\a}}\left(\eta\right) \eta^{\sigma}_{1}\left(1\right) f\left( \Psi_{k-\ell}\left(\eta\right)\right) - \bar{\e}\left(\theta^{k}\boldsymbol{\a}\right)\int_{\Omega_{*}} d\mu_{\theta^{k}\boldsymbol{\a}}\left(\eta\right) \eta^{\sigma}_{1}\left(1\right) f\left( \eta\right) \\
                &= \sum_{h = \ell}^{k-1} \frac{\left(\theta^{h}\a\right)_1}{1 - \left(\theta^{h}\a\right)_1} \left( c_{h+1} a_{0,k} + d_{h+1}b_{0,k} \right) \\
                &= \frac{\left(c_{\ell}-c_{\ell+1}\right)a_{0,k}}{2} + \frac{\left(d_{\ell}-d_{\ell+1} - 1\right)b_{0,k}}{2} \label{eq:skip_stat_prop_2},
            \end{align}
        where $a_{0,k}, b_{0,k}, c_{\ell}, d_{\ell}$ are defined in the proof of Theorem \ref{thm:skip_stat}. 
        Hence, by \eqref{eq:skip_stat_7}, \eqref{eq:skip_stat_13},  \eqref{eq:skip_stat_10}, \eqref{eq:skip_stat_12}, \eqref{eq:skip_stat_prop_0} and \eqref{eq:skip_stat_prop_2}, we obtain 
            \begin{align}
                &\bar{\e}\left(\boldsymbol{\a}\right)\int_{\Omega_{*}} d\mu_{\boldsymbol{\a}}\left(\eta\right) \eta\left(0\right) f\left( \Psi_{k}\left(\eta\right)\right) - \bar{\e}\left(\theta^{k}\boldsymbol{\a}\right)\int_{\Omega_{*}} d\mu_{\theta^{k}\boldsymbol{\a}}\left(\eta\right) \eta\left(0\right) f\left( \eta \right)  \\
                &= \sum_{h=0}^{k-1} \frac{\left(\theta^{h}\a\right)_1}{1 - \left(\theta^{h}\a\right)_1} \left( c_{h+1} a_{0,k} + d_{h+1}b_{0,k} \right) \\
                & \quad + \sum_{h' = 1}^{k-1} \frac{\left(c_{h'} - c_{h'+1}\right)a_{0,k}}{2} + \frac{\left(d_{h'} - d_{h'+1} - 1\right)b_{0,k}}{2} \\
                & \quad + k \bar{\e}\left(\theta^{k}\boldsymbol{\a}\right)\int_{\Omega_{*}} d\mu_{\theta^{k}\boldsymbol{\a}}\left(\eta\right) \eta^{\uparrow}_{1}\left(1\right) f\left( \eta\right) \\
                &= \frac{c_{0} - 1}{2} a_{0,k} + \frac{d_{0} - k}{2} b_{0,k} + k \bar{\e}\left(\theta^{k}\boldsymbol{\a}\right)\int_{\Omega_{*}} d\mu_{\theta^{k}\boldsymbol{\a}}\left(\eta\right) \eta^{\uparrow}_{1}\left(1\right) f\left( \eta\right).
            \end{align}
        Therefore \eqref{eq:skip_stat_prop} holds. 
            
        \end{proof}
                
        By using \eqref{eq:rel_aq_O}, we can rewrite Theorem \ref{thm:skip_stat} in terms of $\phi_\q$. For any $k \in \N$, we define $\tilde{C}_{k} : \mathcal{Q} \to \mathcal{Q}$ as 
            \begin{align}
                \left(\tilde{C}_{k}\q\right)_{\ell} := 
                \begin{dcases}
                    q_{\ell} \ & \ 1\le \ell \le k, \\
                    0 \ & \ \text{otherwise},
                \end{dcases}
            \end{align}
        for any $\q \in \mathcal{Q}$.
        We note that $C_{k}\boldsymbol{\a}\left(\q\right) =\boldsymbol{\a}\left(\tilde{C}_{k}\q\right)$.
        \begin{corollary}\label{cor:q_skip}
            Suppose that $q \in \mathcal{Q}^{+}$. Then, for any $k \in \N$ and local function $f : \{0,1\}^{\Z} \to \R$, we have
                \begin{align}
                            \int_{\Omega_*} d\phi_{\q}\left(\eta | s_{\infty}\left(\eta,0\right) = 0 \right) f\left( \Psi_{k}\left(\eta\right)\right) = \int_{\Omega_*} d\phi_{\tilde{\theta}^{k}\q}\left(\eta | s_{\infty}\left(\eta,0\right) = 0 \right) f\left( \eta\right),
                \end{align}
            and 
            \begin{align}
                            &\int_{\Omega_*} d\phi_{\q}\left(\eta\right) f\left( \Psi_{k}\left(\eta\right)\right) \\ 
                            &= \frac{\bar{\e}\left(\tilde{\theta}^{k}\q\right)\bar{\e}\left(\tilde{C}_{k}\q\right)}{\bar{\e}\left(\q\right)}\int_{\Omega_*} d\phi_{\tilde{\theta}^{k}\q}\left(\eta\right) f\left( \eta\right) \\
                            & \quad + \frac{\bar{\e}\left(\tilde{\theta}^{k}\q\right)\left(\bar{\e}\left(\q\right) - \bar{\e}\left(\tilde{\theta}^{k}\q\right)\bar{\e}\left(\tilde{C}_{k}\q\right)\right)}{\bar{\e}\left(\q\right)\left(\bar{\e}\left(\tilde{\theta}^{k}\q\right) - \bar{\e}\left(\tilde{\theta}^{k+1}\q\right)\right)}\int_{\Omega_*} d\phi_{\tilde{\theta}^{k}\q}\left(\eta\right) \left(\sum_{\sigma \in \{\uparrow, \downarrow\}} \eta^{\sigma}_{1}\left(1\right)\right) f \left(\eta\right).
            \end{align}
            
        \end{corollary}

\subsection{Two-sided Markov distribution case}\label{subsec:Markov}

    In this subsection we will show that if the law of $\left(\eta\left(x\right)\right)_{x \in \Z}$ is a two-sided homogeneous Markov chain conditioned on $\Omega_{r}$, then the law of $\left(\Psi_{1}\left(\eta\right)\left(x\right)\right)_{x \in \Z}$ is also a two-sided homogeneous Markov chain conditioned on $\Omega_{r}$. Before describing the precise statement, we prepare some notations and recall some facts on Markov chains on $\{0,1\}$. 
    In the following discussion, we denote the transition matrix of a given two-sided homogeneous Markov chain $\left(\eta(x)\right)_{x\in\mathbb{Z}}$ by $P=\left(p(r,s)\right)_{r,s=0,1}$, where
        \begin{align}
                    p(r,s) := \mathbb{P}\left( \eta(1) = s | \eta(0) = r \right),  
        \end{align}
    and assume that { $0 < p(0,1) < p(1,0)$}. 
    If we define $a, b$ as
                        \begin{align}
                            a := p(0,1)p(1,0), \quad b := p(0,0)p(1,1),
                        \end{align}
    then $a, b$ satisfies $0 < a < 1$, $0 \le b < 1$ and $\sqrt{a} + \sqrt{b} < 1$. In addition, $p\left( 0,0 \right)$ and $p\left( 1,0 \right)$ are expressed in terms of $a, b$ as
                \begin{align}
                    p\left( 0,0 \right) &= \frac{1 - a + b + \sqrt{\left(1 - (a + b)\right)^{2} - 4 ab}}{2}, \\
                    p\left( 1,0 \right) &= \frac{1 + a - b + \sqrt{\left(1 - (a + b)\right)^{2} - 4 ab}}{2}.
                \end{align}
    Conversely, for a given $a,b$ such that $0 < a < 1$, $0 \le b < 1$ and $\sqrt{a} + \sqrt{b} < 1$, one can construct a transition matrix with condition { $0 < p(0,1) < p(1,0)$}. Actually, we can show the following. 
                \begin{lemma}\label{lem:bi_PQ}
                    We define   
                        \begin{align}
                            \mathcal{P} := \left\{ \left(p\left(i,j\right)\right)_{i,j = 0, 1} {\in} [0,1]^{4} \ ; \ \sum_{j = 0,1}  p\left(i,j\right) = 1, \ i = 0, 1, \ {0 < p\left(0,1\right) < p\left(1,0\right)} \right\},
                        \end{align}
                    and 
                        \begin{align}
                            \mathcal{P}' := \left\{ \left(a,b\right) {\in} [0,1]^2 \ ; \ a > 0, \ \sqrt{a} + \sqrt{b} < 1  \right\}.
                        \end{align}
                    In addition, we define a map $P' : \mathcal{P} \to \mathcal{P}'$ as
                        \begin{align}
                            P'\left(\left(p\left(i,j\right)\right)_{i,j = 0, 1}\right) := \left(p\left(0,1\right)p\left(1,0\right), p\left(0,0\right)p\left(1,1\right)\right).
                        \end{align}
                    Then, $P'$ is a bijection between $\mathcal{P}$ and $\mathcal{P}'$. 
                \end{lemma}
        \begin{proof}[Proof of Lemma \ref{lem:bi_PQ}]
            From the definition, it is clear that $P'$ is injective. To show that $P'$ is surjective, 
            we define a map $P : \mathcal{P}' \to \R^{4}$ as 
                \begin{align}
                    P\left(a,b\right) &= \left(p\left(a,b;i,j\right)\right)_{i,j = 0, 1} \\ &:= \begin{pmatrix}
                        \frac{1 - a + b + \sqrt{\left(1 - (a + b)\right)^{2} - 4 ab}}{2} & \frac{1 + a - b - \sqrt{\left(1 - (a + b)\right)^{2} - 4 ab}}{2} \\
                        \frac{1 + a - b + \sqrt{\left(1 - (a + b)\right)^{2} - 4 ab}}{2} & \frac{1 - a + b - \sqrt{\left(1 - (a + b)\right)^{2} - 4 ab}}{2}
                    \end{pmatrix}. \label{eq:transition}
                \end{align}
            By a direct computation, one can show that 
            $P\left(\mathcal{P}'\right) \subset \mathcal{P}$, and $P'  \circ P\left(a,b\right) = \left(a,b\right)$ for any $(a,b) \in \mathcal{P}'$. Hence $P'$ gives a bijection between $\mathcal{P}$ and $\mathcal{P}'$. 
        \end{proof}
            \begin{remark}
                If we define $F\left(a,b\right)$ as 
                    \begin{align}
                        F\left(a,b\right) :=
                            \begin{dcases}
                                \frac{1 - (a + b) - \sqrt{\left(1 - (a + b)\right)^{2} - 4 ab}}{2b} \ & \ b > 0, \\
                                a \ & \ b = 0,
                            \end{dcases}
                    \end{align}
                then $P\left(a,b\right)$ can be represented as 
                    \begin{align}
                        P\left(a,b\right) = \begin{pmatrix}
                            \frac{1}{1 + F\left( a, b \right)} & \frac{F\left( a, b \right)}{1 + F\left( a, b \right)} \\
                            \frac{a\left(1 + F\left( a, b \right)\right)}{F\left( a, b \right)} & b \left(1 + F\left( a, b \right)\right)
                        \end{pmatrix}.
                    \end{align}
                We note that $F\left(ab,b\right)$ coincides with the generating function of the Narayana numbers. 
            \end{remark}
    For any $\left(a,b\right) \in \mathcal{P}'$, we define $\left(a_{k}, b_{k}\right) \in \mathcal{P}'$, $k \in \Z_{\ge 0} $ as $a_{0} := a$, $b_{0} := b$ for $k = 0$, and 
        \begin{align}\label{eq:ab_1}
            a_{k} := \frac{a_{k-1}b_{k-1}}{\left(1 - a_{k-1}\right)^2}, \quad b_{k} := \frac{b_{k-1}}{\left(1 - a_{k-1}\right)^2},
        \end{align}
    for $k \ge 1$. We note that the fact $\left(a_{k}, b_{k}\right) \in \mathcal{P}'$ can be checked by direct computation. In addition, the sequence $\left(a_{k}\right)_{k \in \Z_{\ge 0}}$ is strictly decreasing and $\lim_{k \to \infty} a_k = 0$.

        \begin{theorem}\label{thm:skip_markov}
            Suppose that $\left(a,b\right) \in \mathcal{P}'$. 
            If $\left(\eta\left(x\right)\right)_{x \in \Z}$ is a two-sided homogeneous Markov chain conditioned on $\Omega_r$ with the transition matrix $P\left(a,b\right)$, then for any $k \in \N$,   $\left(\Psi_{k}\left(\eta\right)\left(x\right)\right)_{x \in \Z}$ is also a two-sided homogeneous Markov chain conditioned on $\Omega_r$ with the transition matrix $P\left(a_{k}, b_{k}\right)$. 
        \end{theorem}
                \begin{proof}[Proof of Theorem \ref{thm:skip_markov}]
                First we consider the case $k = 1$. 
                From Remark \ref{rem:Markov}, the distribution of $\left(\eta\left(x\right)\right)_{x \in \Z}$ can be expressed as $\hat{\mu}_{\boldsymbol{\a}}$ by setting $\boldsymbol{\a}$ as $\a_{\ell} := ab^{\ell-1}$ for any $\ell \in \N$. On the other hand, from Theorem \ref{thm:skip_stat}, we see that the distribution of $\left(\Psi_{1}\left(\eta\right)\left(x\right)\right)_{x \in \Z}$ is $\hat{\mu}_{\theta\a}$, and 
                        \begin{align}
                            \left(\theta \a\right)_{\ell} = \frac{ab^{\ell}}{\left(1 - a\right)^{2\ell}} = \frac{ab}{\left(1 - a\right)^{2}} \left(\frac{b}{\left(1 - a\right)^{2}} \right)^{\ell-1} = a_{1} b_{1}^{\ell-1},
                        \end{align}
                for any $\ell \in \N$. Hence, $\left(\Psi_{1}\left(\eta\right)\left(x\right)\right)_{x \in \Z}$ is a two-sided homogeneous Markov chain with the transition matrix $P\left(a_{1}, b_{1}\right)$. 

                By repeating the above argument, we can prove the claim of this theorem inductively for any $k \in \N$.
                \end{proof}

        {At the end of this subsection, we give an explicit formula for $\q \in \mathcal{Q}^{+}$ where $\hat{\phi}_{\q}$ is a two-sided Markov distribution conditioned on $\Omega_{r}$. This formula is a counterpart of Remark \ref{rem:Markov} in terms of the $\q$-parametrization. 
        For any $\left(p\left(i,j\right)\right)_{i,j = 0, 1} \in \mathcal{P}$, we define $c,d \in [0,1)$, $d > 0$ as
            \begin{align}
                c := \frac{ p\left(1,1\right)}{ p\left(0,0\right)}, \quad d := \frac{ p\left(0,1\right)}{ p\left(1,0\right)}.
            \end{align}
        \begin{lemma}\label{lem:markov_cd}
            Suppose that $\hat{\phi}_{\q}$ is a two-sided Markov distribution conditioned on $\Omega_{r}$ with the transition matrix $\left(p\left(i,j\right)\right)_{i,j = 0, 1} \in \mathcal{P}$. Then, for any $k \in \N$, we have 
                \begin{align}\label{eq:q_Markov}
                    q_{k} = \frac{c^{k-1}d\left(1-c\right)^2}{\left(1-c^{k}d\right)^2}.
                \end{align}
            Conversely, if $\q \in \mathcal{Q}^{+}$ is given by \eqref{eq:q_Markov} for some $c, d \in [0,1)$, $d > 0$, then, $\hat{\phi}_{\q}$ is a two-sided Markov distribution conditioned on $\Omega_{r}$. 
        \end{lemma}
        \begin{remark}
            In \cite[Proofs of Lemmas 3.6 and 3.7]{FG}, the authors conjectured a simple
        explicit formula for $\mathbf{q}$ when $\hat{\phi}_{\q}$ is a two-sided Markov distribution conditioned on $\Omega_{r}$. 
        Equation \eqref{eq:q_Markov} gives such a formula. 
        \end{remark}
        \begin{remark}
            By Lemma \ref{lem:bi_PQ}, for any $\left(p\left(i,j\right)\right)_{i,j = 0, 1}$ there exists $\left(a,b\right) \in \mathcal{P}'$  such that $\left(p\left(i,j\right)\right)_{i,j = 0, 1} $ $=$ $\left(p\left(a,b ; i,j\right)\right)_{i,j = 0, 1}$, and by direct computation, $(a,b)$ can be expressed in terms of $(c,d)$ as follows : 
            \begin{align}
                \begin{dcases}\label{eq:cd_ab}
                    a = \frac{d(1-c)^2}{(1-cd)^2}, \\
                    b = \frac{c(1-d)^2}{(1-cd)^2}.
                \end{dcases}
                \end{align}
        \end{remark}
        \begin{remark}
            We can obtain a counterpart of Theorem \ref{thm:skip_markov} in terms of $(c,d)$. If $\q$ is given by \eqref{eq:q_Markov}, then 
                \begin{align}
                    \left(\theta^{\ell}\q\right)_{k} = \frac{c_{\ell}^{k-1}d_{\ell}(1-c_{\ell})}{(1-c^{k}_{\ell} d_{\ell})^2}
                \end{align}
            where $c_{\ell}, d_{\ell} \in [0,1)$ is given by 
                \begin{align}
                    c_{\ell} = c, \quad d_{\ell} = c^{\ell - 1}d.
                \end{align}
        \end{remark}
        \begin{proof}[Proof of Lemma \ref{lem:markov_cd}]
            Assume that $\hat{\phi}_{\q}$ is a two-sided Markov distribution conditioned on $\Omega_{r}$. Then, there exists $(a,b) \in \mathcal{P}'$ such that  $\a(\q)_{k} = ab^{k-1}$ for any $k \in \N$.
            Then, by \eqref{def:atoq}, $q_k$ satisfies the following recursive equation with boundary condition. 
                \begin{align}
                    \begin{dcases}
                        q_{k+1} q_{k-1} = \left(\frac{q_k}{1-q_k} \right)^2, \quad k \in \N, \\
                        q_1 = a, \  q_2 = \frac{ab}{(1-a)^2}.
                    \end{dcases}
                \end{align}
            By \eqref{eq:cd_ab} and direct computation, we see that $(q_{k})_{k \in \N}$ given in \eqref{eq:q_Markov} solves the above recursive equation with boundary condition, and thus $q_k$ must satisfy \eqref{eq:q_Markov}. 

            Conversely, if we assume that $q_k$ is given by \eqref{eq:q_Markov}, then by \eqref{def:qtoa} and \eqref{eq:cd_ab}, $\a(\q)_{k} = ab^{k-1}$ for any $k \in \N$. Thus by Remark \ref{rem:Markov} and \eqref{eq:rel_aq_O}, $\hat{\phi}_{\q}$ is a two-sided Markov distribution conditioned on $\Omega_{r}$.
        \end{proof}
        }

\subsection{Expectation of the carrier with seat numbers}\label{subsec:expectation}

    {We first recall the relation between seats and solitons. As explained in
Remark~\ref{rem:seat_sol}, a $k$-soliton contains one \((\ell,\uparrow)\)-seat
and one \((\ell,\downarrow)\)-seat for each \(1\leq \ell\leq k\). Thus the
difference between No.~$k$ seat and No.$k+1$ seat extracts the contribution
of $k$-solitons.}

{From this point of view, the quantity
\begin{align}
    \sum_{m=0}^{n-1}
    \left(
    \mathcal{W}_{k}(T^{m}\eta,0)
    -
    \mathcal{W}_{k+1}(T^{m}\eta,0)
    \right)
\end{align}
may be regarded as the integrated current of $k$-solitons through the origin
up to time $n$. Therefore, under a stationary distribution, its expectation
gives the mean current of $k$-solitons. This is the motivation for computing
the expectations of the carrier with seat numbers in this subsection. 

In the following, we first compute the expectations of
$\mathcal{W}_{k}(\eta,0)$ under the invariant measure
$\phi_{\q}$. After that, we specialize the result to the
two-sided Markov case and give an explicit expression in terms of the corresponding
parameters. In the rest of this paper, for notational simplicity, we write $\tilde{\theta}, \tilde{C}_{k}$ as $\theta, C_k$.  }
        \begin{proposition}
            Suppose that $\q \in \mathcal{Q}^{+}$. Then, for any $k \in \N$, we have 
                \begin{align}\label{eq:car_seat_0}
                    {\bar{\mathcal{W}}_{k}(\q)} :=  \int_{\Omega_*} d{\phi_{\q}}\left(\eta\right)\mathcal{W}_{k}\left(\eta,0\right) = \frac{\bar{\e}\left({\q}\right) - \bar{\e}\left(\tilde{C}_{k-1}{\q}\right)}{2 \bar{\e}\left({\q}\right)}. 
                \end{align}
        \end{proposition}
        \begin{proof}
            To show \eqref{eq:car_seat_0}, it is sufficient to prove 
                \begin{align}\label{eq:car_seat}
                    \mathcal{W}_{k}\left(\eta, 0\right) = \Psi_{k-1}\left(\eta\right)\left(0\right).
                \end{align}
            Actually, from \eqref{eq:density_ball}, {\eqref{eq:rel_aq_O} and Corollary \ref{cor:q_skip}} with $f\left(\eta\right) = \eta\left(0\right)$, we get 
                \begin{align}
                    \int_{\Omega_*} d{\phi_{\q}}\left(\eta\right) \Psi_{k-1}\left(\eta\right)\left(0\right)
                     = \frac{\bar{\e}\left({\q}\right) - \bar{\e}\left(\tilde{C}_{k-1}{\q}\right)}{2 \bar{\e}\left({\q}\right)}. 
                \end{align}
            In the following, we will show \eqref{eq:car_seat}. First we remark that $\mathcal{W}_{1}\left(\eta, x\right) = \eta\left(x\right)$ for any $x \in \N$. This is because if the site $x$ is occupied by a ball (if $x$ is vacant), then the seat with No.$1$ is occupied (vacant). Hence, for the case $k = 1$, \eqref{eq:car_seat} is clear.
            Now we consider the case $k \ge 2$. We observe that from  \eqref{eq:car_n_seat_w}, $\mathcal{W}_{k}\left( \eta, 0 \right)$ can be represented as 
                \begin{align}
                    \mathcal{W}_{k}\left( \eta, 0 \right) &= \sum_{y = s_{\infty}\left(\eta,0\right) + 1}^{0} \left( \eta^{\uparrow}_{k}\left(y\right) - \eta^{\downarrow}_{k}\left(y\right) \right) \\
                    &= \sum_{y = s_{\infty}\left(\eta,0\right)}^{0} \left( \eta^{\uparrow}_{k}\left(y\right) - \eta^{\downarrow}_{k}\left(y\right) \right) \\
                    &= \sum_{y = 0}^{\xi_{k-1}\left( \eta, 0 \right)} \left( \eta^{\uparrow}_{{k}}\left(s_{k-1}\left(\eta, y\right)\right) - \eta^{\downarrow}_{{k}}\left(s_{k-1}\left(\eta, y\right)\right) \right) \\
                    &= \sum_{y = - \xi_{k-1}\left( \eta, 0 \right)}^{0} \left( \Psi_{k-1}\left(\eta\right)^{\uparrow}_{1}\left(y\right) - \Psi_{k-1}\left(\eta\right)^{\downarrow}_{1}\left(y\right) \right),
                \end{align}
            where at the third line we use the facts that for each $\sigma \in \{\uparrow, \downarrow\}$, 
            $\eta^{\sigma}_{k}\left( y \right) = 1$ if and only if $y = s_{k-1}\left(\eta,i\right)$ for some $i \in \Z$, and 
                \begin{align}
                    \left\{  s_{k-1}\left(\eta,i\right) \ ; \ 0 \le i \le \xi_{k-1}\left(\eta,0\right) \right\} \subset  \left[ s_{\infty}\left(\eta,0\right), 0 \right].
                \end{align}
            We then recall that from \eqref{eq:whole_p2**}, $- \xi_{k-1}\left( \eta, 0 \right) = s_{\infty}\left(\Psi_{k-1}\left(\eta\right), 0\right)$. Hence, by using \eqref{eq:car_n_seat_w} again, we get 
                \begin{align}
                    \mathcal{W}_{k}\left( \eta, 0 \right) &= \sum_{y = s_{\infty}\left(\Psi_{k-1}\left(\eta\right), 0\right)}^{0} \left( \Psi_{k-1}\left(\eta\right)^{\uparrow}_{\ell}\left(y\right) - \Psi_{k-1}\left(\eta\right)^{\downarrow}_{\ell}\left(y\right) \right) \\
                    &= \mathcal{W}_{1}\left( \Psi_{k-1}\left(\eta\right), 0 \right) \\
                    &= \Psi_{k-1}\left(\eta\right)\left(0\right). 
                \end{align}
            Therefore \eqref{eq:car_seat} is proved. 
        \end{proof}

    Thanks to {\eqref{eq:rel_aq_O}}, Proposition \ref{prop:skip_stat} and \eqref{eq:car_seat}, one can also compute the correlations between $\mathcal{W}_{k}\left(0\right)$ and $\mathcal{W}_{\ell}\left(0\right)$ with $k \neq \ell$. {These expectations give the static correlations between soliton currents.}
        \begin{proposition}\label{prop:cor_kl}
            Suppose that ${\q \in \mathcal{Q}^{+}}$. For any $k, \ell \in \N$ with $k < \ell$, we have 
                \begin{align}
                    &\int_{\Omega_*} d{\phi_{\q}}\left(\eta\right)  \mathcal{W}_{k}\left(0\right)\mathcal{W}_{\ell}\left(0\right) \\
                    &= \frac{1}{2} - \frac{\bar{\e}\left(\tilde{C}_{k-1}{\q}\right)\left( \bar{\e}\left(\tilde{\theta}^{k-1}{\q}\right) - \bar{\e}\left(\tilde{\theta}^{\ell-1}{\q}\right) + \bar{\e}\left(\tilde{C}_{\ell-k}\tilde{\theta}^{k-1}{\q}\right) + \ell - k + 1\right)}{4 \bar{\e}\left({\q}\right)} \\
                    & \quad - \frac{\left(\bar{\e}\left({\q}\right) - \bar{\e}\left(\tilde{\theta}^{k-1}{\q}\right)\bar{\e}\left(\tilde{C}_{k-1}{\q}\right)\right)\left( \bar{\e}\left(\tilde{C}_{\ell - k}\tilde{\theta}^{k-1}{\q}\right) - \bar{\e}\left(\tilde{C}_{\ell - k-1}\tilde{\theta}^{k}{\q}\right)\right)}{2\bar{\e}\left({\q}\right)\left(\bar{\e}\left(\tilde{\theta}^{k-1}{\q}\right) - \bar{\e}\left(\tilde{\theta}^{k}{\q}\right)\right)}. \label{eq:cor_seat}
                \end{align}
        \end{proposition}
        \begin{proof}[Proof of Proposition \ref{prop:cor_kl}]
            {We write $\boldsymbol{\a}(\q)$ by $\boldsymbol{\a}$ for notational simplicity.}
            By Theorem \ref{thm:skip_stat} and \eqref{eq:car_seat}, we have 
                \begin{align}
                    &\int_{\Omega_*} d\mu_{\boldsymbol{\a}}\left(\eta\right)  \mathcal{W}_{k}\left(0\right)\mathcal{W}_{\ell}\left(0\right) \\ 
                    &= \int_{\Omega_*} d\mu_{\boldsymbol{\a}}\left(\eta\right)  \Psi_{k-1}\left(\eta\right)\left(0\right)\Psi_{\ell-k}\left(\Psi_{k-1}\left(\eta\right)\right)\left(0\right) \\
                    &= \frac{\bar{\e}\left(\theta^{k-1}\boldsymbol{\a}\right)\bar{\e}\left(C_{k-1}\boldsymbol{\a}\right)}{\bar{\e}\left(\boldsymbol{\a}\right)}\int_{\Omega_{*}} d\mu_{\theta^{k-1}\boldsymbol{\a}}\left(\eta\right) \eta\left(0\right) \Psi_{\ell-k}\left(\eta\right)\left(0\right)  \\
                    & \quad + \frac{\bar{\e}\left(\theta^{k-1}\boldsymbol{\a}\right)\left(\bar{\e}\left(\boldsymbol{\a}\right) - \bar{\e}\left(\theta^{k-1}\boldsymbol{\a}\right)\bar{\e}\left(C_{k-1}\boldsymbol{\a}\right)\right)}{\bar{\e}\left(\boldsymbol{\a}\right)\left(\bar{\e}\left(\theta^{k-1}\boldsymbol{\a}\right) - \bar{\e}\left(\theta^{k}\boldsymbol{\a}\right)\right)}\int_{\Omega_{*}} d\mu_{\theta^{k-1}\boldsymbol{\a}}\left(\eta\right) \eta^{\downarrow}_{1}\left(1\right) \Psi_{\ell-k}\left(\eta\right)\left(0\right).
                \end{align}
            By Proposition \ref{prop:skip_stat}, \eqref{eq:density_ball} and \eqref{eq:number_sol}, we get 
                \begin{align}
                    &\int_{\Omega_{*}} d\mu_{\theta^{k-1}\boldsymbol{\a}}\left(\eta\right) \eta\left(0\right) \Psi_{\ell-k}\left(\eta\right)\left(0\right) \\
                    &= \frac{\bar{\e}\left(\theta^{k-1}\boldsymbol{\a}\right) + \bar{\e}\left(\theta^{\ell-1}\boldsymbol{\a}\right) - \bar{\e}\left(C_{\ell-k}\theta^{k-1}\boldsymbol{\a}\right) - \left(\ell - k + 1\right)}{4 \bar{\e}\left(\theta^{k-1}\boldsymbol{\a}\right)}.
                \end{align}
            On the other hand, by \eqref{eq:skip_stat_prop_2}, 
                \begin{align}
                    &\int_{\Omega_{*}} d\mu_{\theta^{k-1}\boldsymbol{\a}}\left(\eta\right) \eta^{\downarrow}_{1}\left(1\right) \Psi_{\ell-k}\left(\eta\right)\left(0\right) \\
                    &= \frac{\bar{\e}\left(\theta^{k-1}\boldsymbol{\a}\right) - \bar{\e}\left(\theta^{k}\boldsymbol{\a}\right) - \bar{\e}\left(C_{\ell - k}\theta^{k-1}\boldsymbol{\a}\right) + \bar{\e}\left(C_{\ell - k-1}\theta^{k}\boldsymbol{\a}\right)}{2\bar{\e}\left(\theta^{k-1}\boldsymbol{\a}\right)}.
                \end{align}
            Hence this proposition is proved.
            
        \end{proof}

    {From now on we consider the expectation under two-sided Markov distribution.  
    \begin{proposition}\label{prop:carrier_ex}
        Suppose that $\q \in \mathcal{Q}^{+}$ is given by \eqref{eq:q_Markov} with some $c, d \in [0,1)$, $d > 0$. Then, we have 
            \begin{align}
                \bar{\mathcal{W}}_{k}(\q) = \frac{1}{2} - \frac{1+c^{k}d}{2(1-c^{k}d)} + \frac{c^{k-1}d(1-c)}{(1-c^{k-1}d)(1-c^{k}d)}\left( k - 1 + \frac{(1+c)(1-d)}{(1-c)(1+d)} \right).
            \end{align}
        In particular, we obtain
            \begin{align}
                &J_{k}\left(\q\right) :=  \bar{\mathcal{W}}_{k}(\q) - \bar{\mathcal{W}}_{k+1}(\q) \\
                &= \frac{k c^{k-1}d (1-c)^2 (1+c^{k}d)}{(1-c^{k-1}d)(1-c^{k}d)(1-c^{k+1}d)} - \frac{2c^{k-1}d^2 (1-c) (1+c) (1-c^{k})}{(1-c^{k-1}d)(1-c^{k}d)(1-c^{k+1}d)(1+d)}.
            \end{align}
    \end{proposition} 
    We can also compute the explicit formula for the effective velocity of solitons. We denote by $v^{\mathrm{eff}}_{k}(\q)$ the effective velocity for $k$-solitons under $\phi_{\q}$, $\q \in \mathcal{Q}^{+}$. 
    \begin{proposition}\label{prop:eff_velo}
        Suppose that $\q \in \mathcal{Q}^{+}$ is given by \eqref{eq:q_Markov} with some $c, d \in [0,1)$, $d > 0$. Then, for any $k \in \N$ and $0\le \ell \le k$, we have 
            \begin{align}\label{eq:eff_velo}
                v^{\mathrm{eff}}_{k-\ell}\left(\theta^{\ell}\mathbf{q}\right) = \left(k-\ell\right) \frac{1 + c^{\ell}d}{1 - c^{\ell}d} - \frac{2c^{\ell}d \left(1+c\right) \left(1-c^{k-\ell}\right)}{\left(1-c\right)\left(1-c^{\ell}d\right)\left(1+c^{k}d\right)},
            \end{align}
        with convention $v^{\mathrm{eff}}_{0}\left(\theta^{k}\mathbf{q}\right) := 0$.
    \end{proposition}
    To show Propositions \eqref{prop:carrier_ex} and \ref{prop:eff_velo} we prepare the following lemma. 
        \begin{lemma}\label{lem:explicit_e}
            Suppose that $\q \in \mathcal{Q}^{+}$ is given by \eqref{eq:q_Markov} with some $c, d \in [0,1)$, $d > 0$. Then, for any $k \in \N \cup\{\infty\}$ and $0 \le \ell \le k$, we have 
                \begin{align}\label{eq:l_k_e}
                    \bar{\e}\left(\tilde{\theta}^{\ell} \tilde{C}_{k} \q\right)
=
\frac{1+c^{k+1}d}{1-c^{k+1}d}
\frac{1+c^{\ell}d}{1-c^{\ell}d}
-
\frac{2 c^{k} d \left(1 - c\right)}
{\left(1 - c^{k}d\right)\left(1 - c^{k+1}d\right)}
\left(
\left(k-\ell\right)\frac{1+c^{\ell}d}{1-c^{\ell}d}
+
\frac{1 + c}{1-c}
\right).
                \end{align}
        \end{lemma}
    Proposition \ref{prop:carrier_ex} is a direct result from \eqref{eq:car_seat_0} and \eqref{eq:l_k_e}, so we omit the proof. We can also find the explicit representation of the right-hand side of \eqref{eq:cor_seat} in terms of $(c,d)$. We omit the fully expanded formula, since it is lengthy.
    \begin{remark}\label{rem:phys}
        The explicit formulae obtained in this subsection for the
two-sided Markov case are consistent with those previously derived in the
physics literature \cite{KMP}, after translating the notation.
        Indeed, since the expected number of $k$-solitons per excursion is 
            \begin{align}
                \bar{\e}\left(\tilde{\theta}^{k} \q\right) \times \frac{q_k}{1-q_k}, 
            \end{align}
        by \eqref{eq:q_Markov} and Lemma \ref{lem:explicit_e}, the density of $k$-soliton is equal to 
            \begin{align}
                \frac{\bar{\e}\left(\tilde{\theta}^{k} \q\right)}{\bar{\e}\left( \q\right)} \times \frac{q_k}{1-q_k} = \frac{c^{k-1}d(1-d)(1 + c^{k}d)(1-c)^2}{(1+d)(1-c^{k-1}d)(1-c^{k}d)(1-c^{k+1}d)} \quad ( = \text{  \cite[(3.26)]{KMP}}).
            \end{align}
        We note that for the Bernoulli product case, the density of $k$-soliton has been rigorously derived in \cite[(10)]{KL}.

        Also, by \eqref{eq:q_Markov}, the hole density of $k$-solitons is 
            \begin{align}
                \frac{\bar{\e}\left(\tilde{\theta}^{k} \q\right)}{\bar{\e}\left(\q\right)} = \frac{(1-d)(1 + c^{k}d)}{(1+d)(1-c^{k}d)} \quad ( = \text{  \cite[(3.25)]{KMP}}).
            \end{align}
            
        The formula for the effective
velocity also agrees with \cite[(D.2)]{KMP}. For the stationary $k$-soliton current, one can check the consistency by the following relation, 
            \begin{align}
                J_{k}\left(\q\right) = \left( \text{density of $k$-solitons} \right) \times \left( \text{effective velocity of $k$-solitons} \right). 
            \end{align}

    \end{remark}}

    {
    \begin{remark}\label{rem:offset_vel}
         The effective velocity under $\phi_{\q}$ can be obtained from the expectation of the offset, appearing in Theorem \ref{thm:linear_whole}. Indeed,
        in the
$\zeta$-coordinate, the mean velocity of $k$-solitons is given by 
    \begin{align}
        k - 2\bar{W}_{k}(\q), \quad \bar{W}_{k}(\q) := \int_{\Omega_*} d\phi_{\q}(\eta) W_{k}\left(\eta,0\right).
    \end{align}
    After taking into account the hole density of $k$-solitons, the effective velocity of $k$-solitons is expected to be 
            \begin{align}
                \frac{\bar{\e}\left(\q\right)}{\bar{\e}\left(\tilde{\theta}^{k} \q\right)} \times \left( k - 2\bar{W}_{k}(\q) \right).
            \end{align}
    This heuristic derivation will be justified rigorously in Appendix \ref{app:eff_velo}.
    
    \end{remark}}

{\begin{proof}[Proof of Lemma \ref{lem:explicit_e}]

    First we consider the case $k < \infty$. For a fixed $k \in \N \cup \{\infty\}$, we write $w_{\ell} : =  \bar{\e}\left(\tilde{\theta}^{\ell} \tilde{C}_{k} \q\right)$, $0 \le \ell \le k$. Then, by \eqref{eq:system_Ce}, $(w_{\ell})_{0 \le \ell \le k}$ satisfies 
        \begin{align}
            \begin{dcases}
                w_{\ell - 1} + w_{\ell + 1} - \frac{2 w_{\ell}}{1 - q_{\ell}} = 0, \quad  1 \le \ell \le k -1, \\
                w_{k} = 1, \quad w_{k-1} = \frac{1+q_k}{1-q_k}. 
            \end{dcases}
        \end{align}
    Two independent solutions of the recursive equation can be taken as 
        \begin{align}\label{eq:two_indep}
            \frac{1 + c^{\ell}d}{1 - c^{\ell}d}, \quad  \frac{\ell\left(1 + c^{\ell}d\right)}{1 - c^{\ell}d} - \frac{1 + c}{1-c}.
        \end{align}
    Then, since the right-hand side of \eqref{eq:l_k_e} can be represented as 
        \begin{align}
            \left(\frac{1+c^{k+1}d}{1-c^{k+1}d} - \frac{2 k c^{k} d \left(1 - c\right)}
{\left(1 - c^{k}d\right)\left(1 - c^{k+1}d\right)}\right) \frac{1 + c^{\ell}d}{1 - c^{\ell}d} + \frac{2 k c^{k} d \left(1 - c\right)}
{\left(1 - c^{k}d\right)\left(1 - c^{k+1}d\right)} \left(\frac{\ell\left(1 + c^{\ell}d\right)}{1 - c^{\ell}d} - \frac{1 + c}{1-c}\right),
        \end{align}
    and satisfies the boundary condition for $\ell = k, k - 1$, the claim of this lemma is proved for the case $k < \infty$. 

    Next we consider the case $k = \infty$. We observe that for any excursion $\e \in \mathcal{E}$ and $\ell \in \N$,
        \begin{align}
            |\Psi_{\ell-1}\left(\e\right)| = 1 + 2 \sum_{h = 1}^{\infty} \sum_{i = 0}^{|\Psi_{\ell + h - 1}\left(\e\right)| - 1} h \zeta_{\ell + h - 1}(\iota(\e),i),
        \end{align}
    and thus from Remark \ref{rem:indep}, by taking the expectation of both sides under $\phi_{\q}$, we see that   $(w_{\ell})_{\ell \in \Z_{\ge 0}}$ satisfies 
        \begin{align}
            \begin{dcases}
                w_{\ell - 1} + w_{\ell + 1} - \frac{2 w_{\ell}}{1 - q_{\ell}} = 0, \quad  1 \le \ell \le k -1, \\
                \lim_{\ell \to \infty} w_{\ell} = 1.
            \end{dcases}
        \end{align}
    Since two independent solutions of the recursive equation are \eqref{eq:two_indep}, we have 
        \begin{align}
            w_{\ell} = \frac{1 + c^{\ell}d}{1 - c^{\ell}d},
        \end{align}
    and thus the claim of this lemma is shown for the case $k = \infty$. 
    \end{proof}}

{Finally, we show Proposition \ref{prop:eff_velo}.

\begin{proof}[Proof of Proposition \ref{prop:eff_velo}]

    We fix $k \in \N$. We define $v_{\ell} := v^{\mathrm{eff}}_{k-\ell}\left(\theta^{\ell}\mathbf{q}\right)$, $0 \le \ell \le k - 1$ and set $v_{k} := 0$. 
    By \cite[(4.12),(4.13)]{OSS}, $(v_{\ell})_{0 \le \ell \le k}$ satisfies 
        \begin{align}\label{eq:rec_eff}
            \begin{dcases}
                v_{\ell - 1} + v_{\ell + 1} - \frac{2 v_{\ell}}{1 - q_{\ell}} = 0, \quad  1 \le \ell \le k -1, \\
                v_{k} = 0, \quad v_{k-1} = \frac{1}{\bar{\e}\left(\tilde{\theta}^{k} \q\right)}. 
            \end{dcases}
        \end{align}
    The above recursive equation is the same as that in the proof of Lemma \ref{lem:explicit_e}, but the boundary condition is different. By using \eqref{eq:two_indep}, the right-hand side of \eqref{eq:eff_velo} can be represented as 
        \begin{align}
            \left(k - \frac{1+c}{1-c} \frac{1-c^{k}d}{1+c^{k}d}\right)\frac{1 + c^{\ell}d}{1 - c^{\ell}d} - \left(\frac{\ell\left(1 + c^{\ell}d\right)}{1 - c^{\ell}d} - \frac{1 + c}{1-c}\right),
        \end{align}
    and satisfies the boundary condition for $\ell = k, k - 1$. Therefore we have \eqref{eq:eff_velo}. 
    
\end{proof}}

\section*{Acknowledgment}
    The author thanks Makiko Sasada for insightful discussion and for pointing out some gaps in Section \ref{subsec:Markov}.
    The work of HS has been supported by JST CREST Grant No. JPMJCR1913 and JSPS KAKENHI Grant No. JP21K20332, 24KJ1037, 24K16936 and 24K00528.

\appendix

\section{Takahashi--Satsuma algorithm}\label{app:TS}

    In this section, we briefly recall how a ball configuration $\eta \in \Omega$ is decomposed into $k$-solitons via the Takahashi--Satsuma algorithm \cite{TS}. First we recall that for any $i \in \Z$,  $s_{\infty}\left(i\right) = s_{\infty}\left(\eta, i\right)$ is defined in \eqref{def:sk} for the half-line case, and \eqref{def:s_k_1}, \eqref{def:s_k_3} for the whole-line case. To define solitons in $\eta$, we assume that $\left| s_{\infty}\left(i\right) \right| < \infty$
    for any $i \in \Z$. Under this assumption, for any $i \in \Z$, we define the $i$-th excursion $\e_i$ as 
        \begin{align}
            \e_i := \left( \eta\left(x\right) \ ; \ s_{\infty}\left(i\right) \le x \le s_{\infty}\left(i+1\right) - 1  \right).
        \end{align}
    This is an abuse of notation, but we write $\e_i \setminus \{ s_{\infty}\left(i\right) \}$ the sequence of $1$s and $0$s obtained by omitting the leftmost $0$ from $\e_i$, that is, $\e_i \setminus \{ s_{\infty}\left(i\right) \} := \left( \eta(x) \ ; \  s_{\infty}\left(i\right) < x < s_{\infty}(i+1) \right)$. 

    The identification of solitons in the whole configuration $\eta$ is done through the identification of solitons for each excursion.  
    For each $\e_i$, we can find solitons via the Takahashi--Satsuma algorithm as follows : 
            \begin{itemize}
                \item Select the leftmost run of consecutive $0$s or $1$s in $\e_i \setminus \{ s_{\infty}\left(i\right) \}$ such that the length of the subsequent run is at least as long as the length of it.
                \item Let $k$ be the length of the selected run. Group the $k$ element of the selected run and the first $k$ elements of the subsequent run. The grouped $2k$ elements are identified as a soliton with {size} $k$, or $k$-soliton.
                \item Remove the identified $k$-soliton, and repeat the above procedure until all $1$s are removed. 
            \end{itemize}
        Here, a $k$-soliton is defined as a subset of $\Z$, and its cardinality is $2k$. 
        From the definition of records, if $\e_i \setminus \{ s_{\infty}\left(i\right) \}$ is not empty, then all $1$s and $0$s in $\e_i \setminus \{ s_{\infty}\left(i\right) \}$ are grouped and become components of solitons, and for any $i \in \Z$, the site $s_{\infty}\left(i\right)$ is not a component of any soliton. 

        An example of applying the above algorithm to an excursion $011001110101100010$ is shown in Figure \ref{ex:TSforex}. By repeating the algorithm five times, we see that there are one $4$-soliton, one $2$-soliton and three $1$-solitons.

        \begin{figure}[H]
        \footnotesize
                \setlength{\tabcolsep}{4pt}
                \renewcommand{\arraystretch}{2}
                    \centering
                    \begin{tabular}{rccccccccccccccccccc}
                    $\e$ :  & 0 & 1 & 1 & 0 & 0 & 1 & 1 & 1 & 0 & 1 & 0 & 1 & 1 & 0 & 0 & 0 & 1 & 0 & 0 \\
                    & 0 & \xcancel{1} & \xcancel{1} & \xcancel{0} & \xcancel{0} & 1 & 1 & 1 & 0 & 1 & 0 & 1 & 1 & 0 & 0 & 0 & 1 & 0 & 0 \\
                    & 0 &  &  &  &  & 1 & 1 & 1 & \xcancel{0} & \xcancel{1} & 0 & 1 & 1 & 0 & 0 & 0 & 1 & 0 & 0 \\
                    & 0 &  &  &  &  & 1 & 1 & 1 &  &  & \xcancel{0} & \xcancel{1} & 1 & 0 & 0 & 0 & 1 & 0 & 0 \\
                    & 0 &  &  &  &  & 1 & 1 & 1 &  &  & & & 1 & 0 & 0 & 0 & \xcancel{1} & \xcancel{0} & 0 \\ 
                    & 0 &  &  &  &  & \xcancel{1} & \xcancel{1} & \xcancel{1} &  &  & & & \xcancel{1} & \xcancel{0} & \xcancel{0} & \xcancel{0} &  &  & \xcancel{0}
                   \end{tabular}
                    \caption{Identifying solitons in $\e$ by the TS Algorithm}\label{ex:TSforex}
                \end{figure}

\section{A formula of the effective velocity }\label{app:eff_velo}

    {We fix $\q \in \mathcal{Q}^{+}$. The purpose of this section is to show the following equation explained in Remark \ref{rem:offset_vel} : 
        \begin{align}
            v^{\mathrm{eff}}_{k}(\q) = \frac{\bar{\e}\left(\q\right)}{\bar{\e}\left(\tilde{\theta}^{k} \q\right)} \times \left( k - 2\bar{W}_{k}(\q) \right).
        \end{align}
     We observe that by \eqref{eq:car_seat_0}, the right-hand side is equal to
        \begin{align}
            \frac{1}{\bar{\e}\left(\tilde{\theta}^{k} \q\right)} \sum_{h = 1}^{k} \bar{\e}\left(\tilde{C}_{h-1} \q\right). 
        \end{align}
    We fix $k \in \N$, and we set $v'_{k} := 0$, and
        \begin{align}
            v'_{\ell} := \frac{1}{\bar{\e}\left(\tilde{\theta}^{k} \q\right)} \sum_{h = 1}^{k-\ell} \bar{\e}\left(\tilde{C}_{h-1} \tilde{\theta}^{\ell} \q\right),
        \end{align}
    for $0 \le \ell \le k -1$.
    By using the relation $\tilde{C}_{h-1} \tilde{\theta}^{\ell} \q = \tilde{\theta}^{\ell} \tilde{C}_{\ell + h-1} \q$ and $\bar{\e}\left(\tilde{\theta}^{\ell-1} \tilde{C}_{\ell-1} \q\right) = \bar{\e}\left(\tilde{\theta}^{\ell+1} \tilde{C}_{\ell} \q\right) = 1$, we get 
        \begin{align}
            &v'_{\ell - 1} + v'_{\ell+1} - 2v'_{\ell} \\
            &= \frac{1}{\bar{\e}\left(\tilde{\theta}^{k} \q\right)} \sum_{h = 1}^{k-\ell} \left( \bar{\e}\left(\tilde{\theta}^{\ell-1} \tilde{C}_{\ell + h-1} \q\right) + \bar{\e}\left(\tilde{\theta}^{\ell+1} \tilde{C}_{\ell + h-1} \q\right) - 2 \bar{\e}\left(\tilde{\theta}^{\ell} \tilde{C}_{\ell + h-1} \q\right) \right) \\
            & \quad + \frac{1}{\bar{\e}\left(\tilde{\theta}^{k} \q\right)} \bar{\e}\left(\tilde{\theta}^{\ell-1} \tilde{C}_{\ell} \q\right) - \frac{1}{\bar{\e}\left(\tilde{\theta}^{k} \q\right)} \bar{\e}\left(\tilde{\theta}^{\ell+1} \tilde{C}_{\ell} \q\right) \\
            &= \frac{2q_{\ell}}{1-q_\ell} \times \frac{1}{\bar{\e}\left(\tilde{\theta}^{k} \q\right)} \sum_{h = 1}^{k-\ell}  \bar{\e}\left(\tilde{\theta}^{\ell} \tilde{C}_{\ell + h-1} \q\right) \\
            &= \frac{2q_{\ell}}{1-q_\ell} v'_{\ell}, 
        \end{align}
    where at the second equality we use \eqref{eq:system_Ce}. Thus, $(v'_{\ell})_{0\le \ell \le k}$ satisfies
        \begin{align}
            \begin{dcases}
                v'_{\ell - 1} + v'_{\ell + 1} - \frac{2 v'_{\ell}}{1 - q_{\ell}} = 0, \quad  1 \le \ell \le k -1, \\
                v'_{k} = 0, \quad v'_{k-1} = \frac{1}{\bar{\e}\left(\tilde{\theta}^{k} \q\right)},
            \end{dcases}
        \end{align}
    and the above recursive equation with boundary condition coincides with \eqref{eq:rec_eff}. Therefore we have 
        \begin{align}
            v'_{0} = \frac{1}{\bar{\e}\left(\tilde{\theta}^{k} \q\right)} \sum_{h = 1}^{k} \bar{\e}\left(\tilde{C}_{h-1} \q\right) = v^{\mathrm{eff}}_{k}(\q).
        \end{align}}


\begin{thebibliography}{99}

        \bibitem[CKST]{CKST} {\sc D. A. \ Croydon, T. \ Kato, M. \ Sasada and S. \ Tsujimoto} : {\em Dynamics of the box-ball system with random initial conditions via Pitman’s transformation}. Mem. Amer. Math. Soc. 283 (2023), no. 1398
        
		\bibitem[CS]{CS}  {\sc D. A. \ Croydon and M. \ Sasada} : {\em Generalized Hydrodynamic Limit for the Box-Ball System}. Commun. Math. Phys. 383, 427-463 (2021)
        
		\bibitem[CS2]{CS2} {\sc D. A. \ Croydon and M. \ Sasada} : {\em Invariant measures for the box-ball system based on stationary Markov chains and periodic Gibbs measures}. J. Math. Phys. 60, 083301 (2019)
        
        
		\bibitem[FG]{FG} {\sc P. A. \  Ferrari and D. \ Gabrielli} : {\em BBS invariant measures with independent soliton components}. Electron. J. Probab. 25: 1-26 (2020)
        
		\bibitem[FNRW]{FNRW} {\sc P. A. \  Ferrari, C. \ Nguyen, L. \ Rolla, and M. \ Wang} : {\em Soliton decomposition of the box-ball system}. Forum of Mathematics, Sigma 9 (2021)
        
		\bibitem[FOY]{FOY} {\sc K.\ Fukuda, Y. \ Yamada and M. \ Okado} : {\em Energy functions in box ball systems}. Int. J. Mod. Phys. A 15(09), 1379 (2000)
        
		
		\bibitem[IKT]{IKT} {\sc R. \ Inoue, A. \ Kuniba and T. \ Takagi} : {\em Integrable structure of box–ball systems: crystal, Bethe ansatz, ultradiscretization and tropical geometry}. J. Phys. A: Math. Theor. 45(7), 073001 (2012)
		\bibitem[KL]{KL} {\sc A. \ Kuniba and H. \ Lyu} :  {\em Large Deviations and One-Sided Scaling Limit of Randomized Multicolor Box-Ball System.} J. Stat Phys 178, 38--74 (2020)
        \bibitem[KLO]{KLO} {\sc A. \ Kuniba, H. \ Lyu and M. \ Okado} : {\em Randomized box-ball systems, limit shape of rigged configurations and thermodynamic Bethe ansatz}. Nuclear Physics B 937 240--271 (2018)
        \bibitem[KMP]{KMP} {\sc A. \ Kuniba,  G. \ Misguich and V. \ Pasquier} : {\em Generalized hydrodynamics in box-ball system}. J. Phys. A: Math. Theor. 53 404001 (2020)
        \bibitem[KMP2]{KMP2} {\sc A. \ Kuniba,  G. \ Misguich and V. \ Pasquier} : {\em Generalized hydrodynamics in complete box-ball system for $U_{q}(\hat{sl}_n)$}. SciPost Phys. 10, 095 (2021)
        \bibitem[KMP3]{KMP3} {\sc A. \ Kuniba,  G. \ Misguich and V. \ Pasquier} : {\em Current correlations, Drude weights and large deviations in a box-ball system}. J. Phys. A: Math. Theor. \textbf{55} 244006 (2022)
		\bibitem[KNTW]{KNTW} {\sc S. \  Kakei, J.  J.  C. \ Nimmo, S. \ Tsujimoto and R. \ Willox} : {\em Linearization of the box-ball system: an elementary approach}. Journal of Integrable Systems. Volume 3, Issue 1, 2018, xyy002
  
  
		
		\bibitem[KOSTY]{KOSTY} {\sc A. \ Kuniba, M. \ Okado, R. \ Sakamoto, T. \ Takagi and Y. \ Yamada } : {\em Crystal interpretation of Kerov-Kirillov-Reshetikhin bijection}. Nuclear Physics B, \textbf{740}, 299--327 (2006) 
		
		\bibitem[KOY]{KOY} {\sc A. Kuniba, M. Okado and Y. Yamada} : {\em Box–ball system with reflecting end}. J. Nonlin. Math. Phys. {\bf 12} 475–507 (2005)

        \bibitem[KS]{KS} {\sc A.N. \ Kirillov, and R. \ Sakamoto} : {\em Relationships Between Two Approaches: Rigged Configurations and 10-Eliminations}. Lett Math Phys 89, 51-65 (2009)
		
		\bibitem[KTT]{KTT} {\sc A. \ Kuniba, T. \ Takagi and A. \ Takenouchi} : {\em Bethe ansatz and inverse scattering transform in a periodic box-ball system}. Nucl. Phys. B 747, 354--397 (2006)
		
		\bibitem[LLP]{LLP} {\sc L. \ Levine,  H. \ Lyu and J. \ Pike} : {\em Double Jump Phase Transition in a Soliton Cellular Automaton}. Int Math Res Notices volume 2022, Issue 1, 665--727 (2020)
        \bibitem[LLPS]{LLPS} {\sc J. \ Lewis, H. \ Lyu, P. Pylyavskyy, and A. \ Sen} : {\em Scaling limit of soliton lengths in a multicolor box-ball system}. Forum of Mathematics, Sigma. 2024;12:e120.
		
		\bibitem[MIT]{MIT} {\sc J. \ Mada, M. \ Idzumi and T. \ Tokihiro} : {\em  On the initial value problem of a periodic box-ball system}. J. Phys. A : Math. Gen. 39 (2006)
        
		\bibitem[MSSS]{MSSS} {\sc M. \ Mucciconi, M. \ Sasada, T. \ Sasamoto and H. \ Suda} : {\em Relationships between two linearizations of the box-ball system : Kerov-Kirillov-Reshetikhin bijection and slot configuration}. Forum of Mathematics, Sigma. 2024;12:e55. 

        \bibitem[OSS]{OSS} {\sc S. \ Olla, M. \ Sasada and H. \ Suda} : {\em Scaling limits of solitons in the box-ball system}. Duke Math. J., to appear. arXiv:2411.14818
		
		\bibitem[T]{T} {\sc T. \ Takagi} : {\em Inverse scattering method for a soliton cellular automaton}. Nuclear Physics \textbf{B707}, 577–601 (2005).
		
		
		\bibitem[TM]{TM} {\sc D. \ Takahashi and J. \ Matsukidaira} : {\em Box and ball system with a carrier and ultra-discrete modified KdV equation}. J. Phys. A \textbf{30} L733–L739 (1997).
		
		
		\bibitem[TS]{TS} {\sc D. \ Takahashi and J. \ Satsuma} : {\em A soliton cellular automaton}. J. Phys. Soc. Japan \textbf{59} 3514--3519 (1990)
		
		

		
\end{thebibliography}
\end{document}